\documentclass{daj}

\usepackage{amsmath, amsthm, amssymb, amsfonts}
\usepackage{comment}
\usepackage{enumitem}
\usepackage[utf8]{inputenc}
\usepackage{mathtools}
\usepackage{a4wide}
\usepackage{tikz}

\numberwithin{equation}{section}

\newcommand{\act}{\curvearrowright}
\newcommand{\lact}{\curvearrowleft}
\newcommand{\id}{\mathord{\operatorname{id}}}

\newcommand{\Z}{\mathbb{Z}}

\newcommand{\N}{\mathbb{N}}

\newcommand{\la}{\left\langle}
\newcommand{\ra}{\right\rangle}

\newcommand{\A}{\mathbf{A}}
\newcommand{\PA}{\mathbf{PA}}
\newcommand{\TA}{\mathbf{TA}}

\newcommand{\HFA}{\mathbf{HFA}}
\newcommand{\HTA}{\mathbf{HTA}}

\newcommand{\I}{\mathcal I}

\newcommand{\NN}[1]{N_{C_{#1}}}

\newcommand{\rE}{\operatorname{E}}

\newcommand{\Aut}{\operatorname{Aut}}

\newcommand{\pat}{\mathsf{path}}

\DeclareMathOperator{\supp}{\mathrm{supp}}

\newcommand{\Gr}{\mathcal{G}}

\newcommand{\EG}{\rE(\Gr)}

\newcommand{\Pres}[2]{\left\langle #1 \, \vert \, #2 \right\rangle}

\newcommand{\HNN}{\operatorname{HNN}}

\DeclareMathOperator{\diam}{\mathrm{diam}}
\DeclareMathOperator{\st}{\mathrm{st}}

\DeclareMathOperator{\dom}{\mathrm{dom}}
\DeclareMathOperator{\rng}{\mathrm{rng}}

\newcommand{\inv}{^{-1}}

\newcommand{\sym}{\operatorname{Sym}}

\newcommand{\Bc}{\mathcal B}
\newcommand{\Cc}{\mathcal C}

\newcommand{\Hc}{\mathcal H}

\newcommand{\Rc}{\mathcal R}
\newcommand{\Tc}{\mathcal T}
\newcommand{\Uc}{\mathcal U}

\theoremstyle{plain}
\newtheorem{theorem}{Theorem}[section]
\newtheorem{thmi}{Theorem}

\newtheorem{cori}[thmi]{Corollary}

\newtheorem*{claim}{Claim}
\newtheorem*{theoremSansNum}{Theorem}

\newtheorem{lemma}[theorem]{Lemma}
\newtheorem{proposition}[theorem]{Proposition}
\newtheorem{corollary}[theorem]{Corollary}
\theoremstyle{definition}

\newtheorem{definition}[theorem]{Definition}

\newtheorem{example}[theorem]{Example}

\newtheorem{remark}[theorem]{Remark}

\newenvironment{cproof}{\begin{proof}[Proof of the claim]}{\end{proof}}

\dajAUTHORdetails{%
  title = {A characterization of high transitivity for groups acting on trees}, 
  author = {Pierre Fima, François Le Maître, Soyoung Moon, and Yves Stalder},
  plaintextauthor = {Pierre Fima, François Le Maître, Soyoung Moon, and Yves Stalder},
  plaintexttitle = {A characterization of high transitivity for groups acting on trees}, 
  keywords = {high transitivity, groups acting on trees, HNN extensions, amalgamated products, Baumslag-Solitar groups, Baire category theorem.},
}   

\dajEDITORdetails{%
   year={2022},
   number={8},
   received={24 August 2020},   
   revised={19 April 2021},    
   published={31 August 2022},  
   doi={10.19086/da.37645},       
}   

\begin{document}

\begin{frontmatter}[classification=text]

\title{A characterization of high transitivity for groups acting on trees} 

\author[fima]{Pierre Fima\thanks{Partially supported by ANR project ANCG (No. 
		ANR-19-CE40-0002), ANR project AODynG (No. ANR-19-CE40-0008), and Indo-French 
		Centre for the Promotion of Advanced Research - CEFIPRA.}}
\author[lemaitre]{François Le Maître\thanks{Partially supported by ANR project AODynG (No. 
		ANR-19-CE40-0008), ANR Project AGRUME (No. ANR-17-CE40-0026) and 
		Indo-French Centre for the Promotion of Advanced Research - CEFIPRA}}
\author[moon]{Soyoung Moon}
\author[stalder]{Yves Stalder}

\begin{abstract}
We establish a sharp sufficient condition for groups acting on trees to be highly transitive
when the action on the tree is minimal of general type. This gives new examples of highly transitive groups, including icc non-solvable Baumslag-Solitar groups, thus answering a question of Hull and Osin.
\end{abstract}
\end{frontmatter}


\tableofcontents
\section{Introduction}

Given a countably infinite group $\Gamma$, one is naturally led to the 
study of its transitive actions, or equivalently of the homogeneous spaces 
$\Gamma/\Lambda$ where $\Lambda$ is a subgroup of $\Gamma$. A basic 
invariant for such an action is the \textbf{transitivity degree}, namely 
the supremum of the $n\in\N$ such that for any two $n$-tuples of distinct 
points, the first can be taken to the second by a group element. Note that 
the transitivity degree of an action can be infinite, as is witnessed by 
the natural action of the group of finitely supported permutations of a 
countably infinite set. One can then lift the transitivity degree to a 
group invariant $\mathrm{td}(\Gamma)$ defined as the supremum of the 
transitivity degrees of the \emph{faithful} $\Gamma$-actions. The most 
transitive groups are the \textbf{highly transitive} groups, namely those 
which admit a faithful action whose transitivity degree is infinite. Note 
that such groups automatically have infinite transitivity degree. As noted 
by Hull and Osin in
\cite{hullTransitivitydegreescountable2016},
it is actually unknown whether there is a countable group 
$\Gamma$ with infinite transitivity degree, but which fails to be highly 
transitive. 

\subsection{Some highly transitive groups}

Let us now give a brief overview of groups which are known to be highly transitive.
First, the group of finitely supported permutations of a countably infinite set is highly transitive. 
Other examples of locally finite highly transitive groups are provided by the forward orbit stabilizers of minimal $\Z$-actions on the Cantor space, such as the group of dyadic permutations, and by the Hall group.

For finitely generated amenable groups, one can upgrade the group $S_f(\Z)$ of finitely supported permutations of $\Z$ to the $2$-generated group $S_f(\Z)\rtimes\Z$ of permutations which are translations except on a finite set. 
Other natural examples are provided by derived groups of topological full groups of minimal $\Z$-subshifts acting on an orbit (the fact that they are finitely generated is due to Matui \cite{matuiremarkstopologicalfull2006}, while their amenability is a celebrated result of Juschenko and Monod
\cite{juschenkoCantorsystemspiecewise2013}).

In the non-amenable realm, the first explicit examples of highly transitive groups are free groups $\mathbb{F}_n$ for $2\leq n\leq +\infty$, as was shown in 1976 by McDonough \cite{zbMATH03558079} (see also the work of Dixon in \cite{zbMATH04105195}). 
The case of a general free product has been studied Glass and McCleary in \cite{zbMATH04181610} and later settled by Gunhouse \cite{zbMATH04193986} and independently by Hickin  \cite{zbMATH00120147}. 

In the last few years, many new examples of highly transitive groups have been discovered such as 
surface groups \cite{zbMATH06204044}, Out$(\mathbb{F}_n)$ for $n\geq 4$ \cite{garion_highly_2013}, and  non-elementary hyperbolic groups with trivial finite radical \cite{zbMATH06120601}. 
A vast generalization of these results was then found by Hull and Osin.
\begin{theoremSansNum}[{\cite[Theorem~1.2]{hullTransitivitydegreescountable2016}}]
	Every countable acylindrically hyperbolic group admits a highly transitive action with finite kernel. 
	In particular, every countable acylindrically hyperbolic group with trivial finite radical is highly transitive.
\end{theoremSansNum}
Let us recall that a group is called \textbf{acylindrically hyperbolic} if
it admits a non-elementary acylindrical action on a hyperbolic space. 
For equivalent definitions, and for more background on acylindrically hyperbolic groups, we refer the reader to \cite{osin_acylindrically_2016} or \cite{ osin_groups_2019}.

On the other hand, examples which are not entirely covered by Hull and Osin's result come from groups acting on trees as in the work of
the first, third and fourth authors
\cite{fimaHighlyTransitiveActions2015}.
Other examples are provided by a recent result of Gelander, Glasner and So\u{\i}fer, which states that any center free unbounded and non-virtually solvable countable subgroups of ${\rm SL}_2(k)$ is highly transitive, where $k$ is a local field \cite{gelander_maximal_2020}. 

Our main result is an optimal generalization of the aforementioned result 
of the first, third and fourth authors.

\begin{thmi}\label{ThmMain}
	Let $\Gamma\curvearrowright\mathcal{T}$ be a minimal action of general type of a countable group $\Gamma$ on a tree $\mathcal{T}$. If the action on the boundary $\Gamma\curvearrowright \partial\mathcal{T}$ is topologically free, then $\Gamma$ admits a highly transitive and highly faithful action; in particular, $\Gamma$ is highly transitive.
\end{thmi}	

The above \textbf{minimality} assumption means that there are no nontrivial invariant subtrees, 
while the \textbf{topological freeness} assumption means that no half-tree 
can be pointwise fixed by a non-trivial group element (in particular, the 
action is faithful). 
An action on a tree is of \textbf{general type} when there are two transverse hyperbolic elements (see Section \ref{PrelimTrees}). 
All these hypotheses are necessary in Theorem \ref{ThmMain}: 
for topological freeness this is discussed in the next section, while for 
the type of the action and the minimality this is discussed in section  
\ref{MinimalityIsNecessary}.

Finally, \textbf{high faithfulness} is a natural strengthening of 
faithfulness introduced in 
\cite{fimaHighlyTransitiveActions2015}, which states that 
the intersection of the supports of finitely many nontrivial group elements 
is always infinite (see Section \ref{Group actions} for equivalent 
definitions). 
Let us remark that the group of finitely 
supported permutations does not admit highly transitive highly faithful 
actions \cite[Remark 8.23]{fimaHomogeneousActionsUrysohn2018}, and that the 
natural 
highly transitive action of a topological full group is never highly 
faithful. It would be interesting to understand whether the highly 
transitive actions of acylindrically hyperbolic groups with trivial finite 
radical built by Hull and Osin are highly faithful.

\subsection{Obstructions to high transitivity}

Let us now move on to obstructions to high transitivity, which will lead us to a reformulation of our main theorem as a series of equivalences 
thanks to the work of Hull and Osin 
\cite{hullTransitivitydegreescountable2016}
and of Le Boudec and Matte Bon	
\cite{leboudecTripleTransitivityNonfree2019}. 

First, one can use the fact that the group of permutations of a countably 
infinite set is topologically simple for the product of the discrete 
topology, 
and that high transitivity can be reformulated as arising as a dense 
subgroup of this group. This yields the well-known fact that in a highly 
transitive group, the 
centralizer of every non-trivial group element is core-free (see Corollary~
\ref{cor: ht restriction}).
In particular, highly transitive groups cannot be solvable or contain 
nontrivial commuting normal subgroups, and they must be icc (all their 
non-trivial conjugacy classes are infinite).

In another direction, Hull and Osin have shown that given a highly transitive faithful action of a group $\Gamma$, the following are equivalent:
\begin{enumerate}[label=(\arabic*)]
	\item There is a non-trivial group element with finite support;
	\item The alternating group over an infinite countable set embeds into $\Gamma$;
	\item The group $\Gamma$ satisfies a \emph{mixed identity} .
\end{enumerate}
In particular, any simple highly transitive group which is not the 
alternating group over an infinite countable set must be MIF (mixed identity 
free). Moreover,  the fact that the highly transitive actions of the groups 
we consider in Theorem \ref{ThmMain} are highly faithful yields that those 
groups are MIF. We refer the reader to \cite[Sec. 
5]{hullTransitivitydegreescountable2016} for the definition of mixed 
identities, and the proof of the above-mentioned result.

Finally, there are some groups for which one can actually classify 
sufficiently transitive actions, and show that none of them are highly 
transitive. The first and only examples have been 
uncovered by Le Boudec 
and Matte Bon , who proved the 
following remarkable result.
\begin{theoremSansNum}[\cite{leboudecTripleTransitivityNonfree2019}]
	Suppose a group  $\Gamma$ admits a faithful minimal action of 
	general type on a tree $\Tc$
	which is not topologically free on the boundary. Then every faithful
	$\Gamma$-action 
	of transitivity degree at least $3$ is conjugate to the restriction to 
	one orbit of the $\Gamma$-action on the
	boundary of $\Tc$ 
	(whose transitivity degree is at most $3$), and the group is not MIF.
\end{theoremSansNum}

They also proved a similar statement for groups acting on the circle, and 
provided examples of groups, coming from
\cite{leboudecGroupsActingTrees2016, le_boudec_c-simplicity_2017},
satisfying the above assumptions.
Combining their result with ours, we obtain a large class of groups for which
high transitivity is completely understood:

\begin{thmi}\label{Thm td} Let $\Gamma\curvearrowright\mathcal{T}$ be a 
	faithful minimal action of general type  on a 
	tree $\mathcal{T}$.  The following are equivalent
	\begin{enumerate}[label=(\arabic*)]
		\item $\mathrm{td}(\Gamma)\geq 4$;
		\item $\Gamma$ is highly transitive;
		\item $\Gamma$ is MIF;
		\item \label{it: top free on boundary}$\Gamma\act \partial\Tc$ is topologically free.
	\end{enumerate}
\end{thmi} 
Note that the topological freeness of the action on the boundary $\partial 
\Tc$ (item \eqref{it: top free on boundary}) is a strengthening of the global 
assumption that the $\Gamma$ action on the tree $\Tc$ is faithful.

In relation to the above quoted question by Hull and Osin, let us note that 
Theorem \ref{Thm td} yields the equivalence between high transitivity and 
infinite transitivity 
degree 
for countable groups admitting a faithful and minimal action of general type on a tree.

\subsection{The cases of amalgams and HNN extensions}
In order to prove Theorem \ref{ThmMain}, we use Bass-Serre theory and 
reduce 
the proof to the case of an HNN extension or an amalgamated free product.

Let us first describe the case of an HNN extension $\Gamma={\rm 
	HNN}(H,\Sigma,\vartheta)$. Let $\Tc$ be the Bass-Serre tree of $\Gamma$ 
(see section \ref{PrelimHNN}). 
Then it is easy to check that the action $\Gamma\curvearrowright\Tc$ is minimal of general type if and only if $\Sigma\neq H\neq\vartheta(\Sigma)$. 

The HNN extension case of Theorem \ref{ThmMain} that we show in the present paper is the following.
\begin{thmi}\label{ThmMainHNN}
	Let $\Gamma$ by an HNN extension $\Gamma={\rm HNN}(H,\Sigma,\vartheta)$ with $\Sigma\neq H\neq\vartheta(\Sigma)$.
	If the action of $\Gamma$ on the boundary of its Bass-Serre tree is topologically free, then $\Gamma$ admits a highly transitive and highly faithful action; in particular, $\Gamma$ is highly transitive.
\end{thmi}
Examples of HNN extensions which are not acylindrically hyperbolic and which do not satisfy the hypothesis of \cite{fimaHighlyTransitiveActions2015} are Baumslag-Solitar groups. 
A direct application of Theorem~\ref{ThmMainHNN} allows us to answer a 
question raised by Hull and Osin in 
\cite[Question 6.3]{hullTransitivitydegreescountable2016}: what is the 
transitivity 
degree of the non-solvable icc Baumslag-Solitar groups? Given $m,n\in\Z^*$, 
recall that the Baumslag-Solitar group with parameter $m,n$ is:
$${\rm BS}(m,n):=\langle a,b\,:\,ab^ma^{-1}=b^n\rangle.$$
It is not solvable if and only if $\vert n\vert \neq 1$ and $\vert m\vert 
\neq 1$, and icc if and only if $\vert n \vert \neq\vert m\vert$. As noted 
by Hull and Osin, if a Baumslag-Solitar group is either solvable or not 
icc, then its transitivity degree is equal to $1$.
We prove the following in Section \ref{sec: ex BS}, and 
provide more new examples to which Theorem
\ref{ThmMainHNN} applies in Sections \ref{sec: ex HNN from fs perm} and 
\ref{sec: ex fg HNN from fs perm}.

\begin{cori}\label{tdBSIntro}
	All the non-solvable icc Baumslag-Solitar groups are highly transitive. 
	In particular, the group ${\rm BS}(2,3)$ is highly transitive.	
\end{cori}


Then, let us describe the case of an amalgamated free product (or amalgam for short) $\Gamma = \Gamma_1 *_\Sigma \Gamma_2$, where $\Sigma$ is a common subgroup of $\Gamma_1$ and $\Gamma_2$.
Such an amalgam	is said to be \textbf{non-trivial} if $\Gamma_1 \neq \Sigma \neq \Gamma_2$, 
and \textbf{non-degenerate} if moreover $[\Gamma_1 : \Sigma] \geq 3$ or $[\Gamma_2 : \Sigma] \geq 3$.


Let $\Tc$ be the Bass-Serre tree of the amalgam $\Gamma$ (see Section \ref{PrelimAmalg}). It easy to see that the action $\Gamma\curvearrowright\Tc$ always minimal, is of general type 
if and only if the amalgam is non-degenerate, and is faithful if and only if $\Sigma$ is core-free in $\Gamma$. Let us now state our result in the case of an amalgam.

\begin{thmi}\label{ThmGroupsHTAmalgam}
	Consider a non-degenerate amalgam $\Gamma = \Gamma_1 *_\Sigma \Gamma_2$ and its Bass-Serre tree $\Tc$. If the induced $\Gamma$-action on $\partial\Tc$ is topologically free, then $\Gamma$ admits a highly transitive and highly faithful action; in particular, $\Gamma$ is highly transitive.
\end{thmi}

Notice that, in the context of the theorem, if the induced $\Gamma$-action 
on $\partial\Tc$ is topologically free, then obviously the action on $\Tc$ 
is faithful, hence $\Sigma$ is core-free in $\Gamma$. Sections
\ref{ExampleHTnonAcylHyp} and \ref{sec: amalgams of fs permutations} 
provides new highly transitive examples obtained via Theorem 
\ref{ThmGroupsHTAmalgam}.


\subsection{Comparison with former results}
Here is a corollary of our result which does not mention the action on the 
boundary, where given a subtree $\Tc'$ of $\Tc$, we  denote by $\Gamma_{\Tc'}$ the 
pointwise stabilizer of $\Tc'$ in $\Gamma$. 
\begin{cori}\label{core-free implies high transitivity}
	Let $\Gamma\curvearrowright\mathcal{T}$ be an action of a countable 
	group $\Gamma$ on a tree $\mathcal{T}$, which is faithful, minimal, and 
	of general type. 
	If there exist a bounded subtree $\mathcal{B}$ and a vertex $u$ in 
	$\mathcal{B}$ such that $\Gamma_\Bc$ is core-free in $\Gamma_u$, then, 
	$\Gamma$ admits an action on a countable set which is both highly 
	transitive and highly faithful. In particular, $\Gamma$ is highly 
	transitive.
\end{cori}
This corollary encompasses all the previously 
known results of high transitivity for groups with a minimal action of 
general type on a tree, which fall in two categories. The first examples 
are the acylindrically hyperbolic ones, for which one can use 
the following result by Minasyan and Osin, combined with the high transitivity 
for acylindrically hyperbolic groups result of Hull and Osin. In its 
statement, 
we denote by $[u,v]$ 
the geodesic between $u$ and $v$.
\begin{theoremSansNum}\cite[Theorem 2.1]{minasyanAcylindricalHyperbolicityGroups2015}
	Let $\Gamma$ be a group acting minimally on a tree $\Tc$. 
	Suppose that $\Gamma$ is not virtually cyclic, $\Gamma$ does
	not fix any point in
	$\partial\Tc$, and there exist vertices $u, v$ of $\Tc$ such that 
	$\Gamma_{[u,v]}$ is finite. 
	Then $\Gamma$ is acylindrically hyperbolic.
\end{theoremSansNum}

We will check in Proposition \ref{proof applicability Cor F to MO+HO}, that
all groups satisfying the hypotheses of the above theorem, and having a trivial finite radical,
also satisfy the hypotheses of Corollary \ref{core-free implies high transitivity}.

Furthermore, as Hull and Osin noticed \cite[Corollary 5.12]{hullTransitivitydegreescountable2016}, there are groups acting on trees which are non-acylindrically hyperbolic, but highly transitive thanks to the following result by the first, third, and fourth authors. In the terminology of the present article, the assertion  ``$\Gamma_e$ is highly core-free in $\Gamma_v$'' means that the action $\Gamma_v \curvearrowright \Gamma_v/\Gamma_e$ is highly faithful.
\begin{theoremSansNum}\cite[Theorem 4.1]{fimaHighlyTransitiveActions2015}
	Let a countable group $\Gamma$ act without inversion on a tree $\mathcal{T}$, and let $R\subset E(\Tc)$ be a set of representatives of the edges of the quotient graph $\Gamma \backslash \Tc$. Then $\Gamma$ is highly transitive, provided
	$\Gamma_v$ is infinite and $\Gamma_e$ is \textbf{highly core-free} in $\Gamma_v$,
	for \textbf{every couple} $(e,v)$ where $e\in R$ and $v$ is one of its endpoints.
\end{theoremSansNum}

The fact that the groups satisfying the hypotheses of the above theorem 
also satisfy those of Corollary \ref{core-free implies high transitivity} 
is checked in Proposition \ref{proof applicability Cor F to FMS}.

There are examples of groups acting on trees which are highly 
transitive thanks to Corollary~\ref{core-free implies high transitivity}, 
but to which the previously 
known results 
do not apply. 	
We also check that the icc non-solvable Baumslag-Solitar 
groups provide examples of HNN extensions for which Theorem \ref{ThmMain} applies while Corollary 
\ref{core-free implies high transitivity} does not. All these examples can be found in Section~
\ref{Examples and applications}.

\subsection{About the proofs}
In order to prove high transitivity for a general class of groups without 
constructing an explicit highly transitive action, two approaches can be 
tried. 
The first is by working in the space of subgroups of $\Gamma$, and 
proceeds by inductively building a subgroup $\Lambda\leq\Gamma$ such that 
the associated homogeneous space $\Gamma/\Lambda$ is highly transitive. 
To our knowledge, this approach made its first appearance in a paper of 
Hickin 
\cite{hickinApplicationsTreeLimitsGroups1988}, and was then made more 
explicit 
in \cite{zbMATH00120147}.
It was notably used by Chaynikov when proving that hyperbolic groups with 
trivial finite radical are highly transitive, and also by Hull and Osin in 
their aforementioned result.

The second approach, pioneered by Dixon, goes by fixing an infinite 
countable set $X$, considering a well-chosen Polish space of group actions 
on $X$, and showing that in there, the space of faithful highly transitive 
actions is a countable intersection of dense open sets, hence not empty by 
the Baire category theorem. In this work, we follow this second approach, 
using the same space of actions as the one considered in 
\cite{fimaHighlyTransitiveActions2015}, but with a much finer construction 
in order to show density. 

As explained before, the proof of the general result goes through the HNN 
and the amalgam cases. The two proofs are actually very similar, so for 
this introduction we only explain in more details what goes on for HNN 
extensions.

Given a non-degenerate HNN extension $\Gamma=\HNN(H,\Sigma,\vartheta)$, the 
idea   is to  start with a free $H$-action on a set $X$ with infinitely 
many orbits, and  then to turn it into a highly transitive faithful 
$\Gamma$-action via a generic permutation. To be more precise, the Polish 
space under consideration is the set of all permutations which intertwine 
the $\Sigma$ and the $\vartheta(\Sigma)$-actions, thus yielding a natural 
$\Gamma$-action. The result is then that there is a dense $G_\delta$ of 
such permutations which induce a highly transitive faithful 
$\Gamma$-action. 

For this to work, the notion of high core freeness was handy in \cite{fimaHighlyTransitiveActions2015}: it allows one to ``push'' the situation by a group element in $H$ so as to get to a place where both $\Sigma$ and $\vartheta(\Sigma)$ act in a more controllable way.  Let us note that this approach was generalized in \cite{fimaHomogeneousActionsUrysohn2018} to show that all the groups considered in \cite{fimaHighlyTransitiveActions2015} actually have a faithful homogeneous action onto \emph{any} bounded $S$-Urysohn space.

Here, we use a different approach, similar to the one due to the third and 
fourth named authors when 
they re-discovered the characterization
of free products of finite 
groups which are highly transitive \cite{moonHighlyTransitiveActions2013}. The 
main difficulty is that the group element that we use to ``push'' 
things out does not belong to $H$, in particular it can contain a number of 
powers of the permutation at hand.

In order to solve this, we first modify the permutation so as to make sure 
such a push is possible. The
modification  is actually very natural.
Informally speaking, there are two steps:
\begin{enumerate}
	\item ``erasing the permutation'' outside a suitable finite set of $\Sigma$-orbits and $\vartheta(\Sigma)$-orbits, which leaves us with a partial bijection;
	\item make a ``free globalization'' of this partial bijection,
	which is obtained by gluing partial bijections inducing portions of the $\Gamma$-action by right translations on itself. 
\end{enumerate}
This results in a new $\Gamma$-action satisfying a very natural universal 
property, which we state by introducing the notion of \emph{pre-action} of 
an HNN extension, see Theorem \ref{thm: HNN free globalization} (and Theorem 
\ref{thm: amalgam free globalization} for its amalgam counterpart). 
The construction in step (2) allows us to use the topological freeness of the left 
action on the boundary in order to find the further modification of the 
permutation which yields high transitivity, following an approach close to 
the proof of  \cite[Theorem 
3.3]{moonHighlyTransitiveActions2013}.  We do not know if our approach can 
be 
generalized so as to obtain faithful homogeneous action onto bounded 
$S$-Urysohn space.

\subsection{Organization of the paper}
Section \ref{Preliminaries} is a preliminary section in which we introduce our notations and definitions concerning group actions, graphs, amalgams and HNN extensions. 
Section \ref{SectFreeGlobalizationHNN} contains the main technical tools to prove Theorem \ref{ThmMainHNN}: the notion of a pre-action of an HNN extension, its Bass-Serre graph and its free globalization. 
In Section~\ref{High transitivity for HNN extensions} we prove Theorem \ref{ThmMainHNN}. 
Section \ref{SectFreeGlobalizationAmalgams} contains the main technical tools to prove Theorem \ref{ThmGroupsHTAmalgam}: the notion of a pre-action of an amalgam, its Bass-Serre graph and its free globalization. 
In Section \ref{High transitivity for amalgams} we prove Theorem \ref{ThmGroupsHTAmalgam} while in Section \ref{GroupsActingTrees} we prove Theorem \ref{ThmMain}, Theorem \ref{Thm td} and Corollary \ref{core-free implies high transitivity}. 
Section \ref{Examples and applications} is dedicated to concrete 
examples where our results apply. Finally, in Section 
\ref{MinimalityIsNecessary} we show that 
the minimality assumption in Theorem \ref{ThmMain} is needed, and we 
discuss other types of actions on trees.

\subsection{Acknowledgments} 
Y.S. is grateful to ANR SingStar 
(ANR-14-CE25-0012-01) which funded a 
visit to 
Paris dedicated to research presented in the current paper. He also warmly 
thanks Julien Bichon for showing him an elementary argument to prove that 
the group of finitely supported permutations on $\N$ is not linear. 

We are 
grateful to Adrien Le Boudec and to Nikolay A. Ivanov for pointing out a 
mistake in our examples from Section \ref{ExampleFaithfulNonTopolFree} and 
for explaining to us how to correct it. We also thank Nikolay A. Ivanov,  
Adrien Le Boudec, 
Nicolás Matte Bon and Tron Omland for their useful comments on other parts 
of the paper. Finally, we are very grateful to the referee for their
detailed remarks.

\section{Preliminaries}\label{Preliminaries}

The notation $A\Subset B$ means that $B$ is a set and $A$ is a finite 
subset of $B$.

\subsection{Group actions}\label{Group actions}
Throughout the article, we will use the symbol $X$ to denote an infinite 
countable set. 
Then, $S(X)$ denotes the Polish group of bijections of $X$. Unless 
specified otherwise, groups will act on $X$ on the \emph{right}. One of our
motivations for doing so is that we will associate paths to words in our 
groups, so it will be much easier to read both in the same order (see for 
instance Section \ref{SubsectionPathsBSHNN}). 

So given two permutations $\sigma,\tau \in S(X)$ and $x\in X$, the image of 
$x$ by $\sigma$ is denoted $x\sigma$, and the product $\sigma\tau$ is the 
permutation obtained by applying $\sigma$ first and then $\tau$. 
This way, $S(X)$ acts on $X$ on the right and any right $G$-action $X 
\curvearrowleft^\alpha G$ is equivalent to a morphism of groups 
$\alpha:G\to S(X)$. 
The image of an element $g\in G$ by $\alpha$ will be denoted by $\alpha(g)$ 
or $g^\alpha$, or just $g$ if there is no possible confusion. 
Similarly, the image of a subgroup $H$ of $G$ by $\alpha$ will be denoted 
by $\alpha(H)$ or $H^\alpha$, or just $H$.

Notice however that actions on other kinds of spaces, especially on Bass-Serre trees, will be on the left.

\begin{definition}
	An action $X\curvearrowleft G$ is \textbf{highly transitive} if, for any $k\in\N^*$ and any $k$-tuples $(x_1,\ldots,x_k)$, $(y_1,\ldots,y_k) \in X^k$, each with pairwise distinct coordinates, there exists $\gamma\in G$ such that $x_i \gamma = y_i$ for all $i=1\ldots,k$.
\end{definition}

\begin{lemma}\label{HT and disjoint supports}
	An action $X\curvearrowleft G$ is highly transitive if and only if, for every $k\in\N^*$, and every $x_1,...,x_k,y_1,...,y_k$ all pairwise distinct, we can find $g\in G$ such that $x_i g=y_i$ for $i=1,...,k$.
\end{lemma}
\begin{proof}
	Take $x_1,...,x_k$ pairwise distinct, and $z_1,...,z_k$ pairwise 
	distinct, we need to find $\gamma$ such that $x_i \gamma=z_i$ for 
	$i=1,...,k$. 
	Since $X$ is infinite, we find $y_1,...,y_k$ pairwise distinct and 
	distinct from all the $x_i$'s and $z_i$'s, 
	then by our assumption there are both $g$ and $h$ such that for all 
	$i=1,...,k$ we have $x_i g=y_i$ and $y_i h=z_i$, 
	so the element $\gamma=gh$ is the element we seek.
\end{proof}

Given a bijection $\sigma\in S(X)$, its \textbf{support} is the set $\supp 
\sigma =\{x\in X\colon x\sigma\neq x\}$.
Recall that an action $X\curvearrowleft G$ is \textbf{faithful} if for every 
$g\in G\setminus\{1\}$ the support of $g$ is not empty. 
\begin{definition}
	An action $X\curvearrowleft G$ is called \textbf{strongly faithful} if, for any finite subset $F\subseteq G \setminus \{1\} $, 
	the intersection of the supports of the elements of $F$ is not empty.
	It is called \textbf{highly faithful} if for every finite subset $F\subseteq G \setminus \{1\}$, 
	the intersection of the supports of the elements of $F$ is infinite.
\end{definition}
Given a strongly faithful action $X\curvearrowleft G$, and a finite subset 
$F\subseteq G$, it is easy to see that there exists $x\in X$ such that the 
translates $xg$, for $g\in F$, are pairwise distinct. Indeed, any element 
$x\in\bigcap_{g,h\in F}\supp(g h\inv)$ will do.

Let us check that our definition of high faithfulness coincides with the one given in 
\cite{fimaHighlyTransitiveActions2015}.

\begin{lemma}
	An action $X\curvearrowleft G$ is highly faithful if and only if for 
	every $n\in\N$, 
	if $F$ is a finite subset of $X$ and $X_1,...,X_n$ are subsets of $X$ 
	such that $X=F\cup X_1\cup\cdots\cup X_n$, 
	then there is some $k\in\{1,...,n\}$ such that for every $g\in 
	G\setminus\{1\}$, 
	there is $x\in X_k$ such that $x\cdot g\neq x$. 
\end{lemma}
\begin{proof}
	We prove the lemma by the contrapositive in both directions.
	
	Suppose that for a fixed $n\in\N$, we can find a decomposition $X=F\cup 
	X_1\cup\cdots \cup X_n$ such that for all $k\in\{1,...,n\}$, there is 
	$g_k\in G$ whose support is disjoint from $X_k$. Then in particular, 
	the 
	intersection of the supports of the $g_k$'s is contained in $F$, hence 
	finite, contradicting high faithfulness.
	
	Conversely, suppose that we found $g_1,...,g_n\in G$ whose supports 
	have finite intersection. 
	Then the sets $F=\bigcap_{k=1}^n\supp g_k$ and $X_k=X\setminus \supp 
	g_k$ satisfy that for all $k\in\{1,...,n\}$, there is some $g\in G$ 
	(namely $g_k$) such that for all $x\in X_k$, $x\cdot g=x$.
\end{proof}

Of course, we have the implications:
\[
\text{ free } \Rightarrow \text{ highly faithful }
\Rightarrow \text{ strongly faithful } \Rightarrow \text{ faithful }
\]
Let us now see that strong faithfulness and high faithfulness coincide in many cases.
\begin{proposition}
	Given an action $X\curvearrowleft G$ of a nontrivial group $G$, the following assertions are equivalent:
	\begin{enumerate}[label=(\arabic*)]
		\item the action is strongly faithful, but not highly faithful;
		\item there are finite orbits in $X$ on which $G$ acts freely (in particular, $G$ has to be finite), but only finitely many of them.
	\end{enumerate}
\end{proposition}
\begin{proof}
	Define the \emph{free part} $X_f$ of our action as the union of the orbits in $X$ on which $G$ acts freely. Note that we can write $X_f$ as 
	\[
	X_f = \bigcap_{g\in G \setminus \{1\}} \supp(g).
	\]
	
	Assume first that (2) holds. In this case, $X_f$ is a non-empty finite union of finite orbits, hence a non-empty finite set. 
	Therefore, the action is strongly faithful, since $X_f$ is non-empty, and it  is not highly faithful, since $G$ and $X_f$ are finite.
	
	Assume now that (1) holds. 
	Since the action is not highly faithful, there exists $F_0\Subset G\setminus\{1\}$ such that $\bigcap_{g\in F_0} \supp(g)$ is finite.
	Thus, the family $(Y_F)_{F_0\subseteq F \Subset G\setminus\{1\}}$ given by
	\[
	Y_F = \bigcap_{g\in F} \supp(g) 
	\]
	is a decreasing family of finite sets, which are all non-empty since the action is strongly faithful.
	Now, we have
	\[
	X_f = \bigcap_{g \in G\setminus\{1\}} \supp(g) = \bigcap_{F_0\subseteq F \Subset G\setminus\{1\}} Y_F \, ,
	\]
	hence $X_f$ is finite and non-empty. 
	Consequently, there are finite orbits in $X$ on which $G$ acts freely, and their number is finite.
\end{proof}

\begin{corollary}\label{EquivStrongAndHighFaithfulness}
	In case $G$ is infinite, an action $X\curvearrowleft G$ is highly faithful if and only if it is strongly faithful.
\end{corollary}

Let us end this section by remarking a reformulation of strong 
faithfulness which we won't use.

\begin{remark}
	A transitive action is strongly faithful if and only 
	if the stabilizer of every (or equivalently, some) point is not 
	a \emph{confined} subgroup of the acting group (see Section 1.5 from 
	\cite{matte-bonRigidityPropertiesFull2018} for a discussion of the 
	notion 
	of confined subgroup). 
\end{remark}

\subsection{Graphs}\label{PrelimGraphs} 
First, let us recall the definition of a non-simple graph. 
\begin{definition}
	A \textbf{graph} $\mathcal G$ is given by a \textbf{vertex set} 
	$V(\mathcal G)$, 
	an \textbf{edge set} $E(\mathcal G)$, 
	a fixed-point-free involution $\bar\cdot: E(\mathcal G)\to E(\mathcal G)$ 
	called the \textbf{antipode map}, 
	a \textbf{source map} $s:E(\mathcal G)\to V(\mathcal G)$ 
	and a \textbf{range map} $r:E(\mathcal G)\to V(\mathcal G)$ subject to the 
	condition:  
	\[
	\text{ for all } e\in E(\mathcal G), \quad s(\bar e)=r(e).
	\]
\end{definition}

The graph $\Gr$ is \textbf{oriented} if a partition $E(\mathcal G) = 
E(\mathcal G)^+ \sqcup E(\mathcal G)^-$ such that $E(\mathcal G)^- = 
\overline{E(\mathcal G)^+}$ is given. 
In this case, the edges in $E(\mathcal G)^+$ are called \textbf{positive 
	edges} and the edges in $E(\mathcal G)^-$ are called \textbf{negative edges}.

Recall that a \textbf{path} $\omega$ in a graph $\Gr$ is a finite sequence of 
edges $\omega=(e_1,\dots,e _n)$, such that, for all $1\leq k\leq n-1$, 
$r(e_k)=s(e_k+1)$. 
We call $s(e_1)$ the \textbf{source} of $\omega$ and $r(e_n)$ the 
\textbf{range} of $\omega$. We also say that $\omega$ is a path from 
$s(\omega):=s(e_1)$ to $r(\omega):=r(e_n)$. 
The \textbf{inverse path} of $\omega$ is defined by 
$\overline{\omega}:=(\overline{e_n},\dots,\overline{e_1})$.
The integer $n$ is called the \textbf{length} of $\omega$ and denoted by $\ell(\omega)$. 
Similarly, an \textbf{infinite path}, also called a \textbf{ray}, is a sequence of edges $\omega=(e_k)_{k\geq 1}$ such that $r(e_k)=s(e_{k+1})$ for all $k\geq 1$ and the vertex $s(\omega):=s(e_1)$ is called the \textbf{source} of $\omega$.

Given a path $\omega=(e_k)_{1\leq k \leq n}$, respectively an infinite path 
$\omega=(e_k)_{k\geq 1}$, in  $\mathcal{G}$, we use the notation 
$\omega(n):=r(e_n)$, for $n\geq 1$ and $\omega(0)=s(e_1)=s(\omega)$. 
A couple $(e_k,e_{k+1})$ such that $e_{k+1}=\overline{e_k}$, if there is one, 
is called a \textbf{backtracking} in $\omega$. If $\omega$ has no 
backtracking, we also say that it is a \textbf{reduced path}. One says that 
$\omega$ is \textbf{geodesic} in $\Gr$ if, for all $i,j$, the distance in 
$\Gr$ between $\omega(i)$ and $\omega(j)$ is exactly $|j-i|$. 
Obviously, all geodesic paths are reduced.

A \textbf{cycle} in $\Gr$ is a reduced path $c$ of length at least $1$ such 
that $s(c) = r(c)$, that is, a reduced path $c=(e_1,\dots,e _n)$ such that 
$n\geq 1$ and $c(n) = c(0)$. 
Such a cycle is \textbf{elementary} if moreover the vertices $c(k)$, for 
$0\leq k \leq n-1$, are pairwise distinct. Every cycle contains an 
elementary cycle.

When $\Gr$ is oriented, a path $\omega=(e_k)_{1\leq k \leq n}$, 
respectively an infinite path $\omega=(e_k)_{k\geq 1}$, in  $\mathcal{G}$ 
is called \textbf{positively oriented} if $e_k\in\EG^+$ for all $k$ and 
\textbf{negatively oriented} if $e_k\in\EG^-$ for all $k$; it is called 
\textbf{oriented} if it is either positively oriented or negatively 
oriented.	

\begin{definition}
	A \textbf{morphism of graphs} $f:\Gr\to\Gr'$ is a couple of maps $V(\Gr)\to V(\Gr')$ and $E(\Gr)\to E(\Gr')$, which will both be denoted by $f$ for sake of simplicity, such that $f(\bar e) = \overline{f(e)}$, $f(s(e)) = s(f(e))$, and $f(r(e)) = r(f(e))$ for all edges $e$ in $\Gr$.
\end{definition}

The \textbf{star} at a vertex $v$ is the set $\st(v)$ of edges whose source is $v$, and its cardinality is called the \textbf{degree} of $v$. 
A morphism of graphs $f:\Gr\to\Gr'$ is \textbf{locally injective} if, for all $v \in V(\Gr)$, the restriction of $f$ to the star of $v$ is injective. 
Note that a locally injective morphism from a connected graph to a tree is injective.

\begin{definition}
	Given a graph $\mathcal G$, and a set of edges $E\subseteq E(\mathcal G)$, we associate to this subset the \textbf{induced subgraph}  as the graph $\Hc$ such that $V(\Hc) =s(E)\cup r(E)$, $E(\Hc) = E\cup\bar E$, and the structure maps of $\Hc$ are the restrictions of those of $\Gr$.
\end{definition}

\begin{definition}\label{def:halfgraph}
	Given an edge $e$, its associated \textbf{half-graph} is the subgraph induced by the set of edges $f$ such that there is a reduced path starting by $e$, not using $\bar e$, and whose last edge is equal to $f$.
\end{definition}

\begin{remark}\label{RemHalGraphs}
	Suppose $\Gr$ is connected and $e$ is an edge, let $\Hc_e$ and $\Hc_{\bar e}$ be the half-graphs associated to $e$ and $\bar e$. Then, one has $V(\Gr) = V(\Hc_e) \cup V(\Hc_{\bar e})$ and $E(\Gr)= E(\Hc_e) \cup E(\Hc_{\bar e})$.
	Moreover, denoting by $\Gr_0$ the graph obtained by deleting the edges 
	$e,\bar e$ in $\Gr$:
	\begin{enumerate}
		\item if $\Gr_0$ remains connected, then one has $\Hc_e = \Gr = \Hc_{\bar e}$;
		\item if not, then $\Hc_e$ and $\Hc_{\bar e}$ are obtained by adding the edges $e,\bar e$ to the respective connected components of $r(e)$ and $s(e)$ in $\Gr_0$.
	\end{enumerate}
\end{remark}

\subsection{Trees and their automorphisms}\label{PrelimTrees}
For a more detailed account of what follows, we refer the reader to 
\cite{delaharpeSimpleGroupsAmalgamated2011}.
A \textbf{forest} is a graph with no cycle and a \textbf{tree} is a connected forest. In a forest, we recall that any reduced (finite or infinite) path is geodesic. Moreover, any two vertices in a tree are connected by a unique reduced path.

There is a well-known classification (see e.g. \cite{serreTrees1980}) of the 
automorphisms of a tree $\Tc$: if $g$ is such an automorphism, then:
\begin{itemize}
	\item either $g$ is \textbf{elliptic}, which means that $g$ fixes some vertex of $\Tc$,
	\item or $g$ is an \textbf{inversion}, which means that $g$ sends some edge $e$ onto its antipode $\bar e$,
	\item or $g$ is \textbf{hyperbolic}, which means that $g$ acts by a (non-trivial) translation on a bi-infinite geodesic path, called its \textbf{axis}.
\end{itemize}

\begin{definition}
	The \textbf{boundary} $\partial \Tc$ (or set of \textbf{ends}) of a tree $\Tc$ is the set of geodesic 
	rays quotiented by the equivalence relation which identifies two geodesic 
	rays whose 
	ranges differ by a finite set. 	
\end{definition}

Note that since we are working in a tree, if we fix a vertex $o$ then the set 
of geodesic rays starting at $o$ is in bijection with the boundary of the tree 
through the 
quotient map. 

The boundary is equipped with the topology whose basic open sets $U_{\Hc}$ are 
given by fixing a half-tree $\mathcal H$, and letting 
$U_{\Hc}$ be the set of equivalence classes of geodesic rays whose range is 
contained in 
$\Hc$ (so given a geodesic ray $\omega$, its class belongs to $U_{\Hc}$ if and only if 
some terminal subpath of $\omega$ is contained in $\Hc$). 

Any $g\in \Aut(\Tc)$ induces a homeomorphism of $\partial\Tc$, which yields a 
group homomorphism
\[
\Aut(\Tc) \longrightarrow \operatorname{Homeo}(\partial \Tc) \, .
\]
In case $g$ is hyperbolic, $g$ fixes exactly two points in $\partial \Tc$, which are the endpoints of its axis. 
Let us also recall that every action on a tree $\Gamma\curvearrowright \Tc$ 
satisfies exactly one of the following:
\begin{itemize}
	\item it is \textbf{elliptic}, which means that (the image of) $\Gamma$ 
	stabilizes some vertex, or some pair of antipodal edges;
	\item it is \textbf{parabolic}, or \textbf{horocyclic}, that is, $\Gamma$ 
	contains no hyperbolic elements, without being elliptic itself; 
	\item it is \textbf{lineal}, that is $\Gamma$ contains hyperbolic elements, 
	all of them sharing the same axis;
	\item it is \textbf{quasi-parabolic}, or \textbf{focal}, that is $\Gamma$ 
	contains  hyperbolic elements with different axes, but all hyperbolic elements 
	of $\Gamma$ share a common fixed point in $\partial \Tc$;
	\item it is \textbf{of general type}, which means that $\Gamma$ contains two 
	hyperbolic elements with no common fixed point in $\partial \Tc$ (such 
	hyperbolic elements are called \textbf{transverse}).
\end{itemize}
In the parabolic case, it can be shown that every element of $\Gamma$ is 
elliptic, and that $\Gamma$ stabilizes a unique point in $\partial \Tc$. 
In the quasi-parabolic case, it can be shown that the common fixed point in 
$\partial \Tc$ of the hyperbolic elements is unique and fixed by $\Gamma$. 
If $\Gamma\curvearrowright \Tc$ is of general type, it is easy to produce 
infinitely many pairwise transverse hyperbolic elements. An action on 
a tree is called \textbf{minimal} if there is no invariant subtree. 
Since we will be interested in \emph{faithful} minimal actions 
of infinite countable groups on trees, elliptic actions won't occur. Moreover
every vertex will have degree at least $2$, since otherwise
we could trim off all vertices of degree $1$ and get a proper invariant subtree.

Let us recall that the action $\Gamma\curvearrowright \partial\Tc$ by 
homeomorphisms is \textbf{topologically free} if the trivial element is the 
only element in $\Gamma$ which fixes a non-empty open subset of $\partial\Tc$ 
pointwise. We will rather use the following concrete characterization.

\begin{proposition}\label{half-trees and topological freeness}
	Let $\Tc$ be a tree with at least three ends. 
	Given a faithful minimal action of an infinite group $\Gamma$ on 
	$\Tc$, the 
	following are equivalent:
	\begin{enumerate}[label=(\roman*)]
		\item \label{cond: topo free}the induced action $\Gamma\curvearrowright \partial\Tc$ is 
		topologically 
		free;
		\item \label{cond: no fix half tree}no element of $\Gamma\setminus\{1\}$ can fix pointwise a half-tree in 
		$\Tc$.
	\end{enumerate}  
\end{proposition}
\begin{proof}
	The implication from \ref{cond: topo free} to \ref{cond: no fix half tree} is clear since half-trees do define open subsets
	for the topology of $\partial\Tc$.
	Conversely, assume \ref{cond: no fix half tree}. To prove \ref{cond: topo free}, let us fix $\gamma\in\Gamma\setminus\{1\}$ and show that $\gamma$ does not fix any basic open set $U_\Hc$ pointwise. 
	
	Notice that an inversion does not fix any point in $\partial\Tc$. 
	Moreover a hyperbolic automorphism $h$ has exactly two fixed points $\xi^\pm$ and, taking a third end $\eta\in\partial \Tc$, 
	one has $\eta h^{\pm n} \to \xi^{\pm}$ as $n\to +\infty$, 
	so that $\{\xi \in \partial\Tc: \xi h = \xi\}$ has empty interior. 
	Hence, the only case to check is when $\gamma$ is elliptic.
	Notice also that, by minimality of the action, every vertex has degree at least $2$ in $\Tc$, so that every half-tree is covered by the geodesic rays in it.
	
	Consider any basic open set $U_\Hc$, given by a half-tree $\Hc$.
	We claim that $\Hc$ contains a geodesic ray which does not meet the subtree $\mathrm{Fix}(\gamma)$ of $\gamma$-fixed points.
	Indeed, take any geodesic ray $r$ in $\Hc$. If $r$ meets $\mathrm{Fix}(\gamma)$ at some vertex $v$, then consider the edge $e$ in $r$ whose source is $v$, and the half-tree $\Hc_e$ it defines.
	By \ref{cond: no fix half tree}, there is a vertex $w$ in $\Hc_e$ which $\gamma$ does not fix. Extend the geodesic from $v$ to $w$ to a geodesic ray in $\Hc_e$. The tail of this ray from $w$ does not meet $\mathrm{Fix}(\gamma)$, since the latter is a subtree containig $v$. The claim is proved.
	
	Now, since $\gamma$ is elliptic, our ray which does not meet $\mathrm{Fix}(\gamma)$ is moved by $\gamma$ onto a disjoint geodesic ray in $\Tc$. This corresponds to a point $\xi\in U_\Hc$ such that $\xi \gamma \neq \xi$.
\end{proof}

Note that topological freeness of the action on the boundary is called 
\emph{slenderness} by de la Harpe and Préaux 
\cite{delaharpeSimpleGroupsAmalgamated2011}.
Although we won't use it, let us mention that for a minimal action of general 
type, the topological 
freeness 
of the action on the boundary is also equivalent to the action on the tree 
itself being strongly faithful (see \cite[Prop. 
3.8]{bryderSimplicityHNNExtensions2019} for this and other characterizations).

\subsection{Treeing edges}Let us now turn to the link between half-graphs, seen in Section \ref{PrelimGraphs}, and trees.

\begin{definition}
	An edge in a graph $\Gr$ is a \textbf{treeing edge} when its associated half-graph is a tree, in which case we also call the latter its \textbf{half-tree}.
\end{definition}

Here is an easy characterization of treeing edges that will prove useful.

\begin{lemma}\label{Lemma Treeing Edge}
	Let $\mathcal G$ be a graph, let $e$ be an edge. Then the following are equivalent:
	\begin{enumerate}[label=(\roman*)]
		\item \label{cond: treeing edge}the edge $e$ is a treeing edge;
		\item \label{cond: injectivity}the map which takes a reduced path starting by $e$ to its range is injective;
		\item \label{cond: no reduced path s(e) to s(e)}there is no reduced path from $s(e)$ to $s(e)$ starting by the edge $e$.
		
	\end{enumerate}
\end{lemma}
\begin{proof}
	First note that \ref{cond: treeing edge} implies \ref{cond: injectivity} since when $e$ is a treeing edge, all the reduced paths starting by $e$ must belong to its half-tree, and hence have distinct ranges.
	
	We then show that \ref{cond: injectivity} implies \ref{cond: no reduced path s(e) to s(e)} by the contrapositive. If \ref{cond: no reduced path s(e) to s(e)} does not hold, let $c$ be a reduced path starting by the edge $e$ from $s(e)$ to $s(e)$. Then $c$ and the reduction of $cc$ have the same range, so \ref{cond: injectivity} does not hold.
	
	Finally we show that \ref{cond: no reduced path s(e) to s(e)} implies \ref{cond: treeing edge} by the contrapositive. If $e$ is not a treeing edge, consider the following two cases:
	
	\begin{itemize}
		\item In the half-graph of $e$, the vertex $s(e)$ has degree at least two. We then fix some $e'\neq e$ such that $s(e')=s(e)$. If $r(e')=r(e)$ then the reduced path $e\overline{e'}$ witnesses that \ref{cond: no reduced path s(e) to s(e)} does not hold. 
		
		Otherwise by the definition of the half-graph we find a reduced path 
		$\omega$ starting by $e$ whose last edge is either $e'$ or 
		$\overline{e'}$. If the last edge is $\overline{e'}$, then $\omega$ 
		witnesses that \ref{cond: no reduced path s(e) to s(e)} does not 
		hold. If the last edge of $\omega$ is $e'$, then write 
		$\omega=\omega'e'$ and note that $\omega'$ witnesses that \ref{cond: no reduced path s(e) to s(e)} does not hold. So in any case, 
		\ref{cond: no reduced path s(e) to s(e)} does not hold.
		\item In the half-graph of $e$, the vertex $s(e)$ has degree $1$. Then since $e$ is not a treeing edge, we find a non-empty reduced path $\omega$ starting and ending at $r(e)$, and using neither $e$ nor $\bar e$. Then $e\omega \bar e$ witnesses that \ref{cond: no reduced path s(e) to s(e)} does not hold.
	\end{itemize}
	This finishes the proof of the equivalences.
\end{proof}
Note that if a reduced path uses a treeing edge at some point, then from that point on it only uses treeing edges. Moreover, we have the following result.

\begin{lemma}\label{lem: extend path with treeing edge}
	Let $\mathcal G$ be a connected graph admitting a treeing edge, and let $\omega$ be a reduced path in $\mathcal G$. Then $\omega$ can be extended to a reduced path $\omega'$ whose last edge is a treeing edge.
\end{lemma}
\begin{proof}
	Let $e$ be the last edge of $\omega$. If $e$ is a treeing edge, we can take $\omega'=\omega$. If not, by the previous lemma there is a reduced path of the form $ec$ from $s(e)$ to $s(e)$. Let $e'$ be a treeing edge, and denote by $C$ the set of vertices visited by the reduced path $c$. 
	
	We then claim that $s(e')$ is strictly closer to $C$ than $r(e')$. Indeed, otherwise, if we fix a geodesic $\eta$ from $r(e')$ to $C$, the geodesic $\eta$ cannot start by $\bar{e'}$, and there exists a cycle $\kappa$ based at $r(\eta)$ and whose vertices belong to $C$. Then the reduced path $e'\eta\kappa\bar\eta\bar{e'}$ witnesses that $e'$ does not satisfy condition \ref{cond: no reduced path s(e) to s(e)} from the previous lemma, so $e'$ is not a treeing edge, a contradiction.
	
	Now let $\xi$ be a geodesic from $C$ to $s(e')$, by the previous claim we know that $\xi e'$ is still a reduced path. Let $c'$ be the initial segment of $c$ which connects $r(e)$ to the source of $\xi$, then $\omega' = \omega c'\xi e'$ is the desired extension of $\omega$.
\end{proof}

\subsection{HNN extensions}\label{PrelimHNN}
Let $H$ be a group, and let $\vartheta : \Sigma \to \vartheta(\Sigma)$ be an 
isomorphism between subgroups of $H$. The \textbf{HNN extension} associated to 
these data is the group defined by the following presentation 
\[
\HNN(H,\Sigma,\vartheta) := \Pres{H, t}{t^{-1}\sigma t = \vartheta(\sigma) \text{ for all } \sigma\in \Sigma} \, ,
\ \footnote{This notation means that $\HNN(H,\Sigma,\vartheta)$ is the quotient of the free product $H * \langle t \rangle$
	by its smallest normal subgroup containing all elements $t\inv \sigma t \vartheta(\sigma)\inv$ where $\sigma \in \Sigma$.}
\]
where $t$ is an extra generator, called the stable letter, not belonging to 
$H$. 
We refer the reader to \cite[Chap.~1, Prop.~5]{serreTrees1980} for the fact that the HNN extension
defined above
does contain $H$ as a natural subgroup.
Note that the defining relation $t^{-1}\sigma t = \vartheta(\sigma)$ is 
different from the one chosen in \cite{fimaHighlyTransitiveActions2015, 
	fimaHomogeneousActionsUrysohn2018}.
This change is coherent with our choice to let groups act on the right on sets.

We will denote this HNN extension by $\Gamma$. Recall that it is called 
\textbf{ascending} if one of the subgroups $\Sigma,\vartheta(\Sigma)$ is equal 
to $H$.

Let us fix a set of representatives $C^+$ of left $\Sigma$-cosets in $H$, and a set of representatives $C^-$  of left $\vartheta(\Sigma)$-cosets in $H$, which both contain $1$, so that we have
\[
H=\bigsqcup_{c\in C^+} c\Sigma = \Sigma \sqcup \bigsqcup_{c\in C^+ \setminus \{1\}} c\Sigma 
\quad \text{ and } \quad
H = \bigsqcup_{c\in C^-} c\vartheta(\Sigma) = \vartheta(\Sigma) \sqcup \bigsqcup_{c\in C^- \setminus \{1\}} c\vartheta(\Sigma) \, .
\]
It is well-known, see e.g.  \cite{lyndon_combinatorial_2001}, that every element $\gamma\in \Gamma$ admits a unique 
\textbf{normal form}
\[
\gamma = c_{1}t^{\varepsilon_1} \cdots c_n t^{\varepsilon_n}h_{n+1}, 
\]
where $n\geq 0$, $\varepsilon_i = \pm 1$ for $1\leq i \leq n$, $\varepsilon_i=+1$ 
implies $c_i\in C^+$, $\varepsilon_i=-1$ implies $c_i\in C^-$, $h_{n+1}\in H$, and 
there is no subword of the form $t^\varepsilon1t^{-\varepsilon}$. Note that the case 
$n=0$ corresponds to elements in $H$.

The \textbf{Bass-Serre tree} of the HNN extension $\Gamma$ is the oriented graph $\Tc$ defined by
\[
V(\Tc) = \Gamma/H \, ; \quad
E(\Tc)^+ = \Gamma/\Sigma \, ; \quad
E(\Tc)^- = \Gamma/\vartheta(\Sigma) \, ; \quad
\]
where the structural maps are given by the following  formulas
\[
\overline{\gamma\Sigma} = \gamma t \vartheta(\Sigma) \, ; \quad
s(\gamma\Sigma) = \gamma H \, ; \quad
r(\gamma\Sigma) = \gamma t H \, ; \quad
\]
\[
\overline{\gamma\vartheta(\Sigma)} = \gamma t^{-1} \Sigma \, ; \quad
s(\gamma\vartheta(\Sigma)) = \gamma H \, ; \quad
r(\gamma\vartheta(\Sigma)) = \gamma t^{-1} H \, .
\]
This graph is naturally endowed with a left $\Gamma$-action by graph 
automorphisms (respecting the orientation), and classical Bass-Serre theory 
\cite{serreTrees1980} ensures it is a tree. The 
action is always minimal since it is transitive on the vertices.
Let us now recall what kind of action $\Gamma\curvearrowright\Tc$ is, depending on the inclusions $\Sigma\subseteq H$ and $\vartheta(\Sigma)\subseteq H$. Note that the stable letter $t$ always induces a hyperbolic automorphism.
\begin{itemize}
	\item If $\Sigma = H = \vartheta(\Sigma)$, then the Bass-Serre tree is a bi-infinite line (each vertex has degree $2$), hence, the action is lineal.
	
	\item If $\Sigma = H$ and $\vartheta(\Sigma) \neq H$, then there is exactly one positive edge and several negative edges in the star at each vertex. Hence, to each vertex $v$, one can associate a reduced infinite path $\omega_v^+$ starting at $v$ by taking the unique positive edge at each vertex. Given two vertices $u$ and $v$, the paths $\omega_u^+$ and $\omega_v^+$ share a common terminal subpath. Indeed, this is obvious if $u,v$ are linked by an edge, and then, denoting $v_0,\ldots,v_n$ the vertices on the geodesic between $u$ and $v$, all the paths $\omega_{v_i}^+$ share a common terminal subpath.
	
	Now, let $\xi \in \partial \Tc$ be the common endpoint of all paths $\omega_v^+$. Given any hyperbolic element $g\in \Gamma$, 
	and any vertex $v$ in $\Tc$, we have $g\cdot\omega_v^+ = \omega_{g v}^+$ since $\Gamma$ preserves the orientation, whence $g\xi = \xi$. Therfore, all hyperbolic elements of $\Gamma$ fix $\xi$. 
	
	On the other hand, it is easy to see that $t$ and $c^{-1}tc$ don't have the same axis. Hence, the action is quasi-parabolic.
	
	\item Similarly, if $\Sigma \neq H$ and $\vartheta(\Sigma) = H$, then the action is quasi-parabolic.
	
	\item If the HNN extension is non-ascending, then taking $h$ in $C^+ \setminus\{1\}$ and $g\in C^- \setminus\{1\}$, it is fairly easy to see that $g t$ and $t h$ are transverse hyperbolic elements. Hence, the action is of general type. 
\end{itemize}

\subsection{Group amalgams}\label{PrelimAmalg}

Let $\iota_1:\Sigma\to \Gamma_1$ and $\iota_2:\Sigma\to \Gamma_2$ be 
injective morphisms of countable groups. We will denote by $\Sigma_j$ the 
image of $\iota_j$, and by $\vartheta:\Sigma_1 \to \Sigma_2$ the 
isomorphism sending $\iota_1(\sigma)$ to $\iota_2(\sigma)$ for all 
$\sigma\in\Sigma$. The \textbf{free product with amalgamation} (or 
\textbf{amalgam} for short) associated to these data is
\[
\Gamma_1 *_\Sigma \Gamma_2 := \Pres{\Gamma_1,\Gamma_2}{\iota_1(\sigma) = 
	\iota_2(\sigma) \text{ for all } \sigma \in \Sigma} 
= \Pres{\Gamma_1,\Gamma_2}{\sigma = \vartheta(\sigma) \text{ for all } 
	\sigma \in \Sigma_1}\, .
\ \footnote{More precisely, $\Gamma_1 *_\Sigma \Gamma_2$ is the quotient of the free product $\Gamma_1 * \Gamma_2$ 
	by its smallest normal subgroup containing all elements $\iota_1(\sigma) 
	\iota_2(\sigma)\inv$ where $\sigma \in \Sigma$.}
\]
We will denote the amalgam $\Gamma_1 *_\Sigma \Gamma_2$ by $\Gamma$. We 
will still denote by $\Sigma,\Gamma_1,\Gamma_2$ the images of these groups 
in the amalgam $\Gamma$ when there is no risk of confusion. In $\Gamma$, 
one has $\Gamma_1\cap \Gamma_2 = \Sigma$.
Recall that such an amalgam is said to be \textbf{non-trivial} if $\Gamma_j 
\neq \Sigma_j$ for $j=1,2$, and \textbf{non-degenerate} if moreover 
$[\Gamma_1 : \Sigma_1] \geq 3$ or $[\Gamma_2 : \Sigma_2] \geq 3$. 

Let us fix sets of representatives $C_j$ of left $\Sigma_j$-cosets in 
$\Gamma_j$, for $j=1,2$, which both contain $1$, so that we have
\[
\Gamma_1 = \bigsqcup_{c\in C_1} c\Sigma_1 = \Sigma_1 \sqcup \bigsqcup_{c\in 
	C_1\setminus\{1\}} c\Sigma_1
\quad \text{ and } \quad
\Gamma_2 = \bigsqcup_{c\in C_2} c\Sigma_2 = \Sigma_2 \sqcup \bigsqcup_{c\in 
	C_2\setminus\{1\}} c\Sigma_2 \, .
\]
Notice that the intersection of the images of $C_1$ and $C_2$ in $\Gamma$ 
is just $\{1\}$. It is well-known, see e.g. \cite{serreTrees1980}, that any 
element 
$\gamma\in \Gamma\setminus\Sigma$ admits a unique \textbf{normal form}
\[
\gamma = c_1 \cdots c_n \sigma
\]
where $n\in\N$, $c_1, \ldots, c_n$ lie alternatively in $C_1\setminus\{1\}$ and 
$C_2\setminus\{1\}$, and $\sigma\in \Sigma$. 

The \textbf{Bass-Serre tree} of the amalgam $\Gamma = \Gamma_1 *_\Sigma 
\Gamma_2$ is the oriented graph $\Tc$ defined by
\[
V(\Tc) = \Gamma/\Gamma_1 \, \sqcup \,  \Gamma/\Gamma_2 \, ; \quad
E(\Tc)^+ = \Gamma/\Sigma \, ; \quad
s(\gamma\Sigma) = \gamma\Gamma_1 \, ; \quad
r(\gamma\Sigma) = \gamma \Gamma_2
\]
(the set of negative edges $E(\Tc)^-$ just being $\overline{\Gamma/\Sigma} 
:= \{\bar e : \, e\in \Gamma/\Sigma \}$, which is another copy of 
$\Gamma/\Sigma$).
Again, this graph is naturally endowed with a left $\Gamma$-action by graph 
automorphisms (respecting the orientation), and classical Bass-Serre theory 
\cite{serreTrees1980} ensures it is a tree. The 
action is always minimal since $\Gamma$ acts transitively on the set of 
positive edges.
Let us now recall what kind of action $\Gamma\curvearrowright\Tc$ is, 
depending on the inclusions $\Sigma_j\subseteq \Gamma_j$.
\begin{itemize}
	\item If $\Sigma_1 = \Gamma_1$, then $\Gamma=\Gamma_2$, and the vertex 
	$\Gamma_2$ of $\Tc$ is fixed. Hence the action is elliptic.
	Similarly, if $\Sigma_2 = \Gamma_2$, then the action is elliptic.
	\item If the amalgam is non-trivial, and  $[\Gamma_1 : \Sigma_1] = 2 = 
	[\Gamma_2 : \Sigma_2]$, then the Bass-Serre tree is a bi-infinite line 
	(each vertex has degree $2$) and for any $\gamma_j \in \Gamma_j - 
	\Sigma_j$, $j=1,2$, the element $\gamma_1\gamma_2$ is hyperbolic. 
	Hence, the action is lineal.
	\item If the amalgam is non-degenerate, then the action is of general 
	type. Indeed, assuming $[\Gamma_1:\Sigma]\geq 3$, and taking 
	$g_1\neq g_2$ 
	in $C_1\setminus\{1\}$ and $h\in C_2\setminus\{1\}$, it is fairly easy to see that $g_1 
	h$ and $g_2 h$ are transverse hyperbolic elements. The case 
	$[\Gamma_2,\Sigma]\geq 3$ is similar.
\end{itemize}

%
\subsection{Partial actions}\label{sec: prelim partial action}

The \emph{pre-actions} that we will define below are 
tightly linked with the notion of \emph{partial action}. Although the 
latter do not play an essential role in our construction, we will see that 
every pre-action yields a natural partial action, so we feel these are 
worth mentioning. For more details on partial actions, we refer the reader 
to \cite{kellendonkPartialActionsGroups2004}.

Given a set $X$, we denote by $\I(X)$ the set of all partial bijections of 
$X$, which we think of as subsets of $X\times X$ whose vertical and 
horizontal 
fibers all have cardinality at most $1$. We have a natural composition law 
on subsets of $X\times X$ given by: for all $A,B\subseteq X\times X$, \[
AB=\{(x,z)\colon \exists y\in X, (x,y)\in A\text{ 
	and } (y,z)\in B\},
\]
and this restricts to a composition law on $\I(X)$. The inclusion provides 
us with a natural partial order on $\I(X)$. The projection on the first 
coordinate of a partial bijection $\tau$ is its \textbf{domain} $\dom\tau$, 
and the 
projection on the second coordinate is its \textbf{range} $\rng\tau$. 
Finally, we define the inversion map by $\sigma\inv=\{(y,x): (x,y)\in 
\sigma\}$.

\begin{definition}
	A (right) \textbf{partial action} of a group $\Gamma$ on a set $X$ is a 
	map 
	$\pi:\Gamma\to\I(X)$ such that for all $g,h\in\Gamma$
	\begin{enumerate}[label=(\arabic*)]
		\item $\pi(1_\Gamma)=\id_X$;
		\item $\pi(g)\pi(h)\subseteq \pi(gh)$;
		\item $\pi(g)\inv=\pi(g\inv)$.
	\end{enumerate}
\end{definition}

The main example of a partial action is provided by the restriction of an 
action to a subset. Conversely, every partial action is the restriction of 
a global action, and there is a \emph{universal} such global action 
provided by the following result.

\begin{theorem}[{see 
		\cite[Theorem 3.4]{kellendonkPartialActionsGroups2004}}]
	\label{thm: universal globalization}
	Given a partial action of a countable group $\Gamma$ on a set $X$, 
	there is a $\Gamma$-action on a larger set $\tilde X$ such that 
	whenever $Y\curvearrowleft \Gamma$ is a $\Gamma$-action on a set $Y$ which 
	contains $X$, there is a unique $\Gamma$-equivariant map $f:\tilde X\to 
	Y$ which restricts to the identity on $X$.
\end{theorem}

The action $\tilde X\curvearrowleft \Gamma$ from the previous theorem is 
called the \textbf{universal globalization} of the partial $\Gamma$-action 
on $X$. It is tacit in the theorem that the $\Gamma$-actions on sets containing $X$ extend the initial partial $\Gamma$-action on $X$.

\begin{definition}
	A partial action $X\lact^\pi\Gamma$ is called \textbf{strongly 
		faithful} 
	if for 
	every $F\Subset\Gamma\setminus\{1\}$, there is $x\in X$ such that for 
	all $g\in F$, 
	we have $x\pi(g)\neq x$ (in particular $x\in\bigcap_{g\in F}\dom\pi(g)$).
\end{definition}
\begin{example}
	The partial action of the free group on two generator $\mathbb F_2$ on 
	the set 
	of reduced words which begin by $a$ is strongly faithful.
\end{example}
\section{Free globalizations for pre-actions of HNN extensions}\label{SectFreeGlobalizationHNN}

For this section, as in Section \ref{PrelimHNN}, 
let us fix an HNN extension $\Gamma=\HNN(H,\Sigma,\vartheta)$.
Let us also fix a set of representatives $C^+$ of left $\Sigma$-cosets in $H$, 
and a set of representatives $C^-$  of left $\vartheta(\Sigma)$-cosets in $H$, 
which both contain $1$, 
so that normal forms of elements of $\Gamma$ are well-defined. 
Let us also denote by $\Gamma^+$, respectively $\Gamma^-$, the set of elements  whose normal form leftmost's letter is $t$, respectively $t\inv$. 
Note that $\Gamma^+$ is invariant by left $\Sigma$-multiplication, while $\Gamma^-$ is invariant by left $\vartheta(\Sigma)$-multiplication. 
We then have $\Gamma=H\sqcup C^+\Gamma^+ \sqcup C^-\Gamma^-$.


\subsection{Pre-actions of HNN extensions and their Bass-Serre graph}\label{SubsectPreActionsAndBSGraphsHNN}

Given an action on an infinite countable set $X\curvearrowleft^{\pi} H$,
and a bijection $\tau:X\to X$ such that $\sigma^{\pi}\tau = \tau\vartheta(\sigma)^{\pi}$ for all $\sigma\in \Sigma$,
there exists a unique action $X\curvearrowleft^{\pi_{\tau}} \Gamma$ 
such that $h^{\pi_{\tau}} = h^{\pi}$ for all $h\in H$, 
and $t^{\pi_{\tau}} = \tau$. 
If the action $\pi$ is free, we obtain an example of the following situation.

\begin{definition}\label{def: pre-action}
	A \textbf{pre-action} of the HNN extension $\Gamma$ is a couple $(X,\tau)$ where $X$ is an infinite countable set endowed with a free action $X \curvearrowleft^{\pi} H$, and
	\[
	\tau:\dom(\tau) \to \rng(\tau) 
	\]
	is a partial bijection where $\dom(\tau), \rng(\tau) \subseteq X$, and $\sigma^{\pi}\tau = \tau\vartheta(\sigma)^{\pi}$ for all $\sigma\in \Sigma$.
\end{definition}

The relations $\sigma^{\pi}\tau = \tau\vartheta(\sigma)^{\pi}$ in Definition \ref{def: pre-action} 
are equalities between partial bijections. In particular $\sigma^{\pi}\tau$ and $\tau\vartheta(\sigma)^{\pi}$ must have the same domain and the same range. As a consequence, for any pre-action $(X,\tau)$, the domain of $\tau$ is necessarily $\Sigma$-invariant, its range is necessarily $\vartheta(\Sigma)$-invariant, and $\tau$ sends $\Sigma$-orbits onto $\vartheta(\Sigma)$-orbits. 

A pre-action $(X, \tau)$ is called \textbf{global} if $\tau$ is a genuine 
permutation of $X$. In this case there is an associated action 
$X\curvearrowleft^{\pi_{\tau}} \Gamma$ as above. We will often identify 
global pre-actions and $\Gamma$-actions.
\begin{example}
	If $X \curvearrowleft^\pi \Gamma$ is an action, where $H$ is acting freely, then denoting by $X \curvearrowleft^{\pi_H} H$ its restriction, 
	one obtains a global pre-action $(X,t^{\pi})$, where $X$ is endowed with 
	$\pi_H$. 
	The action $X \curvearrowleft^{\pi_{\tau}} \Gamma$ coincides with $X \curvearrowleft^\pi \Gamma$ in this case. 
	In particular, the right translation action $\Gamma\curvearrowleft\Gamma$
	gives rise to a pre-action $(\Gamma,t^\rho)$, where $t^\rho: \gamma \mapsto \gamma t$, 
	called the \textbf{translation pre-action}.
\end{example}

The above notion of pre-action is close to the notion of a partial action developed in \cite{kellendonkPartialActionsGroups2004} as we will see.
As seen before, actions of $\Gamma$ (such that $H$ acts freely) correspond 
to pre-actions with a global bijection. Another source of examples of 
pre-actions is the following.

\begin{definition}
	Given a pre-action $(X,\tau)$, and an infinite $H$-invariant subset $Y\subseteq X$, 
	the \textbf{restriction} of $(X,\tau)$ to $Y$ is the pre-action $(Y,\tau')$, where $Y$ is endowed with the restriction of $\pi$, 
	and the partial bijection is $\tau'=\tau_{\restriction Y\cap Y\tau\inv}$.
	An \textbf{extension} of $(X,\tau)$ is a pre-action $(\tilde X, \tilde \tau)$ 
	whose restriction to $X$ is $(X,\tau)$.
\end{definition}

\begin{example}
	The sets $\Gamma^+$ and $\Gamma^-$ are $H$-invariant (by right multiplications), thus so are $T^+:=\Gamma^+\sqcup H$ and $T^-:=\Gamma^-\sqcup H$.
	The translation pre-action $(\Gamma,t^\rho)$ admits the restrictions $(T^+,\tau_+)$, and $(T^-,\tau_-)$,
	which we call the \textbf{positive translation pre-action} and the \textbf{negative translation pre-action} respectively. 
\end{example}
Let us compute the domain and range of the partial bijection $\tau_+$ corresponding to the positive translation pre-action. 
Let $x\in T^+$. If $x$ belongs to $\Gamma^+$, then so does $xt$, and so $\tau_+(x)$ is defined.
But if $x$ belongs to $H$ instead, then $xt\not\in H$, and $xt\in\Gamma^+$ if and only if its lefttmost letter is $t$, which happens if and only if $x\in \Sigma$. 
Reciprocally, one has $y t\inv \in T^+$ for every $y\in \Gamma^+$, and  $yt\inv \notin T^+$ for every $y \in H$. 
We conclude that the domain of $\tau_+$ is $\Sigma\sqcup \Gamma^+$, while its range is $\Gamma^+$.

The same computation can be made for the partial bijection $\tau_-$ associated to the negative translation pre-action on $T^-$:  
the domain of $\tau_-$ is equal to $\Gamma^-$, and its range is equal to $\Gamma^-\sqcup \vartheta(\Sigma)$.

Let us now associate a graph to any  $\Gamma$-pre-action $(X,\tau)$ as follows. Informally speaking, we start with a graph whose vertices are of two kinds: the $\Sigma$-orbits in $X$, and the $\vartheta(\Sigma)$-orbits in $X$. Then we put an edge from $x\Sigma$ to $y\vartheta(\Sigma)$ when $(x\Sigma)\tau = y\vartheta(\Sigma)$, and finally we identify all the $\Sigma$-orbits and all $\vartheta(\Sigma)$-orbits that are in a same $H$-orbit. We may, and will, identify $H$ with its image in $S(X)$ by the action $X\curvearrowleft^\pi H$, since the action $\pi$ is free, hence faithful. Consequently, we don't write superscripts $\pi$ from now, as soon as there is no risk of confusion. 
\begin{definition}\label{defBassSerreGraphHNN}
	The \textbf{Bass-Serre graph} of $(X,\tau)$ is the oriented graph $\Gr_\tau$ defined by
	\[
	V(\Gr_\tau) = X / H \, , \quad
	E(\Gr_\tau)^+ = \dom(\tau)/\Sigma \, , \quad
	E(\Gr_\tau)^- = \rng(\tau)/\vartheta(\Sigma) \, ,
	\]
	where the structural maps are given by the following  formulas
	\[
	\overline{x\Sigma} = x\tau\vartheta(\Sigma) \, ; \quad
	s(x\Sigma) = x H \, ; \quad
	r(x\Sigma) = x\tau H \, ; \quad
	\]
	\[
	\overline{y\vartheta(\Sigma)} = y\tau^{-1}\Sigma \, ; \quad
	s(y\vartheta(\Sigma)) = y H \, ; \quad
	r(y\vartheta(\Sigma)) = y\tau^{-1} H \, .
	\]
\end{definition}
The Bass-Serre graph will also be denoted by $\mathbf{BS}(X,\tau)$.
\begin{example}
	\begin{enumerate}[label=(\arabic*)]
		\item \label{itex: BS translation} The Bass-Serre graph of the 
		translation pre-action $(\Gamma,t^\rho)$ is the classical Bass-Serre tree 
		$\Tc$ of $\Gamma$.
		\item \label{itex: BS translation pos pre}The Bass-Serre graph of the 
		positive translation pre-action $(T^+,\tau_+)$ is the half-tree of the 
		edge $\Sigma$ in $\Tc$.
		\item \label{itex: BS translation neg pre}The Bass-Serre graph of the 
		negative translation pre-action $(T^-, \tau_-)$ is the half-tree of the 
		edge $\vartheta(\Sigma)$ in $\Tc$.
		\item \label{itex: BS forest}The Bass-Serre graph of a global 
		$\Gamma$-pre-action  is actually a 
		forest if and only if the associated $\Gamma$-action is free.
	\end{enumerate}
\end{example}
Example \ref{itex: BS translation} is obvious. Examples \ref{itex: BS 
	translation pos pre} and \ref{itex: BS translation neg pre}, if not 
obvious yet, will become clear after Remark \ref{Remark Path HNN}. Example 
\ref{itex: BS forest} will be seen in Remark \ref{EqFreeForest}.

Now, let us link the star at a vertex in a Bass-Serre graph $\mathbf{BS}(X,\tau)$ to small normal forms in $\Gamma$. 
Given a vertex in an oriented graph, let us denote by $\st^+(v)$, respectively $\st^-(v)$, the set of positive, respectively negative, edges whose source is $v$, so that we have a partition $\st(v) = \st^+(v) \sqcup \st^-(v)$ of the star at $v$.
Given a point $x\in X$, there are natural (maybe sometimes empty) maps
\[
\begin{array}{cccc}
e^+_{x}: & \{ct : c\in C^+, \, xc\in \dom(\tau)\} & \to & \st^+(x H) \\
& ct & \mapsto &  x c \Sigma   \\
e^-_{x}: & \{ct\inv : c\in C^-, \, xc\in\rng(\tau)\} & \to & \st^-(x H) \\
& ct\inv & \mapsto & x c \vartheta(\Sigma)
\end{array}
\]
and we notice that $e^+_{x}(ct)$ goes from $xH$ to $x c\tau H$, while $e^-_{x}(ct\inv)$ goes from $x H$ to $x c\tau\inv H$.

These maps are surjective, since the orbits $xc \Sigma$ for $c\in C^+$, respectively the orbits $xc \vartheta(\Sigma)$ for $c\in C^-$, cover $xH$. 
Since the action $X \curvearrowleft^{\pi} H$ is free, we have $xH = \bigsqcup_{c\in C^+} xc\Sigma$ and $xH = \bigsqcup_{c\in C^-} xc\vartheta(\Sigma)$, so that $e_x^+,e_x^-$ are in fact bijective.
Then, by merging $e^+_{x}$ and $e^-_{x}$, we get a bijection
\[
e_x : \{ct : c\in C^+, \, xc\in \dom(\tau)\} \sqcup \{ct\inv : c\in C^- , \, xc\in\rng(\tau)\} \to \st(xH) \, .
\]


\subsection{Morphisms and functoriality of Bass-Serre graphs}
We shall now see that there is a functor, that we will call the Bass-Serre functor, from the category of $\Gamma$-pre-actions to the category of graphs, which extends Definition \ref{defBassSerreGraphHNN}. Let us start by turning $\Gamma$-pre-actions into a category.
\begin{definition}
	A \textbf{morphism of pre-actions} from $(X,\tau)$ to $(X',\tau')$ is a 
	$H$-equivariant map $\varphi:X\to X'$, such that for all $x\in \dom\tau$, 
	$\varphi(x\tau) = \varphi(x)\tau'$.
\end{definition}
Note that in particular, $\varphi$ maps $\dom(\tau)$ into $\dom(\tau')$, 
and $\rng (\tau)$ into $\rng (\tau')$. Now, given a morphism of pre-actions 
$\varphi : (X,\tau) \to (X',\tau')$, and denoting by $\Gr_\tau$ and 
$\Gr_{\tau'}$ the corresponding Bass-Serre graphs, let us 
define a map $V(\Gr_\tau) \to V(\Gr_{\tau'})$ by
\[
xH \mapsto \varphi(x) H \, , \text{ for } x\in X \, ,
\]
and a map $E(\Gr_\tau) \to E(\Gr_{\tau'})$ by
\[
x\Sigma \mapsto \varphi(x)\Sigma \, , \text{ for } x\in \dom(\tau)
\quad \text{ and } \quad
y\vartheta(\Sigma) \mapsto \varphi(y) \vartheta(\Sigma) \, , \text{ for } 
y\in \rng(\tau) \, .
\]
It is routine to check that these maps define a morphism of graphs, that we 
denote by $\Gr_{\varphi}$. For instance, the image of $x\Sigma$ is 
$\varphi(x)\Sigma$, the image of $\overline{x\Sigma} = 
x\tau\vartheta(\Sigma)$ is 
$\varphi(x\tau) \vartheta(\Sigma) = \varphi(x)\tau'\vartheta(\Sigma)$, and 
one has 
$\overline{\varphi(x)\Sigma} = \varphi(x)\tau' \vartheta(\Sigma)$ in 
$\Gr_{\tau'}$ as expected.

\begin{lemma}
	The assignments $(X,\tau)\mapsto \Gr_\tau$ and $\varphi \mapsto \Gr_{\varphi}$ define a functor from the category of $\Gamma$-pre-actions to the category of graphs.
\end{lemma}
We will denote this functor by $\mathbf{BS}$ and call it the \textbf{Bass-Serre functor} of $\Gamma$. The morphism $\Gr_\varphi$ will also be denoted by $\mathbf{BS}(\varphi)$.

\begin{proof}
	First, given the identity morphism on a pre-action $(X,\tau)$ it is obvious that the associated morphism of graphs is the identity on $\Gr_\tau$. 
	
	Now, take two morphisms of pre-actions $\varphi: (X,\tau) \to 
	(X',\tau')$ and $\psi: (X',\tau') \to (X'',\tau'')$. It is also clear that 
	the composition of $\Gr_{\varphi}$ followed by $\Gr_{\psi}$, and the 
	morphism $\Gr_{\psi\circ\varphi}$ are both given by the map $V(\Gr_\tau) 
	\to V(\Gr_{\tau''})$ by
	\[
	x H \mapsto \psi\circ\varphi(x) H \, , \text{ for } x\in X \, ,
	\]
	and the map $E(\Gr_\tau) \to E(\Gr_{\tau''})$ by
	\[
	x\Sigma \mapsto \psi\circ\varphi(x) \Sigma \, , \text{ for } x\in 
	\dom(\tau)
	\quad \text{ and } \quad
	y\vartheta(\Sigma) \mapsto \psi\circ\varphi(y)  \vartheta(\Sigma) \, 
	, \text{ for } y\in \rng(\tau) \, .
	\]
	This completes the proof.
\end{proof}

To conclude this section, let us notice a consequence of freeness of the $H$-actions in the definition of $\Gamma$-pre-actions.

\begin{lemma}\label{BSmorphismsLocallyInjectiveHNN}
	Every morphism of the form $\mathbf{BS}(\varphi)$ is locally injective. More precisely, its restriction to the star at a vertex $xH$,  
	is the composition $e_{\varphi(x)}^{} \circ e_{x}^{-1}$, which is an injection into the star at $\varphi(x) H$.
\end{lemma}
\begin{proof}
	Consider a morphism of pre-actions $\varphi: (X,\tau) \to (X',\tau')$, and give names to the actions involved: 
	$X \curvearrowleft^{\pi} H$,
	and $X' \curvearrowleft^{\pi'} H$.
	Let us also recall from Section \ref{SubsectPreActionsAndBSGraphsHNN} 
	that the maps $e_{x}$ and $e_{\varphi(x)}$ are bijective, 
	since these actions are free.
	Now, given $x\in X$ and $e\in \st(xH)$ in $\Gr_\tau$, one must have $e=e_{x}(ct^\varepsilon)$, 
	that is:
	\begin{itemize}
		\item either $e = x c^{\pi} \Sigma$ for a unique $c\in C^+$ satisfying $xc^\pi \in\dom(\tau)$, 
		\item or $e = x c^{\pi} \vartheta(\Sigma)$ for a unique $c\in C^-$ satisfying $xc^\pi \in\rng(\tau)$. 
	\end{itemize}
	Then, in the graph $\Gr_{\tau'}$, one has:
	\begin{itemize}
		\item in the first case, $\varphi(x) c^{\pi'} = \varphi(x c^{\pi}) \in \dom(\tau')$, so that 
		\[
		\Gr_{\varphi}(e) = \varphi(x c^{\pi}) \Sigma = \varphi(x) c^{\pi'}  \Sigma 
		= e_{\varphi(x)}(ct)
		= e_{\varphi(x)}^{} \circ e_{x}^{-1} (e) \, .
		\]
		\item in the second case, $\varphi(x) c^{\pi'} = \varphi(x c^{\pi}) \in \rng(\tau')$, so that 
		\[
		\Gr_{\varphi}(e) = \varphi(x c^{\pi}) \vartheta(\Sigma) = 
		\varphi(x) c^{\pi'} \vartheta(\Sigma) = e_{\varphi(x)}(c t\inv)
		= e_{\varphi(x)}^{} \circ e_{x}^{-1} (e)  \, .
		\]
	\end{itemize} 
	In other words, the restriction of $\Gr_{\varphi}$ to the star at a vertex 
	$xH$ is the composition $e_{\varphi(x)}^{} \circ e_{x}^{-1}$. 
	Furthermore, the latter map is an injection into the star at $\varphi(x) H$. 
\end{proof}


\subsection{Paths in Bass-Serre graphs of global 
	pre-actions}\label{SubsectionPathsBSHNN}
Let us turn to the case of a global pre-action $(X,\tau)$. 
In this case, the bijections $e_x$ defined at the end of Section \ref{SubsectPreActionsAndBSGraphsHNN} become just:
\[
e_x : \{ct : c\in C^+\} \sqcup \{ct\inv : c\in C^-\} 
\longrightarrow  \st^+(xH) \sqcup \st^-(xH) = \st(x H) \, .
\]
Given a point $x\in X$, and an element $\gamma\in \Gamma \setminus H$ with normal form $c_{1}t^{\varepsilon_1} \cdots c_n t^{\varepsilon_n}h_{n+1}$, where $n\geq 1$,
we associate a sequence $(x_0, x_1,\ldots,x_{n})$ in $X$ by setting $x_0 = x$ and $x_i = x_{i-1} c_{i}\tau^{\varepsilon_i}$ for $1\leq i \leq n$,
and notice that $x_n h_{n+1} = x \gamma^{\pi_\tau}$.
Then we associate a sequence $(e_1,...,e_n)$ of edges in the Bass-Serre graph using the bijections $e_x$: for $i=1,\ldots,n$, we set
\[
e_{i} = e_{x_{i-1}}(c_{i}t^{\varepsilon_i}) \, .
\]
Notice that, for any $i=1,\ldots,n-1$, one has $r(e_i) = x_{i-1} c_{i}\tau^{\varepsilon_{i}} H = x_{i} H = s(e_{i+1})$. Hence $(e_1,\ldots,e_n)$ is a path, that we denote by $\pat_x(\gamma)$. 
\begin{remark}
	The vertices $(v_i)_{i=0}^n$ visited by $\pat_x(\gamma)$ are given by $v_i = x_i H$, where:
	$$x_i = x c_{1}\tau^{\varepsilon_1} \cdots c_i \tau^{\varepsilon_i}.$$
	Moreover, defining $\Sigma_1=\Sigma$ and $\Sigma_{-1}=\vartheta(\Sigma)$ one has $e_1=xc_1\Sigma_{\varepsilon_1}$ and, for all $2\leq k\leq n$,
	$$e_k=x c_1t^{\varepsilon_1}\dots c_{k-1}t^{\varepsilon_{k-1}}c_k\Sigma_{\varepsilon_k}.$$
\end{remark}
Now, for $1\leq i \leq n$, let us remark the equivalence
\[
e_{i+1} = \bar e_i \ \Leftrightarrow \
(\varepsilon_{i+1} = - \varepsilon_i \text{ and } c_{i+1} = 1) \, .
\]
Indeed, in case $\varepsilon_i = 1$, one has $e_i = x_{i-1} c_i \Sigma$, therefore
\[
e_{i+1} = \bar e_i \ \Leftrightarrow  \
e_{x_i}(c_{i+1}t^{\varepsilon_{i+1}}) = x_{i-1}c_i\tau \cdot \vartheta(\Sigma) = x_i \vartheta(\Sigma) \ \Leftrightarrow \
(\varepsilon_{i+1} = -1 \text{ and } c_{i+1} = 1)
\]
and the case $\varepsilon_i = -1$ is similar. As we started with a normal form of $\gamma$, we obtain that $\pat_x(\gamma)$ is a reduced path. 
Moreover, given a reduced path $(e'_1,\ldots,e'_n)$ starting at $x H$, one has $\pat_x(\gamma) = (e'_1,\ldots,e'_n)$ if and only if
\[
\text{for all } i=1,\ldots,n, \qquad e_{x_{i-1}}(c_i t^{\varepsilon_i}) = e_i' \, .
\]
Since the maps $e_{x}$ are bijective, there is exactly one element $\gamma \in \Gamma \setminus H$, with normal form $c_1 t^{\varepsilon_1} c_2 t^{\varepsilon_2} \cdots c_n t^{\varepsilon_n}$ such that $\pat_x(\gamma) = (e_1',\ldots,e_n')$.
\begin{remark}
	For any $x\in X$, the map $\pat_x$, from $\Gamma \setminus H$ to the set of reduced paths starting at the vertex $xH$, is surjective. Its restriction to the set of elements with normal form $c_1 t^{\varepsilon_1} c_2 t^{\varepsilon_2} \cdots c_n t^{\varepsilon_n}$ is bijective.
\end{remark}
Let us say that $\gamma \in \Gamma \setminus H$ is a \textbf{path-type element} if its normal form has the form  $t^{\varepsilon_1} c_2 t^{\varepsilon_2} \cdots c_n t^{\varepsilon_n}$ where $n\geq 1$, that is, if $c_1 = 1$ and $h_{n+1} =1$. 
It is said to be \textbf{positive} if $\varepsilon_1 = 1$, and \textbf{negative} if $\varepsilon_1 = -1$.
When $\gamma$ is a positive, respectively negative, path-type element, the first edge of $\pat_x(\gamma)$ is $x\Sigma$, respectively $x \vartheta(\Sigma)$. 
If $n\leq k$, we also say that an element $\tilde \gamma$ with normal form $t^{\varepsilon_1} c_2 t^{\varepsilon_2} \cdots c_k t^{\varepsilon_k}$
is a \textbf{path-type extension} of $\gamma = t^{\varepsilon_1} c_2 t^{\varepsilon_2} \cdots c_n t^{\varepsilon_n}$.
In this case, $\pat_x(\tilde\gamma)$ extends $\pat_x(\gamma)$. 
\begin{remark}\label{Remark Path HNN}
	The map $\pat_x$ induces bijections:
	\begin{itemize}
		\item between the subset of positive path-type elements in $\Gamma$, and the set of reduced paths in $\mathbf{BS}(X,\tau)$ whose first edge is $x\Sigma$;
		\item between the subset of negative path-type elements in $\Gamma$, and the set of reduced paths in $\mathbf{BS}(X,\tau)$ whose first edge is $x\vartheta(\Sigma)$.
	\end{itemize}
	Hence, if $x\Sigma$ (respectively $x\vartheta(\Sigma)$) is a treeing edge then, the images of the first (respectively the second) bijection cover exactly the half-tree of $\Sigma$ (respectively the half-graph of $\vartheta(\Sigma)$) in $\mathbf{BS}(X,\tau)$.
\end{remark}

Let us end this section by linking paths in Bass-Serre trees and Bass-Serre 
graphs so as to understand which edges are treeing edges in the Bass-Serre 
graph.
\begin{remark}\label{PseudoInitialObjectsHNN}
	Consider a global pre-action $(X,\tau)$, and a basepoint $x\in X$. There 
	exists a unique morphism of pre-actions
	\[
	\varphi: (\Gamma,t^\rho) \to (X,\tau)
	\]
	from the translation pre-action, such that $\varphi(1)=x$. In fact, $\varphi$ is the orbital map $\gamma \mapsto x\gamma^{\pi_\tau}$ of the associated $\Gamma$-action. By restriction, one obtains morphisms
	\begin{align*}
	\varphi_{+} & : (T^+,\tau_+) \to (X,\tau) \\
	\varphi_{-} & : (T^-,\tau_-) \to (X,\tau)
	\end{align*}
	from the positive and negative translation pre-actions.
\end{remark}
\begin{lemma}\label{Paths from BS tree to BS graphs ; HNN}
	In the context of the above remark, the Bass-Serre morphism 
	$\mathbf{BS}(\varphi)$, from the Bass-Serre tree $\Tc$ to the 
	Bass-Serre graph $\Gr_\tau$, sends $\pat_{1_\Gamma}^\Tc(\gamma)$ onto 
	$\pat_x^{\Gr_\tau}(\gamma)$.
\end{lemma}
\begin{proof}
	Let us consider $\gamma \in \Gamma \setminus H$, and write its normal form: $\gamma = c_1 t^{\varepsilon_1} \cdots c_n t^{\varepsilon_n} h_{n+1}$. 
	Let us denote by $(e_1, \ldots, e_n)$ the edges of $\pat_{1_\Gamma}^\Tc(\gamma)$,
	and by $(e'_1, \ldots, e'_n)$ the edges of $\pat_x^{\Gr_\tau}(\gamma)$. 
	The auxiliary sequences in $\Gamma$ and $X$ used in the construction of the paths will be denoted by $(\gamma_0,\ldots, \gamma_n)$ and $(x_0,\ldots, x_n)$ respectively.
	
	An easy induction shows that $x_i = \varphi(\gamma_i)$ for all $i = 0,\ldots, n$. 
	Then, we notice that the source of 
	$e_{\gamma_{i-1}}(c_i t^{\varepsilon_i})$
	is $\gamma_{i-1} H$ for all $i= 1,\ldots, n$.
	Thus, using Lemma \ref{BSmorphismsLocallyInjectiveHNN}, 
	we get
	\[
	\Gr_\varphi(e_i) =
	e_{\varphi(\gamma_{i-1})}^{} \circ e_{\gamma_{i-1}}^{-1} 
	\big( e_{\gamma_{i-1}}(c_i t^{\varepsilon_i}) \big) 
	= e_{x_{i-1}}(c_i t^{\varepsilon_i}) = e'_i
	\]
	for all $i= 1,\ldots, n$.
\end{proof}
Therefore, if $x\Sigma$ is a treeing edge then, the image of $\mathbf{BS}(\varphi_{+})$ is the half-tree of $x\Sigma$, while if $x\vartheta(\Sigma)$ is a treeing edge then, the image of $\mathbf{BS}(\varphi_{-})$ is the half-tree of $x \vartheta(\Sigma)$.

\begin{proposition}\label{prop:chara positive treeing edge for HNN}
	Consider a global pre-action $(X,\tau)$, and a basepoint $x \in X$. The 
	following are equivalent:
	\begin{enumerate}[label=(\roman*)]
		\item the morphism of pre-actions $\varphi_{+}: (T^+,\tau_+) 
		\to (X,\tau)$ of Remark \ref{PseudoInitialObjectsHNN} is injective;
		\item the morphism of graphs $\mathbf{BS}(\varphi_{+})$ is 
		injective;
		\item the edge $x\Sigma$ in the Bass-Serre graph 
		$\mathbf{BS}(X,\tau)$ is a treeing edge.
	\end{enumerate}
\end{proposition}
\begin{proof}
	For all $\gamma\in T^+$, recall that $\varphi_{+}(\gamma) = 
	x\gamma^{\pi_{\tau}}$, so that $\mathbf{BS}(\varphi_{+})$ sends 
	vertices $\gamma H$ to $x \gamma^{\pi_{\tau}} H$. At the level of 
	positive edges, it sends $\gamma\Sigma$ to $x \gamma^{\pi_{\tau}} 
	\Sigma$. Fixing $\gamma$, we get $\varphi_{+} (\gamma h) = 
	x\gamma^{\pi_{\tau}}h^{\pi}$ for $h\in H$; since 
	$X\curvearrowleft^{\pi} H$ is free, $\varphi_+$ realizes a bijection 
	between $\gamma H$ and $x \gamma^{\pi_{\tau}} H$, 
	and also a bijection between $\gamma\Sigma$ and $x \gamma^{\pi_{\tau}} 
	\Sigma$.
	Consequently, $\varphi_{+}$ is injective if an only if $\gamma H 
	\mapsto x \gamma^{\pi_{\tau}} H$ and $\gamma\Sigma \mapsto x 
	\gamma^{\pi_{\tau}} \Sigma$ are both injective. This proves that (i) 
	and (ii) are equivalent. Note that (iii) implies (ii) is obvious since, 
	when $x\Sigma$ is a treeing edge $\mathbf{BS}(\varphi_{+})$ is locally 
	injective from the half-tree of $\Sigma$ to the half-tree of $x\Sigma$, 
	hence $\mathbf{BS}(\varphi_{+})$ is injective. Finally assume (ii) and 
	let $\omega$ be a reduced path starting by the edge $x\Sigma$. By 
	Remark \ref{Remark Path HNN} there exists a positive 
	path type element $\gamma\in \Gamma^+$ 
	such that $\omega=\pat_x(\gamma)$. By Lemma 
	\ref{Paths from BS tree to BS graphs ; HNN}, $\omega$ is the image by 
	$\mathbf{BS}(\varphi_{+})$ of $\pat_1^\Tc(\gamma)$. Since 
	$\mathbf{BS}(\varphi_{+})$ is supposed to by injective and since the 
	last vertex of $\pat_1^\Tc(\gamma)$ is not $H$, we deduce that the last 
	vertex of $\omega$ is not $xH$. Hence, $x\Sigma$ is a treeing edge by 
	Lemma~\ref{Lemma Treeing Edge}.\end{proof}

By a very similar argument, wet get also the following result.
\begin{proposition}\label{prop:chara negative treeing edge for HNN}
	Consider a global pre-action $(X,\tau)$, and a basepoint $x \in X$. The 
	following are equivalent:
	\begin{enumerate}
		\item[(i)] the morphism of pre-actions $\varphi_{-}: (T^-,\tau_-) 
		\to (X,\tau)$ of Remark \ref{PseudoInitialObjectsHNN} is injective;
		\item[(ii)] the morphism of graphs $\mathbf{BS}(\varphi_{-})$ is 
		injective;
		\item[(iii)] the edge $x\vartheta(\Sigma)$ in the Bass-Serre graph 
		$\mathbf{BS}(X,\tau)$ is a treeing edge.
	\end{enumerate}
\end{proposition}
\begin{remark}\label{EqFreeForest}
	Putting the two previous propositions together, one can show that given 
	a $\Gamma$-action where $H$ is acting freely, the Bass-Serre graph of 
	the associated pre-action is a forest if and only if the 
	$\Gamma$-action is free. 
\end{remark}

\subsection{The free globalization of a pre-action of an HNN extension}

Say that a pre-action is \textbf{transitive} when its Bass-Serre graph is connected. 
Note that a global pre-action $(X,\tau)$ is transitive if and only if the 
associated $\Gamma$-action is transitive. 
We will show that every transitive pre-action has a canonical extension to a transitive action, which is \textbf{as free as possible}. 
The construction is better described in terms of Bass-Serre graph: we are going to attach as many treeing edges as possible to it.

\begin{theorem}\label{thm: HNN free globalization}
	Every transitive $\Gamma$-pre-action $(X,\tau)$ on a non-empty set $X$ 
	admits a transitive and global extension $(\tilde X, \tilde\tau)$ which 
	satisfies the following universal property: 
	given any transitive and global extension $(Y,\tau')$ of $(X,\tau)$, 
	there is a unique morphism of pre-actions $\varphi:(\tilde 
	X,\tilde\tau) \to (Y,\tau')$ such that 
	\[
	\varphi_{\restriction X}=\mathrm{id}_{X} \, .
	\]
	Moreover, all the edges from the Bass-Serre graph $\mathbf{BS}(X,\tau)$ to its complement in $\mathbf{BS}(\tilde X,\tilde\tau)$ are treeing edges.
\end{theorem}
In terms of $\Gamma$-actions, the extension $(\tilde X, \tilde\tau)$ of the theorem corresponds to an action $\tilde X \curvearrowleft\Gamma$ such that, 
given any action $Y \curvearrowleft^\alpha \Gamma$ satisfying $X\subseteq Y$ as $H$-sets and $yt^\alpha = y\tau$ for all $y\in \dom(\tau)$, 
there exists a unique $\Gamma$-equivariant map $\varphi:\tilde X \to Y$ extending $\id_X$.

\begin{proof}
	
	We will obtain the Bass-Serre graph of the pre-action $(\tilde X, \tilde \tau)$ by adding only treeing edges to the Bass-Serre graph of the pre-action.
	
	First we enumerate the $\Sigma$-orbits which do not belong to the domain of $\tau$ as $(x_i \Sigma)_{i\in I}$. 
	Then, we take disjoint copies of $(T^+,\tau_+)$, for $i\in I$, also disjoint from $X$, which we denote as $(T^+_i,\tau_i)$.  
	Similarly, we enumerate the $\vartheta(\Sigma)$-orbits which do not belong to the range of $\tau$ as $(y_j \vartheta(\Sigma))_{j\in J}$. 
	We then take disjoint copies of $(T^-,\tau_-)$, for $j\in J$, also disjoint from $X$, which we denote as $(T^-_j,\tau_j)$.  
	Now, our extension $(\tilde X, \tilde \tau)$ is obtained as follows. We set
	\[
	\tilde X = {\left( X \sqcup \bigsqcup_{i\in I} T^+_i \sqcup \bigsqcup_{j\in J} T^-_j \right)}\bigg/_{\sim}
	\]
	where $\sim$ identifies the element $x_i h \in X$ with $h\in H \subset T^+_i$, for each $i\in I$ and $h\in H$, and the element $x_j h \in X$ with $h\in H \subset T^-_j$, for each $j\in J$ and $h\in H$. 
	Since the identifications just glue some orbits pointwise and respect 
	the $H$-actions, $\tilde X$ is endowed with a free $H$-action. Then, we 
	set
	\[
	\tilde \tau = \tau \sqcup \bigsqcup_{i\in I} \tau_i \sqcup \bigsqcup_{j \in J} \tau_j \, ,
	\]
	which is possible since the domain of $\tau_i$, for $i\in I$, intersects other components in $\tilde X$ only in the orbit $x_i\Sigma$, 
	the range of $\tau_i$, for $i\in I$, does not intersect other components in $\tilde X$, and the situation is analogue for $\tau_j$ with $j\in J$. 
	We have got a pre-action $(\tilde X,\tilde \tau)$.
	
	This pre-action is transitive. Indeed, all pre-actions $(X,\tau)$, $(T^+_i,\tau_i)$ and $(T^-_j,\tau_j)$ are, and the identifications make connections between all these components in the Bass-Serre graph $\mathbf{BS}(\tilde X,\tilde \tau)$. 
	
	The pre-action is also global. Indeed, every $\Sigma$-orbit, 
	respectively $\vartheta(\Sigma)$-orbit, in $T^+_i$, which is not in the 
	domain, respectively the range, of $\tau_i$, has been identified with 
	an orbit in $X$, and the situation is similar for $T^-_j$. We conclude 
	by noting that all $\Sigma$-orbits and $\vartheta(\Sigma)$-orbits 
	in $X$ are now in the domain and in the range of $\tilde \tau$.
	
	Moreover, the (oriented) edges from the Bass-Serre graph $\mathbf{BS}(X,\tau)$ to its complement in $\mathbf{BS}(\tilde X,\tilde \tau)$ are exactly the edges $x_i\Sigma$ for $i\in I$, and the edges $x_j \vartheta(\Sigma)$ for $j\in J$. For each $i\in I$, the morphism of pre-actions $\varphi_{+}: (T^+,\tau_+) \to (\tilde X,\tilde \tau)$ of Remark \ref{PseudoInitialObjectsHNN}, with basepoint $x_i\in \tilde X$, is injective since it realizes an isomorphism onto $(T^+_i,\tau_i)$, hence $x_i\Sigma$ is a treeing edge by Proposition \ref{prop:chara positive treeing edge for HNN}. One proves similarly that the edges $x_j \vartheta(\Sigma)$ are treeing edges using Proposition \ref{prop:chara negative treeing edge for HNN}. 
	
	It now remains to prove the universal property. To do so, take any 
	transitive and global extension $(Y,\tau')$ of $(X,\tau)$. The unique 
	morphism of pre-actions $\varphi$ from $(\tilde X,\tilde \tau)$ to 
	$(Y,\tau')$ such that $\varphi_{\restriction X}=\id_{X}$ is obtained by 
	taking the union of the morphisms $\varphi_{i}:(T^+_i,\tau_i) \to 
	(Y,\tau')$ and $\varphi_{j}:(T^-_j,\tau_j) \to (Y,\tau')$, coming from 
	Remark \ref{PseudoInitialObjectsHNN} with respect to basepoints $x_i$ 
	or $x_j$, with $\id_{X}$ (all these morphisms are unique).
\end{proof}

It is straightforward to deduce from the universal property above that the 
global pre-action we just built is unique up to isomorphism. We thus call 
it \textbf{the} \textbf{free globalization} of the pre-action $(X,\tau)$. 

\begin{example}\label{ex: free globa for pos and neg preactions}
	The free globalizations of the positive and negative translation 
	pre-action are equal to the right $\Gamma$-action on itself by 
	translation. Indeed, this is true of any transitive pre-action obtained 
	as a restriction of the (global) $\Gamma$-pre-action on itself by right 
	translation, since the latter is universal among transitive 
	$\Gamma$-actions.
\end{example}

Let us furthermore observe that we can always build this pre-action on a fixed set $\bar X$ containing $X$, provided it contains infinitely many free $H$-orbits.

\begin{theorem}\label{free globalization in prescribed set for HNN}
	
	Let $\bar X$ be a countable set equipped with a free $H$-action, suppose $X\subseteq \bar X$ is $H$-invariant and $\bar X \setminus X$ contains infinitely many $H$-orbits. Suppose further that $\tau$ is a partial bijection on $X$ such that $(X,\tau)$ is a transitive pre-action of $\Gamma$. Then there is a permutation $\bar\tau$ of $\bar X$ such that  $(\bar X,\bar \tau)$ is (isomorphic to) the free globalization of $(X,\tau)$.
\end{theorem}
\begin{proof}
	Let $(\tilde X,\tilde \tau)$ be the free globalization of $(X,\tau)$. The fact that $\tilde X\setminus X$ contains infinitely many $H$-orbits and is countable implies that there exist a $H$-equivariant bijection $\varphi: \tilde X \to \bar X$ whose restriction to $X$ is the identity. Then, one can push forward the permutation $\tilde \tau$, to obtain a permutation $\bar \tau: \bar X \to \bar X$ defined by
	\[
	\varphi(x) \bar\tau := \varphi(x \tilde\tau) \quad \text{ for all } x\in \tilde X \, ,
	\]
	which extends $\tau$.
	Now, $\varphi$ is an isomorphism of pre-actions between $(\tilde X,\tilde \tau)$ and $(\bar X,\bar \tau)$.
\end{proof}

\subsection{Connection with partial actions and strong 
	faithfulness}\label{sec: partial actions HNN}

\begin{definition}
	Given a transitive pre-action of the HNN extension $\Gamma$ on 
	$(X,\tau)$, let us denote by $(\tilde X, \tilde\tau)$ its free globalization. 
	The \textbf{partial action associated} to $(X,\tau)$ is the restriction 
	to $X$ of the action $\tilde X 
	\curvearrowleft^{\pi_{\tilde\tau}}\Gamma$. We denote it by 
	$\alpha_\tau$.
\end{definition}
In order to have shorter statements in what follows, we will also call the action $\tilde X \curvearrowleft^{\pi_{\tilde\tau}}\Gamma$ ``free globalization'' of $(X,\tau)$.
\begin{remark} One could also construct the partial action directly as 	
	follows. Let us denote by $\pi$ the $H$-action on $X$.
	Given $\gamma \in \Gamma$ with normal form $c_1 t^{\varepsilon_1} 
	\cdots c_n t^{\varepsilon_n} h_{n+1}$, define the partial bijection 
	$\gamma^{\alpha_\tau}$ by
	$$
	\gamma^{\alpha_\tau} := c_1^\pi \tau^{\varepsilon_1} \cdots c_n^\pi 
	\tau^{\varepsilon_n} h_{n+1}^\pi,
	$$
	where we compose partial bijections as described in Section \ref{sec: 
		prelim partial action}.
	The relation $\gamma_1^{\alpha_\tau} \gamma_2^{\alpha_\tau} \subseteq 
	(\gamma_1 \gamma_2)^{\alpha_\tau}$ follows from the fact that in order 
	to obtain the normal form of $\gamma_1\gamma_2$ from the concatenation 
	of the normal forms of $\gamma_1$ and $\gamma_2$, one only needs to 
	iterate the following three types of operations:
	\begin{enumerate}[label=(\arabic*)]
		\item  replacing a subword $h_1\sigma th_2$ by $h_1t 
		\vartheta(\sigma)h_2$;
		\item replacing a subword $h_1\vartheta(\sigma) t^{-1} h_2$ by
		$t^{-1} \sigma h_2$;
		\item deleting the occurrences of $t t^{-1}$ or $t^{-1} t$.
	\end{enumerate}
	By the definition of a pre-action, types (1) and (2)
	do not affect the partial bijection that we get in the end, while type (3) can only produce extensions (note that $\tau\tau\inv$ and $\tau\inv \tau$ are \textit{restrictions} of the identity on $X$).
\end{remark}

We can now connect the \emph{free} globalization that we constructed to the 
\emph{universal} globalization of Kellendonk-Lawson that was presented in 
Theorem \ref{thm: universal globalization}.	
\begin{proposition}
	The free 
	globalization of a transitive pre-action $(X,\tau)$ is equal to the 
	universal
	globalization of the partial action $\alpha_\tau$.
\end{proposition}	
\begin{proof}
	There is a  unique $\Gamma$-equivariant map $g$ from the universal 
	globalization 
	$Z$ constructed by Kellendonk-Lawson to the 
	free globalization $\tilde X$ because of its universal property. 
	Moreover, in the 
	free globalization, we have that $H$ acts freely, so it follows that 
	$H$ is also acting freely on the universal globalization. 
	
	It is then straightforward to check that since the pre-action 
	$\tau$ is transitive, the associated partial action $\alpha_\tau$ is 
	transitive, i.e. for every $x,y\in X$ there is $\gamma\in\Gamma$ such 
	that 
	$y=x\gamma^{\alpha_\tau}$. Since the $\Gamma$-closure of $X$ inside 
	$Z$ satisfies the same universal property as $Z$, we conclude that the 
	universal globalization is transitive.
	
	We can thus apply the universal property of the free globalization from 
	Theorem \ref{thm: HNN free globalization} so as to obtain a unique 
	$\Gamma$-equivariant map $f:\tilde X\to Z$ which restricts to the 
	identity on $X$. Recalling that $g: Z\to \tilde X$ is the unique 
	$\Gamma$-equivariant map provided by Kellendonk and Lawson's theorem, 
	we conclude by uniqueness that both $g\circ f$ and $f\circ g$ are 
	identity maps, which concludes the proof.
\end{proof}

An important property of the partial action associated to  a pre-action 
$(X,\tau)$ is 
that it is \emph{contained} in the partial action associated to any of the 
globalizations of $(X,\tau)$, in the following sense.

\begin{definition}
	Let $X\lact^{\alpha}\Gamma$ and $X\lact^{\beta}\Gamma$ be two partial 
	actions. We say that $\alpha$ is \textbf{contained} in $\beta$ when for 
	every $\gamma\in\Gamma$, we have $\alpha(\gamma)\subseteq\beta(\gamma)$.
\end{definition}

\begin{proposition}\label{prop: HNN containement partial action}
	Let $(X,\tau)$ be a transitive $\Gamma$-pre-action, let 
	$X\lact^{\alpha_\tau}\Gamma$ be the associated partial action. Then for 
	every 
	global extension $(Y, \sigma)$ of the pre-action, we have 
	that the restriction to $X$ of the action $Y\curvearrowleft^{\pi_\sigma}\Gamma$ 
	contains $\alpha_\tau$.
\end{proposition}
\begin{proof}
	Suppose $(Y,\sigma)$ is a global extension of the pre-action 
	$(X,\tau)$. Since $(X,\tau)$ is transitive, 
	up to shrinking $(Y,\sigma)$ we may assume it it 
	transitive. We can now apply Theorem~\ref{thm: HNN free globalization}:
	the universal property gives a $\Gamma$-equivariant map  $\rho: \tilde X\to Y$ with respect to the actions $\tilde \pi := \pi_{\tilde \tau}$ and $\pi_\sigma$.
	Then, for every $x\in \tilde X$, and every $\gamma\in\Gamma$, we have 
	$\rho(x\gamma^{\tilde\pi})=\rho(x)\gamma^{\pi_\sigma}$. In particular, for every 
	$x\in X$ such that $x\gamma^{\pi_\sigma} \in X$, we obtain 
	$x \gamma^{\alpha_\tau} = x\gamma^{\tilde 
		\pi}=x\gamma^{\pi_\sigma}$, which yields directly the desired result.
\end{proof}

Let us now show that the free globalization of any non-global transitive 
pre-action is highly faithful. We start with a lemma.

\begin{lemma}
	The partial actions associated to the positive and negative translation 
	pre-action are strongly faithful.
\end{lemma}
\begin{proof}
	By Example \ref{ex: free globa for pos and neg preactions}, the partial action associated to the positive translation pre-action 
	is the restriction to $T^+ = \Gamma^+\sqcup H$ of the action $\Gamma \curvearrowleft \Gamma$ by right translations. 
	Thus, it suffices to show that for every 
	$F\Subset\Gamma\setminus\{1\}$, there exists $x\in\Gamma^+\sqcup H$ such 
	that for all $\gamma\in F$, we have $x\gamma\in \Gamma^+\sqcup H$ and 
	$x\gamma\neq x$. The latter assertion is always true, and the former 
	holds if we take $x\in\Gamma^+$ whose normal form is longer than the 
	normal form of all the elements of $F$. We conclude that the partial 
	action associated to the positive translation 
	pre-action is strongly faithful.
	
	A similar argument shows that the partial action associated to the
	negative translation pre-action is strongly faithful.
\end{proof}
\begin{proposition}\label{Free globalization faithful HNN}
	The free globalization of any non-global transitive pre-action is 
	highly 
	faithful.
\end{proposition}
\begin{proof}
	Let $\pi$ be the action associated to the free globalization of a non-global transitive pre-action $(X,\tau)$.
	By Corollary \ref{EquivStrongAndHighFaithfulness}, it suffices to prove 
	that $\pi$ is strongly faithful.
	But since the pre-action $(X,\tau)$ is not global, its free 
	globalization contains a copy of either the positive or the negative 
	translation pre-action. The partial actions of the latter being 
	strongly faithful by the previous lemma, we conclude using Proposition 
	\ref{prop: HNN containement partial action}
	that $\pi$ itself is strongly faithful because it contains a strongly 
	faithful partial action.
\end{proof}

\begin{remark}
	More generally, it follows from  Propositions 
	\ref{prop:chara positive treeing edge for HNN}
	and \ref{prop:chara negative treeing edge for HNN}
	that every $\Gamma$-action whose Bass-Serre graph contains a 
	treeing edge must be highly faithful.
\end{remark}

\section{High transitivity for HNN extensions}\label{High transitivity for HNN extensions}

As in Section \ref{SectFreeGlobalizationHNN}, 
we fix an HNN extension $\Gamma=\HNN(H,\Sigma,\vartheta)$, a set of representatives $C^+$ of left $\Sigma$-cosets in $H$, 
and a set of representatives $C^-$  of left $\vartheta(\Sigma)$-cosets in $H$, 
which both contain $1$, 
so that normal forms of elements of $\Gamma$ are well-defined. 
From now on, \textbf{we assume that the HNN extension $\Gamma$ is non-ascending}
and that \textbf{the $\Gamma$-action on the boundary of its Bass-Serre tree is topologically free}, 
since these assumptions will become essential.
We still denote by $\Gamma^+$, respectively $\Gamma^-$, the set of elements  whose normal form first letter is $t$, respectively $t\inv$, so that we have $\Gamma=H\sqcup C^+\Gamma^+ \sqcup C^-\Gamma^-$.

\subsection{Using the free globalization towards high transitivity}

This section is devoted to a key result which we will use towards proving high transitivity for HNN extensions. It will allow us to extend any given transitive pre-action which is not an action to one which sends one fixed tuple to another fixed tuple. 

\begin{proposition}\label{prop: key prop for HNN}
	Let $\bar X$ be a countable set, let $\bar X \curvearrowleft^\pi H$ be a free action,
	with infinitely many $H$-orbits,
	let $X$ be a finite union of $H$-orbits in $\bar X$, and suppose	
	that $(X,\tau)$ is a transitive non-global pre-action. 
	For any pairwise distinct points $x_1,...,x_k,y_1,...,y_k\in \bar X$,
	there exists a transitive and global extension $(\bar X,\tilde \tau)$ 
	of $(X,\tau)$ such that:
	\begin{enumerate}[label=(\arabic*)]
		\item the action $\bar X \curvearrowleft^{\pi_{\tilde\tau}}\Gamma$ is (transitive and) highly faithful;
		\item there is an element $\gamma\in\Gamma$ satisfying $x_i \gamma^{\pi_{\tilde \tau}}=y_i$ for all $i$.
	\end{enumerate}
\end{proposition}

\begin{proof}
	
	The set $\{x_1,...,x_k,y_1,...,y_k\}$ will be denoted by $F$.  
	First, by Theorem \ref{free globalization in prescribed set for HNN}, we find a permutation $\bar\tau \in S(\bar X)$ such that $(\bar X,\bar \tau)$ is the free globalization of $(X,\tau)$. 
	Given $x\in \bar X$, and a path-type element $\gamma$, we will denote by $\Hc_{x}(\gamma)$ the half-graph of the last edge of $\pat_x(\gamma)$ in the Bass-Serre graph of $(\bar X,\bar \tau)$.
	\begin{claim}
		There exists a path-type element $\gamma$ in $\Gamma \setminus H$ such that for every $x\in F$, the last edge of $\pat_x(\gamma)$ is a treeing edge (that is, $\Hc_x(\gamma)$ is a tree).
	\end{claim}
	\begin{cproof}
		Using the correspondence between path-type elements and reduced paths established in Section \ref{SubsectionPathsBSHNN}, 
		it follows from Lemma \ref{lem: extend path with treeing edge} that for every $x\in\bar X$, and every path-type element $\gamma$,
		there is a path-type extension $\gamma'$ of $\gamma$ such that the last edge of $\pat_x(\gamma')$ is a treeing edge. 
		Now, it suffices to start with any path-type element $\gamma_0$, to extend it to a path-type element $\gamma_1$ such that the last edge of $\pat_{x_1}(\gamma_1)$ is a treeing edge, 
		then to extend $\gamma_1$ to a path-type element $\gamma_2$ such that the last edge of $\pat_{y_1}(\gamma_2)$ is a treeing edge 
		(note that $\pat_{x_1}(\gamma_2)$ also ends with a treeing edge since it extends $\pat_{x_1}(\gamma_1)$), \ldots, 
		and to iterate this extension procedure until we reach an element 
		$\gamma_{2k}$ such that all last edges of $\pat_{x}(\gamma_{2k})$, 
		for $x\in F$, are treeing edges. 
	\end{cproof}
	Let us fix some element $c\in C^+ \setminus\{1\}$ (here we use that $\Sigma \neq H$). Then, any path-type element $\gamma$ admits $\gamma ct$ as a path-type extension.
	
	\begin{claim}
		There exists a path-type element $\gamma$ in $\Gamma \setminus H$ such that for every $x\in F$, the last edge  of $\pat_x(\gamma)$ is a treeing edge, and the half-trees $\Hc_{x}(\gamma)$, for $x\in F$, are pairwise disjoint subgraphs, and disjoint from $\mathbf{BS}(X,\tau)$.
	\end{claim}
	
	\begin{cproof}
		We start with a path-type element $\gamma$ such that for every $x\in F$, the last edge of $\pat_x(\gamma)$ is a treeing edge. 
		Since $X$ is a finite union of $H$-orbits, $\mathbf{BS}(X,\tau)$ has finitely many vertices. 
		Hence, by extending further the path-type element $\gamma$, we may and will 
		assume that, for every $x\in F$, the half-tree $\Hc_{x}(\gamma)$ does not intersect $\mathbf{BS}(X,\tau)$.
		
		We now notice that, given $x,y\in F$, if the half-trees $\Hc_{x}(\gamma)$ and $\Hc_{y}(\gamma)$ are disjoint, 
		then so are the half-trees $\Hc_{x}(\gamma')$ and $\Hc_{y}(\gamma')$ for every path-type extension $\gamma'$ of $\gamma$. 
		Hence, it suffices to prove that, for any $x,y\in F$, with $x\neq y$ and such that $\Hc_{x}(\gamma)$ and $\Hc_{y}(\gamma)$ intersect, 
		there exists a path-type extension $\gamma'$ of $\gamma$ such that $\Hc_{x}(\gamma')$ and $\Hc_{y}(\gamma')$ are disjoint. 
		Indeed, an easy induction gives then an extension $\gamma^{(n)}$ such that the half-trees $\Hc_{x}(\gamma^{(n)})$, 
		for $x\in F$, are pairwise disjoint.
		
		Take now $x,y\in F$ with $x\neq y$ and such that $\Hc_{x}(\gamma)$ and $\Hc_{y}(\gamma)$ intersect. 
		These half-trees have to be nested. 
		Indeed, if they aren't, $\Hc_{x}(\gamma)$ contains the antipode of the last edge of $\pat_y(\gamma)$, 
		hence contains $\mathbf{BS}(X,\tau)$, which is impossible.
		Without loss of generality, we assume $\Hc_{x}(\gamma) \subseteq \Hc_{y}(\gamma)$. 
		We now distinguish two cases.
		\begin{itemize}
			\item If $\Hc_{x}(\gamma) \subsetneq \Hc_{y}(\gamma)$, there is a path-type extension $\gamma''$ of $\gamma$, with $\gamma'' \neq \gamma$,
			such that $\pat_x(\gamma)$ and $\pat_y(\gamma'')$ have the same last edge. 
			Since the HNN extension $\Gamma$ is non-ascending, there is another path-type extension $\gamma'$ of $\gamma$ with the same length as $\gamma''$.
			Now, $\pat_y(\gamma')$ and $\pat_y(\gamma'')$ are distinct reduced paths extending $\pat_{y}(\gamma)$,
			hence $\Hc_{y}(\gamma')$ and $\Hc_{y}(\gamma'')= \Hc_{x}(\gamma)$ have to be disjoint.
			Since $\Hc_x(\gamma')\subset\Hc_x(\gamma)$, the half-trees $\Hc_{y}(\gamma')$ and $\Hc_{x}(\gamma')$ are disjoint.
			
			\item If $\Hc_{x}(\gamma) = \Hc_{y}(\gamma)$, then $\pat_x(\gamma)$ and $\pat_y(\gamma)$ have the same terminal edge. 
			Let us assume that $\pat_x(\gamma ct)$ and $\pat_y(\gamma ct)$ have the same terminal edge, since otherwise we are done with $\gamma' = \gamma ct$. 
			This edge is $e := x'c \Sigma = y'c \Sigma$, 
			where $x' = x \gamma^{\pi_{\bar\tau}}$ and $y' = y \gamma^{\pi_{\bar\tau}}$.
			Consequently, one has $y'c = x'c\sigma$ for some $\sigma\in \Sigma\setminus\{1\}$. 
			Consider now the morphism from the  translation pre-action
			\[  
			\varphi : (\Gamma, t^\rho) \to (\bar X,\bar\tau) \, ,
			\quad \text{ with basepoint } x'c \, 
			\]
			coming from Remark \ref{PseudoInitialObjectsHNN},
			and note that the left translation $\psi_\sigma: \gamma^*\mapsto 
			\sigma\gamma^*$ 
			defines an automorphism of $(\Gamma,t^\rho)$. Using Lemma 
			\ref{Paths from BS tree to BS graphs ; HNN}, one sees that, for 
			any $\gamma^*\in \Gamma^+$:
			\begin{itemize}
				\item $\mathbf{BS}(\psi_\sigma)$ maps $\pat_{1}^\Tc(\gamma^*)$ onto $\pat^\Tc_{\sigma}(\gamma^*)$ in the Bass-Serre tree $\Tc$, and these paths both start by the edge $\Sigma$ hence, they are both in the half-tree of $\Sigma$.
				\item $\mathbf{BS}(\varphi)$ maps $\pat_{1}^\Tc(\gamma^*)$ onto $\pat_{x'c}(\gamma^*)$;
				and $\pat^\Tc_{\sigma}(\gamma^*)$ onto $\pat_{y'c}(\gamma^*)$.
			\end{itemize}
			
			Since the left $\Gamma$-action on the boundary $\partial\Tc$ of its Bass-Serre tree is topologically free, 
			the identity is the only element of $\Gamma$ fixing the half-tree of $\Sigma$ in $\Tc$ pointwise. 
			Hence, there exists a path $\omega$ in this half-tree whose first edge is $\Sigma$, and such that $\omega$ and $\sigma\cdot \omega$ have distinct ranges. 
			Then, by Remark~\ref{Remark Path HNN}, there exists a path-type element $\gamma^+ \in \Gamma^+$ such that $\pat_{1}^\Tc(\gamma^+) = \omega$.
			We have $\pat_{\sigma}^\Tc(\gamma^+) = \mathbf{BS}(\psi_\sigma)(\omega) = \sigma\cdot \omega$, so that
			$\pat_{1}^\Tc(\gamma^+)$ and $\pat_{\sigma}^\Tc(\gamma^+)$
			have distinct ranges. 
			
			Since the edge $e$ is a treeing edge, the restriction $\varphi_+$ of $\varphi$ to the positive translation pre-action $(T^+,\tau_+)$ 
			is injective by Proposition \ref{prop:chara positive treeing edge for HNN}, and so is $\mathbf{BS}(\varphi_+)$.
			Consequently, $\pat_{x'c}(\gamma^+)$ and $\pat_{y'c}(\gamma^+)$ diverge at some point in the half-tree of $e$ in $\mathbf{BS}(\bar X, \bar \tau)$.
			
			Finally, for any path-type element $\gamma^*$ in $\Gamma^+$, by construction,
			$\pat_{x}(\gamma c \gamma^*)$ is the concatenation of $\pat_x(\gamma)$ and $ \pat_{x'c}(\gamma^*)$, and $\pat_{y}(\gamma c \gamma^*)$ is the concatenation of $\pat_y(\gamma)$ and $ \pat_{y'c}(\gamma^*)$. Hence $\pat_{x}(\gamma c \gamma^+)$  and $\pat_{y}(\gamma c \gamma^+)$ diverge at some point in the half-tree of $e$. Setting $\gamma' = \gamma c \gamma^+$, we obtain that $\Hc_x(\gamma')$ and $\Hc_y(\gamma')$ are disjoint.
		\end{itemize} 
		We are done in both cases.
	\end{cproof}
	
	We then modify the bijection $\bar \tau$ to get the pre-action $(\bar X,\tilde \tau)$ we are looking for. First, given an element $\gamma$ as in the previous claim, we consider, for each $z\in F$, the morphism of pre-actions from the negative translation pre-action coming from Remark \ref{PseudoInitialObjectsHNN}:
	\[  
	\psi_{z} : (T^-,\tau_-) \to (\bar X,\bar\tau)
	\quad \text{  with basepoint } z'c \, ,
	\]
	
	\noindent where $z':= z\gamma^{\pi_{\bar\tau}} $. 
	Note that the image of this morphism corresponds to the half-graph opposite to the half-tree $\Hc_{z}(\gamma ct)$. 
	Then, we define $X'=\bigcap_{z\in F}\rng(\psi_{z})\subset\bar X$, and take the restriction $(X',\tau')$ of $(\bar X,\bar\tau)$. 
	Informally speaking, we erase $\bar \tau$ on the $\Sigma$-orbits corresponding to edges in the half-trees $\Hc_{z}(\gamma ct)$ for $z\in F$.
	Note that this leaves infinitely many $H$-orbits in $\bar X$ outside $\dom(\tau')$ and $\rng(\tau')$, and the pre-action $(X',\tau')$ is transitive.
	Now, we extend $\tau'$. Pick some orbits $z_1 H, \ldots, z_k H$ in $\bar X \setminus (\dom(\tau') \cup \rng(\tau'))$, add them to $X'$, 
	take some $c^- \in C^- \setminus \{1\}$ (here we use that $\vartheta(\Sigma) \neq H$), and set
	\[
	x_i'c\sigma\tau' := z_i \vartheta(\sigma)
	\quad \text{ and } \quad 
	y_i'c \sigma \tau' := z_i c^- \vartheta(\sigma)                                                                                                                                                                                     
	\]                                                                                                                                                                  
	for $i=1,\ldots,k$ and $\sigma \in \Sigma$. 
	This is possible since the $\vartheta(\Sigma)$-orbits at points $z_i$ 
	and $z_i c^-$ are pairwise disjoint (we use again the freeness of the 
	$H$-action), and since the $\Sigma$-orbits at $x_i' c$ and $y_i' c$ were not initially 
	in the domain of $\tau'$.
	Note that, after this extension, the pre-action $(X',\tau')$ is still transitive. Then we apply Theorem \ref{free globalization in prescribed set for HNN} to get an extension $\tilde \tau: \bar X\to \bar X$ of $\tau'$ such that $(\bar X, \tilde\tau)$ is the free globalization of $(X',\tau')$. 
	A computation shows then that $x_i(\gamma ct c^- (\gamma 
	ct)\inv)^{\pi_{\tilde\tau}} = y_i$ for all $i=1,\ldots,k$. Finally, the 
	action $\pi_{\tilde\tau}$ is highly faithful by Proposition \ref{Free 
		globalization faithful HNN}.
\end{proof}

\subsection{High transitivity for HNN extensions}

From now on, we fix a free action $X\curvearrowleft^{\pi} H$ with 
infinitely many orbits. We then consider the space $\A$ of 
$\Gamma$-actions on 
$X$ which extend $\pi$, which can be written as
\[
\A = \{\tau \in S(X): \, \sigma^\pi \tau = \tau 
\vartheta(\sigma)^\pi \text{ for all } \sigma\in \Sigma\} \, .
\]
In other words, $\A$ is the set of permutations $\tau$ of $X$ such 
that $(X,\tau)$ is a global pre-action of $\Gamma$. The set $\A$ 
is clearly a 
closed subset of $S(X)$ for the topology of pointwise convergence, hence a 
Polish 	space.  
Recall that the action associated to a permutation $\tau\in\A$  is 
denoted by $\pi_\tau$.

\begin{definition}
	Let us set
	\begin{align*}
	\TA &= \{\tau\in \A 
	\colon \pi_\tau \text{ is transitive} \}; \\
	\HFA &= \{\tau\in \A \colon
	\pi_\tau
	\text{ is highly faithful} \}; \\
	\HTA &= \{\tau\in \A \colon
	\pi_\tau 
	\text{ is highly transitive} \}.
	\end{align*}
\end{definition}
The subset $\TA$ is not closed for the 
topology of pointwise convergence.
However, we have the following result.
\begin{lemma}\label{PolishSpaceHNN}
	The set $\TA$ is $G_\delta$ in $\A$, hence it is a 
	Polish space. Moreover, it is non-empty.
\end{lemma}
\begin{proof}
	Since  $X\curvearrowleft^{\pi}H$ has infinitely many orbits, there is an $H$-equivariant bijection $\varphi: \Gamma \to X$. 
	It then suffices to push-forward the translation pre-action by 
	$\varphi$ to get an element of $\TA$ (its Bass-Serre graph 
	will be isomorphic to the classical Bass-Serre tree and $\pi_\tau$ will 
	be conjugated to the translation action $\Gamma\curvearrowleft\Gamma$). 
	
	To show that $\TA$ is a $G_\delta$ subset,  we write 
	$\TA=\bigcap_{x,y\in X}O_{x,y}$, where for $x,y\in X$, 
	\[
	O_{x,y}=\{\tau\in\A\,:\,\text{ there exists 
	}\gamma\in\Gamma\text{ such that }x\gamma^{\pi_\tau}=y\}.
	\]
	The latter sets are clearly open in $\A$ for all $x,y\in X$, thus 
	finishing the proof.
\end{proof}

We now show that our HNN extension $\Gamma$ has a highly 
transitive highly faithful action, thus proving Theorem \ref{ThmMainHNN}.
\begin{theorem}\label{ThmGroupsHTHNNBaire}
	The set $\HTA\cap\HFA$ is  dense $G_\delta$ in
	$\TA$. 
	In particular, $\Gamma$ admits actions which are both highly transitive and highly faithful.
\end{theorem}

\begin{proof}
	For $k\geq 1$ and $x_1,\dots x_k,y_1,\dots,y_k\in X$ pairwise distinct, 
	consider the open subsets
	\[
	V_{x_1,\dots,x_k,y_1,\dots,y_k}=\{\tau\in\TA\colon
	\exists\gamma\in\Gamma\,,\,x_i\gamma^{\pi_\tau}=y_i\text{ for all }1\leq i\leq k\}
	\]
	Their intersection is the set $\HTA$ by Lemma \ref{HT and disjoint 
		supports}, so $\HTA$ is $G_\delta$ in $\TA$.  
	Similarly, we can write the set of strongly faithful actions as the 
	intersection over all finite subsets $F\Subset\Gamma\setminus\{1\}$ of the open sets
	\[
	W_F = \{ \tau\in\TA \colon
	\exists x\in X\,,\, x f^{\pi_\tau} \neq x\text{ for all } f\in F \}
	\]
	Now, since  strong faithfulness is equivalent to high faithfulness by 
	Corollary 
	\ref{EquivStrongAndHighFaithfulness},  the set $\HFA$ is $G_\delta$ in 
	$\TA$.

	To conclude, it suffices to show that each set 
	$(V_{x_1,\dots,x_k,y_1,\dots,y_k}) \cap \HFA$ is dense in $\TA$, since 
	this immediately implies that each open set 
	$(V_{x_1,\dots,x_k,y_1,\dots,y_k}) \cap W_F$ is dense in $\TA$. 
	To do this, let $\tau\in\TA$ and let $F'$ be a finite subset of $X$.
	Consider a finite connected subgraph $\Gr$ of $\mathbf{BS}(X,\tau)$ 
	containing the edges $z\Sigma$ for $z\in F'$, 
	and denote by $\tau_0$ the restriction of $\tau$ to the union of the 
	$\Sigma$-orbits in $X$ corresponding to the edges of $\Gr$.
	
	Then, apply Proposition~\ref{prop: key prop for HNN} to the transitive 
	pre-action $(\dom(\tau_0)\cdot H \cup \rng(\tau_0)\cdot H,\tau_0)$, whose 
	Bass-Serre graph is $\Gr$, to get an extension $\tau'$ such that $\tau'\in 
	V_{x_1,\dots,x_k,y_1,\dots,y_k}\cap\HFA$. Moreover, since $F'\subset 
	\dom(\tau_0)$, it follows that $\tau$ and $\tau'$ coincide on $F'$.
\end{proof}

\begin{remark} \label{rmk: HNN HT without Baire} We give below a direct 
	proof that $\HFA \cap \HTA$ is dense in $\TA$, without relying on Baire's 
	Theorem. 
	\begin{proof}
		Start with an element $\tau_0 \in\TA$, consider the 
		transitive and global pre-action $(X,\tau_0)$, and fix a finite subset 
		$F_0\Subset X$. What we have to prove is that there exists $\tau' \in 
		\HFA \cap \HTA$ such that the restrictions of $\tau_0$ and $\tau'$ 
		on $F_0$ coincide. 
		
		Let us now take an enumeration $(g_n)_{n\geq 0}$ of $\Gamma \setminus \{1\}$, 
		and an enumeration $(k_n,\bar x_n, \bar y_n)_{n\geq 0}$ of the set of 
		triples $(k,\bar x, \bar y)$, where $k$ is a positive integer and 
		$\bar x =(x_1,\ldots,x_k) , \bar y = (y_1,\ldots,y_k)$ are $k$-tuples 
		of elements of $X$ such that $x_1,\ldots,x_k ,y_1,\ldots,y_k$ are 
		pairwise distinct.
		Starting with $\tau_0$ and $F_0$, we construct inductively a sequence 
		$(\tau_n)_{n\geq 0}$ in $\TA$, and an increasing sequence 
		$(F_n)_{n\geq 0}$ of finite subsets of $X$, as follows. 
		\begin{enumerate}
			\item Starting with $F_n \Subset X$ and $\tau_n \in \TA$, we set $F$ to be the union of $F_n$ and the coordinates of 
			$\bar x_n$ and $\bar y_n$, and consider the smallest connected 
			subgraph $\Gr$ of $\mathbf{BS}(X,\tau_n)$ which contains all edges 
			$z\Sigma$ and $z\vartheta(\Sigma)$ for $z\in F$. (Notice that $\mathbf{BS}(X,\tau_n)$ is 
			connected, since $\tau_n \in \TA$). 
			
			\item We take the restriction $\tau$ of $\tau_n$ on the union of the 
			$\Sigma$-orbits in $X$ corresponding to edges in $\Gr$, 
			and get a transitive pre-action $(\dom(\tau)\cdot 
			H\cup\rng(\tau)\cdot H,\tau)$, 
			whose Bass-Serre graph is $\Gr$, and such that $\tau^{\pm 1}$ coincides with 
			$\tau_n^{\pm 1}$ on $F$.
			
			\item By Proposition \ref{prop: key prop for HNN}, we get an 
			extension $\tau_{n+1}$ of $\tau$, which lies in $\TA \cap 
			\HFA$, and an element $\gamma_n \in \Gamma$ such that $\bar x_n 
			\gamma_n^{\pi_{\tau_{n+1}}}= \bar y_n$. 
			Moreover, $\tau_{n+1}^{\pm 1}$ coincides with $\tau_n^{\pm 1}$ on $F_n$ by 
			construction.
			Let also $v_n$ be an element of $X$ such that $v_n 
			g_k^{\pi_{\tau_{n+1}}} \neq v_n$ for all $k=0,\ldots,n$ (which exists 
			since $\pi_{\tau_{n+1}}$ is highly faithful).
			
			\item We take a finite subset $F_{n+1} \Subset X$ which contains $F$, 
			and all elements $z\in X$ such that $z\Sigma$, or its antipode, is in 
			$\pat_{v_n}(g_k)$ for some $k\leq n$, or in $\pat_u(\gamma_n)$ for 
			some coordinate $u$ of $\bar x_n$. Now, for any $\tau^*$ coinciding 
			with $\tau_{n+1}$ on $F_{n+1}$, one has $\bar x_n 
			\gamma_n^{\pi_{\tau^*}}= \bar y_n$ and $v_n g_k^{\pi_{\tau^*}} \neq 
			v_n$ for all $k=0,\ldots,n$.
		\end{enumerate}
		Theses sequences satisfy $\bar x_m \gamma_m^{\pi_{\tau_n}}= \bar y_m$ 
		and $v_m g_k^{\pi_{\tau_n}}\neq v_m$ for all $0\leq k \leq m < n$. 
		Moreover, the subsets $F_n$ exhaust $X$, and $\tau_n^{\pm 1}$ coincides with 
		$\tau_m^{\pm 1}$ on $F_m$ for all $n>m$. Consequently the sequence $(\tau_n)$ 
		converges to a bijection $\tau'$ and the action $\pi_{\tau'}$ is highly 
		transitive by Lemma \ref{HT and disjoint supports}. 
		Notice finally that $\pi_{\tau'}$ is also strongly faithful, since it 
		satisfies $v_m g_k^{\pi_{\tau'}}\neq v_m$ for all $k\leq m$. Thus, 
		$\pi_{\tau'}$ is highly faithful by Corollary 
		\ref{EquivStrongAndHighFaithfulness}.
	\end{proof}
\end{remark}

\section{Free globalization for pre-actions of amalgams}\label{SectFreeGlobalizationAmalgams}

We now turn to the case of amalgams, where the notion of pre-action is a 
bit less intuitive than for HNN extensions since it will involve two sets, 
reflecting the fact that the corresponding graph of groups has two 
vertices. The results and proofs are very similar to those we have proved 
for HNN 
extensions in Sections~\ref{SectFreeGlobalizationHNN} and~\ref{High 
	transitivity for HNN extensions}, but for the convenience of the reader we 
will give full proofs.

For this section, as in Section \ref{PrelimAmalg}, let us fix an amalgam 
$\Gamma=\Gamma_1 *_\Sigma \Gamma_2$, and sets of representatives $C_j$ of 
left $\Sigma_j$-cosets in $\Gamma_j$ such that $1\in C_j$, for $j=1,2$, so 
that normal forms of elements of $\Gamma$ are well-defined. Let us also 
denote by $\NN j$ the set of elements of $\Gamma$ whose normal form begins 
(from the left) with an element of $C_j\setminus\{1\}$, for $j=1,2$, so 
that we have 
$\Gamma = \Sigma \sqcup \NN 1 \sqcup \NN 2$.

\subsection{Actions and pre-actions of amalgams, and Bass-Serre graphs}\label{SubsectPreActionsAndBSGraphs}

Given two actions on infinite countable sets $X_1 \curvearrowleft^{\pi_1} \Gamma_1$ and $X_2 \curvearrowleft^{\pi_2} \Gamma_2$, and a bijection $\tau:X_1\to X_2$ such that $\sigma^{\pi_1}\tau = \tau\vartheta(\sigma)^{\pi_2}$ for all $\sigma\in \Sigma_1$,
there exists a unique action $X_1\curvearrowleft^{\pi_{1,\tau}} \Gamma$ 
such that $g^{\pi_{1,\tau}} = g^{\pi_1}$ for all $g\in \Gamma_1$, and 
$h^{\pi_{1,\tau}} = \tau h^{\pi_2} \tau^{-1}$ for all $h\in \Gamma_2$.
Similarly, there exists a unique action $X_2\curvearrowleft^{\pi_{2,\tau}} 
\Gamma$ such that $h^{\pi_{2,\tau}} = h^{\pi_2}$ for all $h\in \Gamma_2$, 
and $g^{\pi_{2,\tau}} = \tau\inv g^{\pi_1} \tau$ for all $g\in \Gamma_1$. 
Of course, these actions are conjugate: one has $\gamma^{\pi_{2,\tau}} = 
\tau\gamma^{\pi_{1,\tau}}\tau\inv$ for every $\gamma\in \Gamma$.
Turning back to the general case, if the actions $\pi_1,\pi_2$ are free, 
we obtain an example of the following situation.

\begin{definition}
	A \textbf{pre-action} of the amalgam $\Gamma$ is a triple $(X_1,X_2,\tau)$ where $X_1,X_2$ are infinite countable sets endowed with free actions $X_1 \curvearrowleft^{\pi_1} \Gamma_1$ and $X_2 \curvearrowleft^{\pi_2} \Gamma_2$, and
	\[
	\tau:\dom(\tau) \to \rng(\tau) 
	\]
	is a partial bijection such that $\dom(\tau) \subseteq X_1$, $\rng(\tau)\subseteq X_2$, and $\sigma^{\pi_1}\tau = \tau\vartheta(\sigma)^{\pi_2}$ for all $\sigma\in \Sigma_1$.
\end{definition}

The relations $\sigma^{\pi_1}\tau = \tau\vartheta(\sigma)^{\pi_2}$ are equalities between partial bijections. In particular $\sigma^{\pi_1}\tau$ and $\tau\vartheta(\sigma)^{\pi_2}$ must have the same domain and the same range. As a consequence, for any pre-action $(X_1,X_2,\tau)$, the domain of $\tau$ is necessarily $\Sigma_1$-invariant, its range is necessarily $\Sigma_2$-invariant, and $\tau$ sends $\Sigma_1$-orbits onto $\Sigma_2$-orbits. 

A pre-action $(X_1,X_2,\tau)$ is called \textbf{global} if $\tau$ is a global bijection 
between $X_1$ and $X_2$. In this case there are associated actions 
$X_1\curvearrowleft^{\pi_{1,\tau}} \Gamma$ and 
$X_2\curvearrowleft^{\pi_{2,\tau}} \Gamma$ as above.
\begin{example}
	If $X \curvearrowleft^\pi \Gamma$ is an action, where $\Gamma_1$ and $\Gamma_2$ are acting freely, then denoting by $X 
	\curvearrowleft^{\pi_1} \Gamma_1$ and $X \curvearrowleft^{\pi_2} \Gamma_2$ 
	its restrictions, one obtains a global pre-action $(X,X,\id_X)$, where the 
	first copy of $X$ is endowed with $\pi_1$ and the second with $\pi_2$. The 
	actions $X_1\curvearrowleft^{\pi_{1,\tau}} \Gamma$ and 
	$X_2\curvearrowleft^{\pi_{2,\tau}} \Gamma$ both coincide with $X 
	\curvearrowleft^\pi \Gamma$ in this case. In particular, the right 
	translation action $\Gamma\curvearrowleft\Gamma$ gives rise to a 
	pre-action $(\Gamma,\Gamma,\id)$, called the (right) \textbf{translation 
		pre-action}.
\end{example}

As seen before, actions of $\Gamma$ (such that the factors act freely) 
correspond to pre-actions with a global bijection. Another source of 
examples of pre-actions is the following.

\begin{definition}
	Given a pre-action $(X_1,X_2,\tau)$, and infinite $\Gamma_j$-invariant subsets $Y_j\subseteq X_j$, the \textbf{restriction} of $(X_1,X_2,\tau)$ to $(Y_1,Y_2)$ is the pre-action $(Y_1,Y_2,\tau')$, where $Y_j$ is endowed with the restrictions of $\pi_j$, and the partial bijection is $\tau'=\tau_{\restriction Y_1\cap Y_2\tau\inv}$.
	An \textbf{extension} of $(X_1,X_2,\tau)$ is a pre-action $(\tilde X_1, \tilde X_2, \tilde \tau)$ whose restriction to $(X_1,X_2)$ is $(X_1,X_2,\tau)$.
\end{definition}

In the following important example of restrictions, for $j\in\{1,2\}$ we 
denote by 
$\NN j$ the set of elements of $\Gamma$ whose normal form begins with an 
element of $C_j$, so that we have $\Gamma = \Sigma \sqcup \NN 1 \sqcup 
\NN 2$.

\begin{example}
	The set $\NN 2$ is $\Gamma_1$-invariant by right multiplication, and 
	$\NN 1$ 
	is $\Gamma_2$-invariant by right multiplication. Thus, by taking 
	complements, $\Sigma \sqcup \NN 1$ is $\Gamma_1$-invariant, and $\Sigma 
	\sqcup \NN 2$ is $\Gamma_2$-invariant. The translation pre-action 
	$(\Gamma,\Gamma,\id)$ admits the restrictions $(\Gamma_1 \sqcup 
	\NN 2,\Gamma_2 \cup \NN 2,\tau_+)$, and \mbox{$(\Gamma_1 \cup \NN 1, \Gamma_2 
	\sqcup 
	\NN 1, \tau_-)$,} which we call the \textbf{positive translation 
		pre-action} 
	and the \textbf{negative translation pre-action} respectively. Notice that 
	$\tau_+ = \id_{\restriction \Sigma \sqcup \NN 2}$ and $\tau_- = 
	\id_{\restriction \Sigma \sqcup \NN 1}$.
\end{example}
Let us now associate a graph to any  $\Gamma$-pre-action $(X_1,X_2,\tau)$ as follows. Informally speaking, we start with a graph whose vertices are of two kinds: the $\Sigma_1$-orbits in $X_1$, and the $\Sigma_2$-orbits in $X_2$. Then we put an edge from $x\Sigma_1$ to $y\Sigma_2$ when $(x\Sigma_1)\tau = y\Sigma_2$, and finally we identify all the $\Sigma_j$-orbits that are in a same $\Gamma_j$-orbit, for $j=1,2$. We may, and will, identify the groups $\Gamma_j,\Sigma_j$ with their images by $\pi_j$, since the actions $\pi_1,\pi_2$ are free, hence faithful. Consequently, we don't write superscripts $\pi_1,\pi_2$ from now, as soon as there is no risk of confusion. 
\begin{definition}\label{defBassSerreGraph}
	The \textbf{Bass-Serre graph} of $(X_1,X_2,\tau)$ is the oriented graph $\Gr_\tau$ defined by
	\[
	V(\Gr_\tau) = X_1 / \Gamma_1 \sqcup X_2 / \Gamma_2 \, , \quad
	E(\Gr_\tau)^+ = \dom(\tau)/\Sigma_1 \, , \quad
	E(\Gr_\tau)^- = \rng(\tau)/\Sigma_2 \, ,
	\]
	where the structural maps are given by the following  formulas, for $x\in\dom(\tau)$ and $y\in\rng(\tau)$,
	\[
	\overline{x\Sigma_1} = x\tau\Sigma_2 \, ; \quad
	s(x\Sigma_1) = x \Gamma_1 \, ; \quad
	r(x\Sigma_1) = x\tau \Gamma_2 \, ; \quad
	\]
	\[
	\overline{y\Sigma_2} = y\tau^{-1}\Sigma_1 \, ; \quad
	s(y\Sigma_2) = y\Gamma_2 \, ; \quad
	r(y\Sigma_2) = y\tau^{-1}\Gamma_1 \, .
	\]
\end{definition}
The Bass-Serre graph will also be denoted by $\mathbf{BS}(X_1,X_2,\tau)$.
\begin{example}
	\begin{enumerate}[label=(\arabic*)]
		\item The Bass-Serre graph of the translation pre-action $(\Gamma,\Gamma,\id)$ is the classical Bass-Serre tree $\Tc$ of $\Gamma$.
		\item The Bass-Serre graph of the positive translation pre-action 
		$(\Gamma_1 \sqcup \NN 2,\Gamma_2 \cup \NN 2,\tau_+)$ is the half-tree of 
		the edge $\Sigma_1$ in $\Tc$.
		\item The Bass-Serre graph of the negative translation pre-action 
		$(\Gamma_1 \cup \NN 1, \Gamma_2 \sqcup \NN 1, \tau_-)$ is the half-tree 
		of the edge $\Sigma_2$ in $\Tc$.
	\end{enumerate}
\end{example}

Given points $x\in X_1 , y \in X_2$, there are natural (maybe sometimes empty)
maps from some coset representatives to the 
stars at 
$x\Gamma_1$ and $y\Gamma_2$:
\[
\begin{array}{cccc}
e_{1,x}: & \{c\in C_1: \, xc\in \dom(\tau) \} & \longrightarrow & \st(x\Gamma_1) \\
& c & \longmapsto &  x c \Sigma_1 \\
e_{2,y}: & \{ c\in C_2: \, yc \in\rng(\tau) \} & \longrightarrow &
\st(y\Gamma_2) \\
& c & \longmapsto & y c \Sigma_2
\end{array}
\]
These maps are are surjective, since, for $j=1,2$, the orbits $xc\Sigma_j$, for $c\in C_j$, cover $x \Gamma_j$.
Moreover, since the actions $X_j \curvearrowleft^{\pi_j} \Gamma_j$ are free, we have $x\Gamma_j = \bigsqcup_{c\in C_j} xc\Sigma_j$, so that these maps are in fact bijective.

If $x\in X_1\cap X_2$, then by merging $e_{1,x}$ and $e_{2,x}$, we get a bijection
\[
e_x : \{c\in C_1: \, xc\in \dom(\tau) \} 
\sqcup \{ c\in C_2: \, xc \in\rng(\tau) \} 
\to \st(x\Gamma_1) \sqcup \st(x\Gamma_2) \, .
\]
We also set $e_x = e_{1,x}$ when $x\in X_1 \setminus X_2$, and $e_x = e_{2,x}$ when $x\in X_2 \setminus X_1$.

\subsection{Morphisms and functoriality of Bass-Serre graphs}
We shall now see that there is a functor, that we will call the Bass-Serre functor, from the category of $\Gamma$-pre-actions to the category of graphs, which extends Definition \ref{defBassSerreGraph}. Let us start by turning $\Gamma$-pre-actions into a category.
\begin{definition}
	A \textbf{morphism of pre-actions} from $(X_1,X_2,\tau)$ to 
	$(X'_1,X'_2,\tau')$ is a couple $(\varphi_1,\varphi_2)$, where 
	$\varphi_j:X_j\to X'_j$ is a $\Gamma_j$-equivariant map for $j=1,2$, and 
	for 
	all $x\in\dom\tau$, $\varphi_2(x\tau) = \varphi_1(x)\tau'$.
\end{definition}
Again, we have in particular that $\varphi_1$ maps $\dom(\tau)$ into 
$\dom(\tau')$ and $\varphi_2$ maps $\rng (\tau)$ into $\rng (\tau')$.

Now, given a morphism of pre-actions $(\varphi_1,\varphi_2) : (X_1,X_2,\tau) \to (X'_1,X'_2,\tau')$, and denoting by $\Gr_\tau$ and $\Gr_{\tau'}$ the corresponding Bass-Serre graphs, let us 
define a map $V(\Gr_\tau) \to V(\Gr_{\tau'})$ by
\[
x\Gamma_1 \mapsto \varphi_1(x)\Gamma_1 \, , \text{ for } x\in X_1
\quad \text{ and } \quad
y\Gamma_2 \mapsto \varphi_2(y)\Gamma_2 \, , \text{ for } y\in X_2 \, ,
\]
and a map $E(\Gr_\tau) \to E(\Gr_{\tau'})$ by
\[
x\Sigma_1 \mapsto \varphi_1(x)\Sigma_1 \, , \text{ for } x\in \dom(\tau)
\quad \text{ and } \quad
y\Sigma_2 \mapsto \varphi_2(y)\Sigma_2 \, , \text{ for } y\in \rng(\tau) 
\, .
\]
It is routine to check that these maps define a morphism of graphs, that we 
denote by $\Gr_{(\varphi_1,\varphi_2)}$. For instance, the image of 
$x\Sigma_1$ is $\varphi_1(x)\Sigma_1$, the image of $\overline{x\Sigma_1} = 
x\tau\Sigma_2$ is 
$\varphi_2(x\tau)\Sigma_2 = \varphi_1(x)\tau'\Sigma_2$, and one has 
$\overline{\varphi_1(x)\Sigma_1} = \varphi_1(x)\tau'\Sigma_2$ in 
$\Gr_{\tau'}$.

\begin{lemma}
	The assignments $(X_1,X_2,\tau)\mapsto \Gr_\tau$ and $(\varphi_1,\varphi_2) \mapsto \Gr_{(\varphi_1,\varphi_2)}$ define a functor from the category of $\Gamma$-pre-actions to the category of graphs.
\end{lemma}
We will denote this functor by $\mathbf{BS}$ and call it the \textbf{Bass-Serre functor} of $\Gamma$. The morphism $\Gr_{(\varphi_1,\varphi_2)}$ will also be denoted by $\mathbf{BS}(\varphi_1,\varphi_2)$.

\begin{proof}
	First, given the identity morphism on a pre-action $(X_1,X_2,\tau)$ it is obvious that the associated morphism of graphs id the identity on $\Gr_\tau$. 
	
	Now, let us consider two morphisms of pre-actions $(\varphi_1,\varphi_2): 
	(X_1,X_2,\tau) \to (X'_1,X'_2,\tau')$ and $(\psi_1,\psi_2): 
	(X'_1,X'_2,\tau') \to (X''_1,X''_2,\tau'')$. It is also clear that the 
	composition of $\Gr_{(\varphi_1,\varphi_2)}$ followed by 
	$\Gr_{(\psi_1,\psi_2)}$, and the morphism 
	$\Gr_{(\psi_1\circ\varphi_1,\psi_2\circ\varphi_2)}$ are both given by the 
	map 
	$V(\Gr_\tau) \to V(\Gr_{\tau'})$ by
	\[
	x\Gamma_1 \mapsto \psi_1\circ\varphi_1(x)\Gamma_1 \, , \text{ for } x\in 
	X_1
	\quad \text{ and } \quad
	y\Gamma_2 \mapsto \psi_2\circ\varphi_2(y)\Gamma_2 \, , \text{ for } y\in 
	X_2 \, ,
	\]
	and the map $E(\Gr_\tau) \to E(\Gr_{\tau'})$ by
	\[
	x\Sigma_1 \mapsto \psi_1\circ\varphi_1(x)\Sigma_1 \, , \text{ for } x\in 
	\dom(\tau)
	\quad \text{ and } \quad
	y\Sigma_2 \mapsto \psi_2\circ\varphi_2(y)\Sigma_2 \, , \text{ for } y\in 
	\rng(\tau) \, .
	\]
	This completes the proof.
\end{proof}

Let us notice a consequence of freeness of the $\Gamma_j$-actions in the 
definition of $\Gamma$-pre-actions, analogous to Lemma 
\ref{BSmorphismsLocallyInjectiveHNN}.

\begin{lemma}\label{BSmorphismsLocallyInjective}
	Every morphism of the form $\mathbf{BS}(\varphi_1,\varphi_2) = \Gr_{(\varphi_1,\varphi_2)}$ is locally injective. More precisely, the restriction of $\mathbf{BS}(\varphi_1,\varphi_2)$ to the star at a vertex $x\Gamma_1$, respectively $y\Gamma_2$, 
	is the composition $e_{1,\varphi_1(x)}^{} \circ e_{1,x}^{-1}$, respectively $e_{2,\varphi_2(y)}^{} \circ e_{2,y}^{-1}$, which is an injection into the star at $\varphi_1(x)\Gamma_1$, respectively $\varphi_2(y)\Gamma_2$.
\end{lemma}
\begin{proof}
	Consider a morphism of pre-actions $(\varphi_1,\varphi_2): (X_1,X_2,\tau) \to (X'_1,X'_2,\tau')$, and give names to the $\Gamma_1$-actions involved: 
	$X_1 \curvearrowleft^{\pi_1} \Gamma_1$, 
	and $X'_1 \curvearrowleft^{\pi'_1} \Gamma_1$.
	Let us also recall from Section \ref{SubsectPreActionsAndBSGraphs} that maps of the form $e_{1,x}$ and $e_{2,y}$ are bijective, since the $\Gamma_j$-actions are free.
	Now, given $x\in X_1$ and $e\in \st(x\Gamma_1)$ in $\Gr_\tau$, one has 
	$e=e_{1,x}(c) = x c^{\pi_1} \Sigma_1$ for a unique $c\in C_1$ satisfying $xc^{\pi_1} \in \dom(\tau)$. Then,
	one has $\varphi(x) c^{\pi_1'} = \varphi(x c^{\pi_1}) \in \dom(\tau')$,
	so that in $\Gr_{\tau'}$: 
	\[
	\Gr_{(\varphi_1,\varphi_2)}(e) = \varphi_1(x c^{\pi_1}) \Sigma_1 = 
	\varphi_1(x) c^{\pi'_1} \Sigma_1 = e_{1,x\varphi_1}(c) \, . 
	\]	 
	In other words, the restriction of $\Gr_{(\varphi_1,\varphi_2)}$ to the star at $x\Gamma_1$ is the composition $e_{1,\varphi_1(x)}^{} \circ e_{1,x}^{-1}$. 
	Furthermore, this map is an injection into the star at $\varphi_1(x)\Gamma_1$.
	
	Similarly, one can prove that the restriction of $\Gr_{(\varphi_1,\varphi_2)}$ to the star at a vertex $y\Gamma_2$ is the composition $e_{2,\varphi_2(y)}^{} \circ e_{2,y}^{-1}$, which is an injection into the star at $\varphi_2(y)\Gamma_2$.
\end{proof}

\subsection{Paths in Bass-Serre graphs of global 
	pre-actions}\label{SubsectionPathsAmalgams}

Let us turn to the case of a global pre-action $(X_1,X_2,\tau)$.  
In this case, the bijections $e_{1,x}$ and $e_{2,y}$, defined at the end of Section~\ref{SubsectPreActionsAndBSGraphs}, become just
\[
e_{1,x}: \, C_1  \longrightarrow  \st(x\Gamma_1)
\quad \text{ and } \quad
e_{2,y}: \, C_2  \longrightarrow \st(y\Gamma_2) \, .
\]
Given a 
point $x\in X_1$ and an element $\gamma\in \NN 2$ with normal form 
$\gamma=c_1 \cdots c_n \sigma$ where $n\geq 1$  and 
$c_1\in C_2\setminus\{1\}$, we 
associate a sequence $(x_0, x_1,\ldots,x_{n+1})$ in $X_1\cup X_2$ and a 
sequence $(e_0,e_1,...,e_n)$ of edges in the Bass-Serre graph as follows. 
We set $x_0 = x$, $c_0=1\in C_1$, and then inductively for $i=0,\ldots,n$:

\begin{itemize}
	\item for $i$ such that $c_i\in C_1$, set $e_i = e_{1,x_{i}}(c_i)$, and 
	$x_{i+1} = x_{i} c_i \tau$;
	\item for $i$ such that $c_i\in C_2$, set $e_i = e_{2,x_{i}}(c_i)$, and 
	$x_{i+1} = x_{i} c_i \tau\inv$.
\end{itemize}
Notice that, for any $i=0,\ldots,n-1$, if $c_i\in C_1$ (or equivalently if 
$i$ is even), one has 
$r(e_{i}) = x_{i}c_i\tau \Gamma_2 = x_{i+1} \Gamma_2 = s(e_{i+1})$ , 
and similarly  if $c_i\in C_2$ we have $r(e_i) = s(e_{i+1})$. 
Hence $(e_0,\ldots,e_n)$ is a path, 
that we denote by $\pat_{1,x}(\gamma)$. Note that this path begins by the 
edge 
$e_0 = x\Sigma_1$. 

Let us check that $\pat_{1,x}(\gamma)$ is a reduced 
path. For $0\leq i \leq n-1$ and $c_i\in C_1$, we have
\[
e_{i+1} = \bar e_i \Leftrightarrow x_{i+1} c_{i+1} \Sigma_2 = 
\overline{x_i c_i \Sigma_1} \Leftrightarrow x_i c_i \tau c_{i+1} \Sigma_2 
= x_i c_i \tau \Sigma_2
\Leftrightarrow c_{i+1} = 1
\]
since $X_2 \curvearrowleft^{\pi_2} \Gamma_2$ is free. Since $c_1 \cdots c_n 
\sigma$ is the
normal form of $\gamma$, we cannot have $c_{i+1} = 1$, so
$\pat_{1,x}(\gamma)$ is reduced.

Finally, since the maps 
$e_{1,x}$ and $e_{2,x}$ are bijective, given a reduced path 
$(e_0,\ldots,e_n)$ beginning by $x\Sigma_1$, 
there is exactly one normal form $c_1\cdots c_n$ with $c_1\in C_2\setminus\{1\}$ 
such that $\pat_{1,x}(c_1\cdots c_n) = (e_0,\ldots,e_n)$. The following 
remark is now clear.

\begin{remark}\label{Remark Path Amalgams} For any $x\in X_1$, the map 
	$\pat_{1,x}$  is a surjection from $\NN 2$ to the set of reduced 
	paths starting by the edge $x\Sigma_1$. It becomes a bijection if 
	we restrict it to the subset of elements $\gamma\in \NN 2$ whose normal 
	form 
	is $c_1\cdots c_n$ with $c_1\in C_2\setminus\{1\}$. So if $x\Sigma_1$ is a 
	treeing edge, then the images $\pat_{1,x}(\gamma)$, for $\gamma\in \NN 2$, 
	cover 
	exactly the half-tree of $x\Sigma_1$ in $\Gr_\tau$.
	
\end{remark}
We now give a definition of path-type elements which is analogous to the 
one for HNN extensions, except that we only want to consider paths which 
end in $X_1$. An element 
$\gamma\in \NN 2$ with normal form $\gamma=c_1 \cdots c_n$ such 
that $c_1\in C_2\setminus\{1\}$ and $n\geq 1$ is \emph{odd} will be called a 
\textbf{path 
	type element}  of $\NN 2$. Note that the corresponding path then has 
\emph{even} 
length. If $\gamma'=c_1\cdots c_k$, for some $k\leq 
n$, 
is a path type element in $\NN 2$, then $\gamma$ is called a \textbf{path 
	type extension} of $\gamma'$.\\

Similarly, given a point $x\in X_2$, we can associate to every element 
$\gamma\in \NN 1$ with normal form $\gamma=c_1 \cdots c_n \sigma$, where 
$n\geq 1$ and $c_1\in C_1\setminus\{1\}$, a reduced path $\pat_{2,x}(\gamma) := 
(e_0,e_1,...,e_n)$, as follows. We set $x_0 = x$, $c_0=1\in C_2$, and then 
inductively for $i=0,\ldots,n$:
\begin{itemize}
	\item for $i$ such that $c_i\in C_1$, set $e_i = e_{1,x_{i}}(c_i)$, and 
	$x_{i+1} = x_{i} c_i \tau$;
	\item for $i$ such that $c_i\in C_2$, set $e_i = e_{2,x_{i}}(c_i)$, and 
	$x_{i+1} = x_{i} c_i \tau\inv$.
\end{itemize}
This defines a surjective map 
$\pat_{2,x}$ from $\NN 1$ to the set of reduced paths starting by 
the edge $x\Sigma_2$. Hence, if $x\Sigma_2$ is a treeing edge then, the 
images $\pat_{2,x}(\gamma)$, for $\gamma\in \NN 1$, cover exactly the 
half-tree 
of $x\Sigma_2$ in $\Gr_\tau$.
Moreover, the map $\pat_{2,x}$ becomes a bijection if we restrict it to the 
set 
of elements of with normal form $\gamma=c_1 \cdots c_n$ where $n\geq 1$ and 
$c_1\in C_1\setminus\{1\}$, and such elements will be called \textbf{path type 
	elements} of $\NN 1$ when moreover $n$ is odd. As before, there is a notion 
of path type extension for path type elements in $\NN 1$.
\begin{remark}\label{RemarkPathAmalgams}  Let $i,j\in\{1,2\}$ with $i\neq 
	j$ and $x\in X_j$, $\gamma \in \NN {i}$ with normal form $c_1\cdots 
	c_n\sigma$.
	\begin{enumerate}
		\item By construction, $\pat_{j,x}(\gamma) = \pat_{j,x}(c_1\cdots c_n)$ 
		and, for 
		every $1\leq k \leq n$, the path $\pat_{j,x}(c_1\cdots c_n)$ is an 
		extension 
		of $\pat_{j,x}(c_1\cdots c_k)$.
		
		\item The source of $\pat_{j,x}(\gamma)$ is $s(x\Sigma_j) = x\Gamma_j$.
		
		\item If $\gamma=c_1\cdots c_n$ is a path type element, then the range of 
		$\pat_{j,x}(\gamma)$ is $x\gamma^{\pi_{j,\tau}}\Gamma_j$.
		\item If for some $1\leq k\leq n$ the last edge of 
		$\pat_{j,x}(c_1\cdots c_k)$ is a treeing edge, then 
		for all $k\leq l\leq n$ the last edge of 
		$\pat_{j,x}(c_1\cdots c_l)$ is also a treeing edge.
	\end{enumerate}
\end{remark}

Let us end this section by establishing a link between paths in Bass-Serre 
trees and 
Bass-Serre graphs.
\begin{remark}\label{PseudoInitialObjects}
	Consider a global pre-action $(X_1,X_2,\tau)$, and basepoints 
	$x_1\in X_1$ and $x_2\in X_2$ such that $x_2 = x_1 \tau$. There exists 
	a unique morphism of pre-actions
	\[
	(\varphi_1,\varphi_2): (\Gamma,\Gamma,\id) \to (X_1,X_2,\tau)
	\]
	from the translation pre-action, such that $\varphi_j(1)=x_j$ for 
	$j=1,2$. It satisfies $\varphi_j(\gamma) = x_j\gamma^{\pi_{j,\tau}}$ 
	for 
	all $j=1,2$ and $\gamma\in\Gamma$. By restriction, one obtains morphisms
	\begin{align*}
	(\varphi_{1,+},\varphi_{2,+}) & :  (\Gamma_1 \sqcup \NN 2,\Gamma_2 \cup 
	\NN 2,\tau_+) \to (X_1,X_2,\tau) \\
	(\varphi_{1,-},\varphi_{2,-}) & : (\Gamma_1 \cup \NN 1, \Gamma_2 \sqcup 
	\NN 1, \tau_-) \to (X_1,X_2,\tau)
	\end{align*}
	from the positive and negative translation pre-actions.
\end{remark}
\begin{lemma}\label{Paths from BS tree to BS graphs}
	In the context of the above remark, the Bass-Serre morphism 
	$\mathbf{BS}(\varphi_1,\varphi_2)$, from the Bass-Serre tree $\Tc$ to 
	the Bass-Serre graph $\Gr_\tau$, sends $\pat_{j,1_\Gamma}^\Tc(\gamma)$ 
	onto 
	$\pat_{j,x_j}^{\Gr_\tau}(\gamma)$ for all $j\in\{1,2\}$ and all 
	$\gamma\in \NN i$ with $i\in\{1,2\}\setminus\{j\}$.
\end{lemma}
\begin{proof}
	We make the proof in the case $j=1$ only; the case $j=2$ is similar.
	
	Let us consider $\gamma \in \NN 2$, and write its normal form: $\gamma = c_1 \cdots c_n \sigma$. 
	Let us denote by $(e_0,e_1, \ldots, e_n)$ the edges of $\pat_{1,1_\Gamma}^\Tc(\gamma)$,
	and by $(e'_0,e'_1, \ldots, e'_n)$ the edges of $\pat_{1,x_1}^{\Gr_\tau}(\gamma)$. 
	The auxiliary sequences in $\Gamma$ and $X$ used in the construction of the paths will be denoted by $(\gamma_0,\ldots, \gamma_{n+1})$ and $(x_0,\ldots, x_{n+1})$ respectively.
	
	An easy induction shows that $x_i = \varphi_1(\gamma_i)$ when $i$ even, and $x_i = \varphi_2(\gamma_i)$ when $i$ is odd.
	Then, we notice that the source of $e_i = e_{1,\gamma_{i}}(c_i)$
	is $\gamma_{i} \Gamma_1$ when $i$ even,
	and the source of $e_i = e_{2,\gamma_{i}}(c_i)$
	is $\gamma_{i} \Gamma_2$ when $i$ odd.
	Thus, using Lemma \ref{BSmorphismsLocallyInjective}, 
	we get
	\[
	\mathbf{BS}(\varphi_1,\varphi_2)(e_i) =
	e_{1,\varphi_1(\gamma_{i})}^{} \circ e_{1,\gamma_{i}}^{-1} 
	\big( e_{1,\gamma_{i}}(c_i) \big) 
	= e_{1,x_{i}}(c_i) = e'_i
	\]
	when $i$ is even, and
	\[
	\mathbf{BS}(\varphi_1,\varphi_2)(e_i) =
	e_{2,\varphi_2(\gamma_{i})}^{} \circ e_{2,\gamma_{i}}^{-1} 
	\big( e_{2,\gamma_{i}}(c_i) \big) 
	= e_{2,x_{i}}(c_i) = e'_i
	\]
	when $i$ is odd.
\end{proof}
Therefore, if $x_1\Sigma_1$ is a treeing edge then, the image of $\mathbf{BS}(\varphi_{1,+},\varphi_{2,+})$ is the half-tree of $x_1\Sigma_1$ while, if $x_2\Sigma_2$ is a treeing edge, the image of $\mathbf{BS}(\varphi_{1,-},\varphi_{2,-})$ is the half-tree of $x_2\Sigma_2$.

\subsection{The free globalization of a pre-action of an amalgam}

First, let us notice that, for any $\sigma\in \Sigma$, there is an 
automorphism of pre-actions induced by left translation by $\sigma$ 
\[
(\gamma\mapsto \sigma\gamma,\gamma\mapsto \sigma\gamma)
\]
for each of the following pre-actions:
\begin{itemize}
	\item the translation pre-action $(\Gamma,\Gamma,\id)$;
	\item the positive translation pre-action $(\Gamma_1 \sqcup \NN 2,\Gamma_2 
	\cup \NN 2,\tau_+)$;
	\item the negative translation pre-action $(\Gamma_1 \cup \NN 1, \Gamma_2 
	\sqcup \NN 1, \tau_-)$.
\end{itemize}
Indeed, all sets $\Gamma,\Gamma_1,\Gamma_2,\Sigma,\NN 1,\NN 2$ are 
invariant by 
left translation by $\sigma$, hence the domains and range of $\tau_+$ and 
$\tau_-$ are invariant by left translation by $\sigma$. Then checking we 
have morphisms of pre-actions is a straightforward computation, and 
invertibility is obvious. 

\begin{proposition}\label{prop:chara positive treeing edge for amalgam}
	Consider a global pre-action $(X_1,X_2,\tau)$, and basepoints $x_1\in 
	X_1$ and $x_2=x_1\tau \in X_2$. The following are equivalent:
	\begin{enumerate}
		\item[(i)] the morphism of pre-actions $(\varphi_{1,+},\varphi_{2,+}): 
		(\Gamma_1 \sqcup \NN 2,\Gamma_2 \cup \NN 2,\tau_+) \to (X_1,X_2,\tau)$ 
		of Remark \ref{PseudoInitialObjects} is injective;
		\item[(ii)] the morphism of graphs $\mathbf{BS}(\varphi_{1,+},\varphi_{2,+})$ is injective;
		\item[(iii)] the edge $x_1\Sigma_1$ in the Bass-Serre graph $\mathbf{BS}(X_1,X_2,\tau)$ is a treeing edge.
	\end{enumerate}
\end{proposition}
\begin{proof}
	For all $\gamma\in \Gamma_j \cup \NN 2$, recall that $\varphi_{j,+}(\gamma) 
	= x_j\gamma^{\pi_{j,\tau}}$, so that 
	$\mathbf{BS}(\varphi_{1,+},\varphi_{2,+})$ sends vertices $\gamma\Gamma_j$ 
	to $x_j \gamma^{\pi_{j,\tau}} \Gamma_j$, and edges $\gamma\Sigma_j$ to 
	$x_j \gamma^{\pi_{j,\tau}} \Sigma_j$. Fixing $\gamma$, we get $\varphi_{j,+}(\gamma 
	g) = x_j\gamma^{\pi_{j,\tau}}g^{\pi_j}$ for $g\in \Gamma_j$; 
	since $X_j\curvearrowleft^{\pi_j}\Gamma_j$ is free, $\varphi_{j,+}$ 
	realizes a bijection between $\gamma\Gamma_j$ and $x_j 
	\gamma^{\pi_{j,\tau}} \Gamma_j$, and also a bijection between 
	$\gamma\Sigma_j$ and $x_j \gamma^{\pi_{j,\tau}} \Sigma_j$. Consequently, 
	$\varphi_{j,+}$ is injective if an only if $\gamma\Gamma_j\mapsto x_j 
	\gamma^{\pi_{j,\tau}} \Gamma_j$ and $\gamma\Sigma_j \mapsto x_j 
	\gamma^{\pi_{j,\tau}} \Sigma_j$ are both injective. This proves that (i) 
	and (ii) are equivalent. 
	
	The implication (iii) $\implies$ (ii) follows from the fact that 
	when $x_1\Sigma_1$ is a treeing edge 
	$\mathbf{BS}(\varphi_{1,+},\varphi_{2,+})$ is locally injective from the 
	half-tree of $\Sigma_1$ to the half-tree of $x_1\Sigma_1$, hence 
	$\mathbf{BS}(\varphi_{1,+},\varphi_{2,+})$ is injective. 
	
	Finally assume 
	(ii) and let $\omega$ be a reduced path starting by the edge 
	$x_1\Sigma_1$. By Remark \ref{Remark Path Amalgams} there exists 
	$\gamma\in \NN 2$ such that $\omega=\pat_{1,x_1}(\gamma)$. By Lemma 
	\ref{Paths from BS tree to BS graphs}, $\omega$ is the image by 
	$\mathbf{BS}(\varphi_{1,+},\varphi_{2,+})$ of 
	$\pat_{1,1_\Gamma}^\Tc(\gamma)$. 
	Since $\mathbf{BS}(\varphi_{1,+},\varphi_{2,+})$ is supposed to be 
	injective and since the last vertex of $\pat_{1,1_\Gamma}^\Tc(\gamma)$ is 
	not 
	$\Gamma_1$, we deduce that the last vertex of $\omega$ is not 
	$x_1\Gamma_1$. Hence, $x_1\Sigma_1$ is a treeing edge by Lemma \ref{Lemma 
		Treeing Edge}.\end{proof}

By a very similar argument, we get also the following result.
\begin{proposition}\label{prop:chara negative treeing edge for amalgam}
	Consider a global pre-action $(X_1,X_2,\tau)$, and basepoints $x_1\in 
	X_1$ and $x_2=x_1\tau \in X_2$. The following are equivalent:
	\begin{enumerate}
		\item[(i)] the morphism of pre-actions $(\varphi_{1,-},\varphi_{2,-}): 
		(\Gamma_1 \cup \NN 1, \Gamma_2 \sqcup \NN 1, \tau_-) \to 
		(X_1,X_2,\tau)$ of Remark \ref{PseudoInitialObjects} is injective;
		\item[(ii)] the morphism of graphs $\mathbf{BS}(\varphi_{1,-},\varphi_{2,-})$ is injective;
		\item[(iii)] the edge $x_2\Sigma_2$ in the Bass-Serre graph $\mathbf{BS}(X_1,X_2,\tau)$ is a treeing edge.
	\end{enumerate}
\end{proposition}

\begin{remark}
	Putting the two previous propositions together, one can show that, 
	given a global pre-action of $\Gamma$, its Bass-Serre graph is a forest 
	if and only if the action $X_1\curvearrowleft^{\pi_{1,\tau}}\Gamma$ (or 
	equivalently $X_2\curvearrowleft^{\pi_{2,\tau}}\Gamma$) is free. 
\end{remark}

Say that a pre-action is \textbf{transitive} when its Bass-Serre graph is 
connected. Note that a global pre-action $(X_1,X_2,\tau)$ is transitive if 
and only if the action $X_1 \curvearrowleft^{\pi_{1,\tau}}\Gamma$, or 
equivalently $X_2 \curvearrowleft^{\pi_{2,\tau}}\Gamma$, is a transitive 
action. 
We will show that every transitive pre-action has a canonical extension to 
a transitive action, which is \textbf{as free as possible}. The 
construction is again  better described in terms of the Bass-Serre graph: 
we are 
going to attach as many treeing edges as possible to it.

\begin{theorem}\label{thm: amalgam free globalization}
	Every transitive $\Gamma$-pre-action $(X_1,X_2,\tau)$ admits a 
	transitive and global extension $(\tilde X_1, \tilde X_2,\tilde \tau)$ 
	which satisfies the following universal property: given any transitive 
	and global extension $(Y_1,Y_2,\tau')$ of $(X_1,X_2,\tau)$, there 
	exists a unique morphism of pre-actions $(\varphi_1,\varphi_2):(\tilde 
	X_1, \tilde X_2,\tilde \tau) \to (Y_1,Y_2,\tau')$ such that 
	$$(\varphi_{1\restriction X_1}, \varphi_{2\restriction X_2})=(\id_{X_1},\id_{X_2}) .$$
	Moreover, all the (oriented) edges from the Bass-Serre graph $\mathbf{BS}(X_1,X_2,\tau)$ to its complement in $\mathbf{BS}(\tilde X_1, \tilde X_2,\tilde \tau)$ are treeing edges.
\end{theorem}

\begin{proof}
	
	We will obtain the Bass-Serre graph of this action by adding only treeing edges to the Bass-Serre graph of the pre-action. 
	First enumerate the $\Sigma_1$-orbits which do not belong to the domain 
	of $\tau$ as $(x_i\Sigma_1)_{i\in I_+}$, and the $\Sigma_2$-orbits 
	which do not belong to the range of $\tau$ as $(x_i\Sigma_2)_{i\in 
		I_-}$, with disjoint index sets $I_+, I_-$. Then, we take copies 
	$(Y_{1,i},Y_{2,i}, \tau_i)$,  of the positive translation pre-action 
	$(\Gamma_1 \sqcup \NN 2,\Gamma_2 \cup \NN 2,\tau_+)$, for $i\in I_+$, 
	and copies $(Y_{1,i},Y_{2,i}, \tau_i)$ , of the negative translation 
	pre-action $(\Gamma_1 \cup \NN 1, \Gamma_2 \sqcup \NN 1, \tau_-)$, for 
	$i\in I_-$, which are pairwise disjoint (by this, we mean $Y_{1,i}\cup 
	Y_{2,i}$ is disjoint from $Y_{1,i'}\cup Y_{2,i'}$ whenever $i\neq i'$), 
	and disjoint from the original pre-action $(X_1,X_2,\tau)$. We set then
	\[
	\tilde X_1 = {\left( X_1 \sqcup \bigsqcup_{i\in I_+} Y_{1,i} \sqcup \bigsqcup_{i\in I_-} Y_{1,i} \right)}\bigg/_{\sim_1}
	\quad \text{ and } \quad
	\tilde X_2 = {\left( X_2 \sqcup \bigsqcup_{i\in I_+} Y_{2,i} \sqcup \bigsqcup_{i\in I_-} Y_{2,i} \right)}\bigg/_{\sim_2}
	\]
	where $\sim_1$ identifies the element $x_i g \in X_1$ with $g\in 
	\Gamma_1 \subset Y_{1,i}$, for each $i\in I_+$ and $g\in\Gamma_1$, and 
	$\sim_2$ identifies the element $x_i h \in X_2$ with $h\in \Gamma_2 
	\subset Y_{2,i}$, for each $i\in I_-$ and $h\in\Gamma_2$. Since the 
	identifications just glue some orbits pointwise and respect the 
	$\Gamma_j$-actions, $\tilde X_1$ is endowed with a free 
	$\Gamma_1$-action, and $\tilde X_2$ is endowed with a free 
	$\Gamma_2$-action. Now, we set
	\[
	\tilde \tau = \tau \sqcup \bigsqcup_{i\in I_+} \tau_i \sqcup \bigsqcup_{i\in I_-} \tau_i \, ,
	\]
	which is possible since the domain of $\tau_i$, for $i\in I_+$, intersects other components in $\tilde X_1$ only in the orbit $x_i\Sigma_1$, the range of $\tau_i$, for $i\in I_+$, does not intersect other components in $\tilde X_2$, and the situation is analogue for $\tau_i$ with $i\in I_-$. We have got a pre-action $(\tilde X_1, \tilde X_2,\tilde \tau)$.
	
	This pre-action is transitive, since all pre-actions $(X_1,X_2,\tau)$ and $(Y_{1,i},Y_{2,i}, \tau_i)$ are, and the identifications make connections between each 
	$(Y_{1,i},Y_{2,i}, \tau_i)$ and $(X_1,X_2,\tau)$ in the Bass-Serre graph. 
	It is also global, since every $\Sigma_1$-orbit in $Y_{1,i}$, 
	respectively $\Sigma_2$-orbit in $Y_{2,i}$, which is not in the domain, 
	respectively the range, of $\tau_i$ has been identified with an orbit 
	in $X_1$, respectively $X_2$, 
	and every $\Sigma_1$-orbit in $X_1$, respectively $\Sigma_2$-orbit in $X_2$, is now in the domain, respectively the range, of $\tilde \tau$.
	
	Moreover, the (oriented) edges from the Bass-Serre graph 
	$\mathbf{BS}(X_1,X_2,\tau)$ to its complement in $\mathbf{BS}(\tilde 
	X_1, \tilde X_2,\tilde \tau)$ are exactly the edges $x_i\Sigma_1$ for 
	$i\in I_+$, and the edges $x_i\Sigma_2$ for $i\in I_-$. For each $i\in 
	I_+$, the morphism of pre-actions $(\varphi_{1,+},\varphi_{2,+}): 
	(\Gamma_1 \sqcup \NN 2,\Gamma_2 \cup \NN 2,\tau_+) \to (\tilde X_1, \tilde X_2,\tau)$ 
	of Remark \ref{PseudoInitialObjects}, with basepoints $x_i\in \tilde 
	X_1$ and $x_i \tilde \tau \in \tilde X_2$, is injective since it 
	realizes an isomorphism onto $(Y_{1,i},Y_{2,i}, \tau_i)$, hence 
	$x_i\Sigma_1$ is a treeing edge by Proposition \ref{prop:chara positive 
		treeing edge for amalgam}. One proves similarly that the edges 
	$x_i\Sigma_2$ are treeing edges using Proposition \ref{prop:chara 
		negative treeing edge for amalgam}. 
	
	It now remains to prove the universal property. To do so, take any 
	transitive and global extension $(Y_1,Y_2,\tau')$ of $(X_1,X_2,\tau)$. 
	Then, the unique morphism of pre-actions $(\varphi_1,\varphi_2)$ from 
	$(\tilde X_1, \tilde X_2,\tilde \tau)$ to $(Y_1,Y_2,\tau')$ such that 
	$(\varphi_{1\restriction X_1}, \varphi_{2\restriction 
		X_2})=(\id_{X_1},\id_{X_2})$ is obtained by taking the union of 
	$(\id_{X_1},\id_{X_2})$ with the morphisms 
	$(\varphi_{1,i},\varphi_{2,i})$ from $(Y_{1,i},Y_{2,i}, \tau_i)$ to 
	$(Y_1,Y_2,\tau')$ coming from Remark \ref{PseudoInitialObjects} with 
	respect to basepoints $x_i$ and $x_i\tilde \tau^{\pm 1}$, which are 
	unique.		
\end{proof}

It is straightforward to deduce from the universal property above that the 
action we just built is unique up to isomorphism. We thus call it 
\textbf{the} \textbf{free globalization} of the pre-action 
$(X_1,X_2,\tau)$. The interested reader can establish a connection with the 
notion of partial action, as we did in section \ref{sec: partial actions 
	HNN} for HNN extensions. For the sake of brevity, we just observe the 
following useful analogue of Proposition~\ref{Free globalization faithful HNN}.

\begin{remark}\label{Free globalization faithful}
	In the context of Theorem \ref{thm: amalgam free globalization}, if the 
	pre-action $(X_1,X_2,\tau)$ is not global, then the conjugate actions 
	$\pi_{1,\tilde\tau}$ and $\pi_{2,\tilde\tau}$ induced by the free 
	globalization $(\tilde X_1,\tilde X_2, \tilde \tau)$ are highly 
	faithful. Indeed, by Corollary \ref{EquivStrongAndHighFaithfulness}, it 
	suffices to prove that $\pi_{1,\tilde\tau}$ is strongly faithful.
	Notice that $(\tilde X_1,\tilde X_2, \tilde \tau)$ contains a copy of the 
	positive  (or of the negative) translation pre-action, 
	which correspond to a half-tree in $\mathbf{BS}(\tilde X_1,\tilde X_2, 
	\tilde \tau)$. Now note that the positive  translation pre-action is 
	strongly faithful, meaning that given $F\Subset\Gamma$, we can find $x\in 
	\Gamma_1\cup N_{C_1}$ such that for all $f\in F$, we have $xf\neq x$ and 
	$xf\in \Gamma_1\cup N_{C_1}$ (indeed it suffices to take $x\in N_{C_1}$ 
	with a sufficiently long normal form). Similarly, the negative translation 
	is strongly faithful. It follows that the free globalization is strongly 
	faithful, hence highly faithful as wanted.	 
\end{remark}

Let us furthermore observe that we can always build the free globalization on a fixed couple of sets $(\bar X_1,\bar X_2)$ with $\bar X_j$ containing $X_i$, provided $\bar X_j$ contains infinitely many free $\Gamma_j$-orbits.

\begin{theorem}\label{Globalization in prescribed sets for amalgams}
	Let $\bar X_j$ be a countable set equipped with a free $\Gamma_j$-action for $j=1,2$. Suppose $X_j\subseteq \bar X_j$ is $\Gamma_j$-invariant, and $\bar X_j\setminus X_j$ contains infinitely many $\Gamma_j$-orbits. Suppose further that we have a pre-action $(X_1,X_2,\tau)$. Then there is a bijection $\bar\tau : \bar X_1 \to \bar X_2$ which extends $\tau$ such that $(\bar X_1, \bar X_2,\bar \tau)$ is (isomorphic to) the free globalization of $(X_1, X_2,\tau)$.
\end{theorem}
\begin{proof}
	Let $(\tilde X_1, \tilde X_2,\tilde \tau)$ be the free globalization of $(X_1, X_2,\tau)$. The fact that $\tilde X_j\setminus X_j$ contains infinitely many $\Gamma_j$-orbits and is countable implies that there exist $\Gamma_j$-equivariant bijections $\varphi_j: \tilde X_j \to \bar X_j$ whose restrictions to $X_j$ are the identities. Then, one can push forward the bijection $\tilde \tau$, to obtain a bijection $\bar \tau: \bar X_1 \to \bar X_2$ defined by
	\[
	x \varphi_1 \bar\tau :=  x \tilde\tau \varphi_2 \quad \text{ for all } x\in \tilde X_1 \, ,
	\]
	which extends $\tau$.
	Now, $(\varphi_1,\varphi_2)$ is an isomorphism of pre-actions between $(\tilde X_1, \tilde X_2,\tilde \tau)$ and $(\bar X_1, \bar X_2,\bar \tau)$.
\end{proof}

\section{High transitivity for amalgams}\label{High transitivity for amalgams}

As in Section \ref{SectFreeGlobalizationAmalgams}, we fix an amalgam 
$\Gamma=\Gamma_1 *_\Sigma \Gamma_2$, and sets of representatives $C_j$ of 
left $\Sigma_j$-cosets in $\Gamma_j$ such that $1\in C_j$, for $j=1,2$, so 
that normal forms of elements of $\Gamma$ are well-defined. We still denote 
by $\NN j$ the set of elements of $\Gamma$ whose normal form begins with an 
element of $C_j\setminus\{1\}$, for $j=1,2$, so that we have 
$\Gamma = 
\Sigma \sqcup 
\NN 1 
\sqcup \NN 2$.

Non-degeneracy and topological freeness become now essential. Hence, \textbf{we assume from now on that our amalgam $\Gamma$ is non-degenerate} and that \textbf{the $\Gamma$-action on the boundary of its Bass-Serre tree is topologically free}. 

\subsection{Using the free globalization towards high transitivity}
This section is devoted to a key proposition which will allow us to extend 
any given transitive pre-action which is not global to a global one such 
that the associated $\Gamma$-action sends one fixed tuple to another fixed 
tuple. 

\begin{proposition}\label{prop: key prop for amalgams}
	Suppose $(X_1,X_2,\tau)$ is a transitive non-global pre-action, that 
	$X_j$ is a finite union of orbits of a free action $\bar 
	X_j\curvearrowleft \Gamma_j$, where $\bar X_j$ is countable, and that the complement $\bar X_j \setminus X_j$ 
	contains infinitely many $\Gamma_j$-orbits.
	Let $x_1,...,x_k,y_1,...,y_k\in \bar X_1$ be pairwise distinct points. 
	Then $(X_1,X_2,\tau)$ can be extended to a transitive and global 
	pre-action $(\bar X_1, \bar X_2,\tilde \tau)$ so that there is an 
	element $\gamma\in\Gamma$ such that $x_i \gamma^{\pi_{1,\tilde 
			\tau}}=y_i$, and the action $\pi_{1,\tilde\tau}$ is highly faithful.
\end{proposition}
Notice that the choice to work in $\bar X_1$ is arbitrary. We could prove a similar statement for the $\Gamma_2$-action on $\bar X_2$. 

\begin{proof}
	We will denote the set $\{x_1,...,x_k,y_1,...,y_k\}$ by $F$.  First, by Theorem \ref{Globalization in prescribed sets for amalgams}, we find a bijection $\bar\tau: \bar X_1 \to \bar X_2$ such that $(\bar X_1, \bar X_2,\bar \tau)$ is the free globalization of $(X_1, X_2,\tau)$. 
	
	\begin{claim}
		There exists a path-type element $\gamma$ in $\NN 2$ such that for 
		every $x\in F$, the last edge of $\pat_{1,x}(\gamma)$ is a treeing 
		edge.
	\end{claim}
	\begin{cproof}
		Recall the correspondence established in Section 
		\ref{SubsectionPathsAmalgams} between path-type elements and reduced 
		paths of even length. Since $\mathbf{BS}(\bar X_1, \bar X_2,\bar 
		\tau)$ is connected and has treeing edges, 
		it follows from Lemma \ref{lem: extend path with treeing edge} that 
		for every $x\in\bar X_1$, and every path-type element $\gamma\in 
		\NN 2$, 
		there is a path-type extension $\gamma'$ of $\gamma$ such that the 
		last edge of $\pat_{1,x}(\gamma')$ is a treeing edge. 
		Now, it suffices to start with any path-type element $\gamma_0 \in 
		\NN 2$, to extend it to a path-type element $\gamma_1$ such that the 
		last edge of $\pat_{1,x_1}(\gamma_1)$ is a treeing edge, 
		then to extend $\gamma_1$ to a path-type element $\gamma_2$ such that 
		the last edge of $\pat_{1,y_1}(\gamma_2)$ is a treeing edge, \ldots,
		and iterate this extension procedure until we reach an element 
		$\gamma_{2k}\in \NN 2$ such that all last edges of 
		$\pat_{1,x}(\gamma_{2k})$, for all $x\in F$, are treeing edges (by 
		Remark \ref{RemarkPathAmalgams} $(4)$). 
	\end{cproof}
	Given $x\in \bar X_1$, and a path-type element $\gamma$ in $\NN 2$, we 
	will denote by $\Hc_{x}(\gamma)$ the half-graph of the last edge of 
	$\pat_{1,x}(\gamma)$.
	\begin{claim}
		There exists a path-type element $\gamma$ in $\NN 2$ such that for 
		every $x\in F$, the last edge  of $\pat_{1,x}(\gamma)$ is a treeing 
		edge, 
		and the half-trees $\Hc_{x}(\gamma)$, for $x\in F$, are pairwise 
		disjoint subgraphs, and all disjoint from $\mathbf{BS}(X_1,X_2,\tau)$.
	\end{claim}
	
	\begin{cproof}
		We start with a path-type element $\gamma$ in $\NN 2$ such that for 
		every 
		$x\in F$, the last edge of $\pat_{1,x}(\gamma)$ is a treeing edge. 
		Since $X_j$ is a finite union of $\Gamma_j$-orbits for $j=1,2$, the Bass-Serre graph $\mathbf{BS}(X_1,X_2,\tau)$ has finitely many vertices. 
		Hence, by extending further the path-type element $\gamma$, we can assume that for every $x\in F$, the half-tree $\Hc_{x}(\gamma)$ does not intersect $\mathbf{BS}(X_1,X_2,\tau)$.
		
		Notice that, given $x,y\in F$, if the half-trees $\Hc_{x}(\gamma)$ and $\Hc_{y}(\gamma)$ are disjoint, 
		then so are the half-trees $\Hc_{x}(\gamma')$ and $\Hc_{y}(\gamma')$ for every path-type extension $\gamma'$ of $\gamma$, since $\Hc_{x}(\gamma') \subseteq \Hc_{x}(\gamma)$ and $\Hc_{y}(\gamma') \subseteq \Hc_{y}(\gamma)$.
		Hence, it suffices to prove that, for any $x,y\in F$ with $x\neq y$ and such that $\Hc_{x}(\gamma)$ and $\Hc_{y}(\gamma)$ intersect,
		there exists a path-type extension $\gamma'$ of $\gamma$ such that $\Hc_{x}(\gamma')$ and $\Hc_{y}(\gamma')$ are disjoint. 
		Indeed, an easy induction gives then an extension $\gamma^{(n)}$ such that the half-trees $\Hc_{x}(\gamma^{(n)})$, for $x\in F$, are pairwise disjoint.
		
		Take now $x,y\in F$ with $x\neq y$ and such that $\Hc_{x}(\gamma)$ and $\Hc_{y}(\gamma)$ intersect. 
		These half-trees have to be nested. Indeed, if they are not, 
		$\Hc_{x}(\gamma)$ contains the antipode of the last edge of 
		$\pat_y(\gamma)$, hence contains $\mathbf{BS}(X_1,X_2,\tau)$, which 
		is impossible.
		Without loss of generality, we assume $\Hc_{x}(\gamma) \subseteq \Hc_{y}(\gamma)$. We now distinguish two cases.
		\begin{itemize}
			\item If $\Hc_{x}(\gamma) \subsetneq \Hc_{y}(\gamma)$, there is a 
			path type extension $\gamma''$ of $\gamma$ such that 
			$\pat_{1,x}(\gamma)$ and $\pat_{1,y}(\gamma'')$ have the same last 
			edge. 
			We have the product of normal forms  
			\[
			\gamma'' = \gamma \cdot (c_1  \cdots c_n)  \, ,
			\]
			where $n\geq 2$ is even. Since the amalgam $\Gamma$ is 
			non-degenerate, we can obtain another normal form $\gamma' = 
			\gamma \cdot (c'_1 \cdots c'_n)$
			by replacing a letter $c_i$ in the factor $\Gamma_j$ such 
			that $[\Gamma_j:\Sigma_j]\geq 3$ by another letter $c'_i$ in 
			$C_j\setminus\{1\}$.
			This change has the effect that $\pat_{1,y}(\gamma')$ and 
			$\pat_{1,y}(\gamma'')$ are distinct reduced paths (which are both 
			extensions of the $\pat_{1,y}(\gamma)$). 
			Hence, since $\Hc_{y}(\gamma)$ is a tree, the sub-trees  $\Hc_{y}(\gamma'')$ and $\Hc_{y}(\gamma')$ must be disjoint. 
			Since $\Hc_{x}(\gamma) = \Hc_{y}(\gamma'')$ we are done.			 
			\item If $\Hc_{x}(\gamma) = \Hc_{y}(\gamma)$, then 
			$\pat_{1,x}(\gamma)$ and $\pat_{1,y}(\gamma)$ have the same 
			terminal edge, which is 
			\[
			e := \overline{x' \Sigma_1} = \overline{y' \Sigma_1} \, , 
			\quad \text{ where } \quad 
			x' = x \gamma^{\pi_{1,\bar\tau}} \text{ and } y' = y \gamma^{\pi_{1,\bar\tau}} \, .
			\]
			Consequently, one has $y' = x'\sigma^{\pi_{1,\bar\tau}}$ for some $\sigma\in \Sigma$. 
			Note that, since $x\neq y$, one has $\sigma\neq 1$ and consider 
			the morphism of pre-actions from the negative translation 
			pre-action $(\varphi_{1,-},\varphi_{2,-}) : (\Gamma_1 \cup \NN 1, 
			\Gamma_2 \sqcup \NN 1, \tau_-) \to (\bar X_1, \bar X_2,\bar\tau)$ 
			coming from Remark \ref{PseudoInitialObjects}, with basepoints $x_1 = x'$ and $x_2 = x' \bar \tau$.
			Since $e$ is a treeing edge, this morphism is injective by Proposition \ref{prop:chara negative treeing edge for amalgam}.
			The half-tree $\Hc_{x}(\gamma) = \Hc_{y}(\gamma)$ is thus isomorphic, via $\mathbf{BS}(\varphi_{1,-},\varphi_{2,-})$, to the half-tree $\Hc$ of $\Sigma_2$ in the Bass-Serre tree $\Tc$.
			
			Note that the left translation by $\sigma$ (i.e. $\gamma^*\mapsto 
			\sigma\gamma^*$) defines an automorphism of the negative 
			translation pre-action, which we write as 
			$(\sigma_1,\sigma_2)$. The morphism of graphs
			$\mathbf{BS}(\sigma_1,\sigma_2)$ maps 
			$\pat^\Tc_{2,1_\Gamma}(\gamma^*)$ to 
			$\pat^\Tc_{2,\sigma}(\gamma^*)$ 
			in $\Hc$ by Lemma \ref{Paths from BS tree to BS graphs} (note that these paths both have $\Sigma_2$ as first edge).
			Since the left $\Gamma$-action on the boundary $\partial\Tc$ of 
			its Bass-Serre tree is topologically free, the left $\sigma$ 
			action does not fix 
			the half-tree $\Hc$ pointwise. 
			Hence there exists an element $\gamma^* \in \Sigma \sqcup \NN 1$ 
			such that $\pat^\Tc_{2,1}(\gamma^*)$ and 
			$\pat^\Tc_{2,\sigma}(\gamma^*)$ have distinct ranges. 
			Moreover, up to extending $\gamma^*$, we can further assume that 
			$\gamma^*$ is a path type element of $\NN 1$ so that
			the element 
			$\gamma' := \gamma\gamma^*$ is a path-type element of $\NN 2$.
			
			Now, the images of $\pat^\Tc_{2,1}(\gamma^*)$ and $\pat^\Tc_{2,\sigma}(\gamma^*)$ by $\mathbf{BS}(\varphi_{1,-},\varphi_{2,-})$ 
			are $\pat_{2,x'\bar \tau}(\gamma^*)$ and $\pat_{2,y' 
				\bar\tau}(\gamma^*)$ by Lemma \ref{Paths from BS tree to BS 
				graphs}, 
			and these paths diverge in the half-tree $\Hc_{x}(\gamma) = \Hc_{y}(\gamma)$. 
			Note finally that the starting edges of these paths are both equal 
			to $e$; this implies that $\pat_{1,x}(\gamma')$ and 
			$\pat_{1,y}(\gamma')$ don't have the same range. 
			Hence $\Hc_{x}(\gamma')$ and $\Hc_{y}(\gamma')$ are disjoint.
		\end{itemize}
		We are done in both cases.
	\end{cproof}
	
	We then modify the bijection $\bar \tau$ to get the pre-action $(\bar X_1,\bar X_2,\tilde \tau)$ we are looking for. 
	First, given an element $\gamma$ as in the previous claim, we consider for each $z\in F$ the morphism of pre-actions from the positive translation pre-action
	\[  
	(\psi_{1,z},\psi_{2,z}) : (\Gamma_1 \sqcup \NN 2,\Gamma_2 \cup 
	\NN 2,\tau_+) \to (\bar X_1, \bar X_2,\bar\tau)
	\]
	coming from Remark \ref{PseudoInitialObjects}, with basepoints $z' := z\gamma^{\pi_{1,\bar\tau}} \in \bar X_1$ and $z' \bar \tau \in \bar X_2$. 
	Note that the image of this morphism corresponds to the half-graph 
	opposite to the half-tree $\Hc_{z}(\gamma)$. Then, we define 
	$X_j'=\bigcap_{z\in F}\rng(\psi_{j,z})\subset\bar X_j$, and consider 
	the restriction $(X_1',X_2',\tau')$ of $(\bar X_1,\bar X_2,\bar\tau)$. 
	Informally speaking, we erase $\bar \tau$ on the $\Sigma_1$-orbits 
	corresponding to edges in the half-trees $\Hc_{z}(\gamma)$ for $z\in F$. 
	Note that this leaves infinitely many $\Gamma_1$-orbits in $\bar 
	X_1$ outside $\dom(\tau')$, respectively infinitely many 
	$\Gamma_2$-orbits in $\bar X_2$ outside $\rng(\tau')$, and the 
	pre-action $(X_1',X_2',\tau')$ is transitive.
	Notice also that, for any $z\in F$, we have $z' \Gamma_1 \cap \dom(\tau') = z' \Sigma_1$
	(in other words, the only edge in the star at $z'$ which belongs to $\mathbf{BS}(X_1',X_2',\tau')$ is $z'\Sigma_1$).
	In particular, given any $c_1 \in C_1\setminus\{1\}$, the orbits $x'_i c_1 \Sigma_1$ and $y'_i c_1 \Sigma_1$, for $1\leq i\leq k$, are not in $\dom(\tau')$.
	
	We now extend $\tau'$. Pick some orbits $z_1\Gamma_2, \ldots, z_k 
	\Gamma_2$ in $\bar X_2 \setminus \rng(\tau')$, add them to $X_2'$, take $c_j$ in 
	$C_j\setminus\{1\}$ for $j=1,2$, and set $x_i'c_1\sigma\tau' := z_i 
	\vartheta(\sigma)$ and $y_i'c_1 \sigma \tau' := z_i c_2 
	\vartheta(\sigma)$ for $i=1,\ldots,k$ and $\sigma \in \Sigma_1$. 
	This is possible since the $\Sigma_2$-orbits of the points $z_i$ and 
	$z_ic_2$ are pairwise disjoint (we use again the freeness of the 
	$\Gamma_2$-action), and since the $\Sigma_1$-orbits at $x'_i c_1$, $y'_i c_1$ 
	for $1\leq i\leq k$ are pairwise disjoint and were not initially in the domain of 
	$\tau'$.
	Note that, after this extension, $(X_1',X_2',\tau')$ is still transitive. Then we apply Theorem \ref{Globalization in prescribed sets for amalgams} to get an extension $\tilde \tau: \bar X_1\to \bar X_2$ of $\tau'$ such that $(\bar X_1, \bar X_2,\tilde\tau)$ is the free globalization of $(X_1',X_2',\tau')$. 
	A computation shows then that $x_i(\gamma c_1 c_2 c_1\inv 
	\gamma\inv)^{\pi_{1,\tilde\tau}} = y_i$ for all $i=1,\ldots,k$. 
	Finally, the action $\pi_{1,\tilde\tau}$ is highly faithful by Remark 
	\ref{Free globalization faithful}.
\end{proof}

\subsection{Highly transitive actions of amalgams}\label{SectHTAmalg}

From now on, we fix free actions $X_1 \curvearrowleft^{\pi_1}\Gamma_1$ and $X_2 \curvearrowleft^{\pi_2}\Gamma_2$ with infinitely many orbits. We endow the set of bijections from $X_1$ onto $X_2$ with the topology of pointwise convergence, which is a Polish topology. We then set
\[
\PA = \{\tau: X_1 \to X_2 \text{ bijective} : \, x\sigma \tau = x\tau 
\vartheta(\sigma) \text{ for all } \sigma\in \Sigma_1\} \, .
\]
In other words, $\PA$ is the set of bijections $\tau: X_1 \to X_2$ such 
that 
$(X_1,X_2,\tau)$ is a (global) pre-action of $\Gamma$. This is clearly a 
closed subset for the topology of pointwise convergence, hence a Polish 
space. Recall that every $\tau\in \PA$ induces an action 
$X_j\curvearrowleft^{\pi_{j,\tau}}\Gamma$ for $j=1,2$. 
We will focus on the action
$\pi_{1,\tau}$, which we will abbreviate by $\pi_\tau$.

\begin{definition}
	Let us set
	\begin{align*}
	\TA &= \{\tau\in \PA\colon
	\pi_\tau \text{ is transitive} \} \ ; \\
	\HFA &= \{\tau\in \PA \colon  
	\pi_\tau
	\text{ is highly faithful } \} \ ; \\
	\HTA &= \{\tau\in \PA \colon 
	\pi_\tau
	\text{ is highly transitive } \} \ .
	\end{align*}
\end{definition}

As in the HNN case, the subset $\TA$ isn't closed for the 
topology of pointwise convergence, but we have the following result.
\begin{lemma}\label{PolishSpaceAmalgam}
	The set $\TA$ is $G_\delta$ in $\PA$, 
	hence a Polish space. 
	Moreover, $\TA\neq\emptyset$.
\end{lemma}
\begin{proof}
	Since  $X_1 \curvearrowleft^{\pi_1}\Gamma_1$ and $X_2 \curvearrowleft^{\pi_2}\Gamma_2$ have infinitely many orbits, there are $\Gamma_j$-equivariant bijections $\varphi_j:\Gamma \to X_j$ for $j=1,2$. 
	It then suffices to push-forward the translation pre-action by 
	$(\varphi_1,\varphi_2)$ to get an element of $\TA$ (its Bass-Serre 
	graph will be isomorphic to the classical Bass-Serre tree and 
	$\pi_\tau$ will be conjugated to the translation action 
	$\Gamma\curvearrowleft\Gamma$). Hence, $\TA$ is non-empty. To show that 
	$\TA$ is $G_\delta$ in $\PA$, it suffices to write 
	$\TA=\bigcap_{x,x'\in X_1}O_{x,y}$, where for $x,y\in X_1$, 
	$O_{x,y}=\{\tau\in\PA\,:\,\text{ there exists }\gamma\in\Gamma\text{ 
		such that }x\gamma^{\pi_\tau}=y\}$.
	Since $O_{x,y}$ is obviously open in $\PA$ for all $x,y\in X_1$, 
	this shows that $\TA$ is a $G_\delta$ subset of $\PA$. \end{proof}

Here comes the theorem proving that our amalgam $\Gamma$ admits a highly 
transitive highly faithful action, thus proving Theorem 
\ref{ThmGroupsHTAmalgam}.

\begin{theorem}\label{ThmGroupsHTamalgamBaire}
	The set $\HTA\cap\HFA$ is dense $G_\delta$ in $\TA$. 
	In particular, $\Gamma$ admits actions which are both highly transitive and highly faithful.
\end{theorem}
\begin{proof}
	For $k\geq 1$ and $x_1,\dots x_k,y_1,\dots,y_k\in X_1$ pairwise distinct, the sets 
	\[
	V_{x_1,\dots,x_k,y_1,\dots,y_k}=\{\tau\in\TA\,:\,
	\exists\gamma\in\Gamma\,,\,x_i\gamma^{\pi_\tau}=y_i\text{ for all }1\leq i\leq k\}
	\]
	are obviously open in $\TA$. Similarly, for finite subsets $F$ of 
	$\Gamma\setminus\{1\}$, the sets
	\[
	W_F = \{ \tau\in\TA \, : \,
	\exists x\in X_1\,,\, x f^{\pi_\tau} \neq x\text{ for all } f\in F \}
	\]
	are also obviously open in $\TA$.
	Now, using Lemma \ref{HT and disjoint supports},
	and since every strongly faithful action of $\Gamma$ is highly faithful by Corollary \ref{EquivStrongAndHighFaithfulness}, we have
	\[
	\HTA\cap\HFA = \underset{\underset{\text{ \ \ \ \ \ \ \ \  \ \ \ \ \ \
				pairwise distinct}}
		{F\Subset \Gamma\setminus\{1\}, \, k\geq 1, \, x_1,\dots x_k,y_1,\dots,y_k\in X_1}}
	{\bigcap}(V_{x_1,\dots,x_k,y_1,\dots,y_k} \cap W_F) \, .
	\]
	
	To conclude, it suffices to show that each set 
	$(V_{x_1,\dots,x_k,y_1,\dots,y_k}) \cap \HFA$ is dense in $\TA$, since this 
	immediately implies that each open set $(V_{x_1,\dots,x_k,y_1,\dots,y_k}) 
	\cap W_F$ is dense in $\TA$. 
	To do this, let $\tau\in\TA$ and let $F$ be a finite subset of $X_1$. 
	Fix a finite
	connected subgraph $\Gr$ of $\mathbf{BS}(X_1,X_2,\tau)$ 
	containing the edges $z\Sigma_1$ for $z\in F$, and denote by $\tau_0$ the 
	restriction of $\tau$ to the union of the $\Sigma_1$-orbits in $X_1$ 
	corresponding to the edges of $\Gr$.
	Then apply Proposition~\ref{prop: key prop for amalgams} to the transitive 
	pre-action $(\dom(\tau_0)\cdot\Gamma_1,\rng(\tau_0)\cdot\Gamma_2,\tau_0)$, 
	whose Bass-Serre graph is $\Gr$, to get an extension $\tau'$ such that 
	$\tau'\in V_{x_1,\dots,x_k,y_1,\dots,y_k}\cap\HFA$. Moreover, since 
	$F\subset \dom(\tau_0)$, it follows that $\tau$ and $\tau'$ coincide on 
	$F$.\end{proof}

\begin{remark} As in Remark \ref{rmk: HNN HT without Baire}, one can give a 
	direct proof of the previous theorem without relying on Baire's theorem.
\end{remark}

\section{Highly transitive actions of groups acting on trees}\label{GroupsActingTrees}

\subsection{Proofs of Theorem \ref{ThmMain} and \ref{Thm td}}
Let us begin with a few preliminaries.

Suppose we are given an action  of a countable group $G$ on a 
tree $\mathcal{T}$, and a proper subtree $\Tc'$ such that $\Tc'$ and 
$g\Tc'$, where $g\in G$, are either equal or disjoint subtrees.

We can then can form a ``quotient'' tree $\bar \Tc$ by shrinking each 
subtree $g\Tc'$ to a single vertex, that we will denote by $(g\Tc')$. 
The tree $\bar\Tc$ is naturally endowed with a $G$-action and the ``quotient map'' $q:\Tc \to \bar\Tc$ is $G$-equivariant. 
The image by $q$ of a path in $\Tc$ is a path which is obtained by shrinking each subpath contained in a subtree $g\Tc'$ to the vertex $(g\Tc')$ in $\bar \Tc$. 
In case of a geodesic ray, its image by $q$ is either a geodesic ray in $\bar \Tc$, or a geodesic which ends at a vertex $(g\Tc')$. 
Hence $q$ induces a map $\partial q:\partial \Tc \to V(\bar\Tc) \cup \partial\bar\Tc$.
\begin{remark}\label{Boundary after shrinking}
	The restriction of $\partial q$ to $(\partial q)\inv(\partial\bar\Tc)$ is injective.
\end{remark}
\begin{proof}
	Given $\xi,\xi'\in\partial \Tc$ such that $\partial q (\xi)$ and $\partial q(\xi')$ lie in $\partial \bar\Tc$, consider geodesic rays $\omega,\omega'$ in $\Tc$ tending to $\xi,\xi'$.
	Each ray contains all edges of its image under $q$. 
	Hence $\partial q (\xi) = \partial q(\xi')$ implies that $\omega$ and $\omega'$ have infinitely many common edges, and therefore $\xi = \xi'$. 
\end{proof}
One can also notice, although we do not need this fact below, that $\partial q$ is continuous at each point $\xi\in (\partial q)\inv(\partial\bar\Tc)$.
In case $\Tc'$ is bounded, one has in fact $\partial q:\partial \Tc \to \partial\bar\Tc$, and $\partial q$ is continuous and injective.
\begin{lemma}\label{Permanence under shinking}
	In the context above, assume that $G\curvearrowright\mathcal{T}$ is a minimal action. Then:
	\begin{enumerate}[label=(\arabic*)]
		\item if $G\curvearrowright \Tc$ is of general type, then so is $G\curvearrowright \bar\Tc$;
		\item if $G\curvearrowright \partial\Tc$ is topologically free, then so is $G\curvearrowright \partial\bar\Tc$.
	\end{enumerate}
\end{lemma}
\begin{proof}
	Assume first $G\curvearrowright \Tc$ is of general type, in order to prove (1). 
	The hypotheses on $\Tc'$ guarantee the existence of an edge $e$ in $\Tc$ which lies outside all translates $g\Tc'$. 
	By minimality of $G\curvearrowright\mathcal{T}$, there exists a hyperbolic element $h\in G$ whose axis in $\Tc$ contains $e$. 
	Pick $g_1,g_2 \in G$ which induce transverse hyperbolic automorphisms of $\Tc$. 
	For $n$ sufficiently large, $h_1 = g_1^n h g_1^{-n}$ and $h_2 = g_2^n h g_2^{-n}$ induce transverse hyperbolic automorphisms of $\Tc$. 
	Moreover, their axes do contain edges in the orbit of $e$, so that their images by $q$ lie in $\partial \bar \Tc$.
	Hence, by Remark \ref{Boundary after shrinking}, $h_1$ and $h_2$ induce transverse hyperbolic automorphisms of $\bar\Tc$.
	This proves that $G\curvearrowright \bar\Tc$ is of general type.
	
	Assume now that $G\curvearrowright \partial\Tc$ topologically free, in order to prove (2). 
	Then assume that $g\in G$ fixes a half-tree $\Hc'$ in $\bar \Tc$, corresponding to some edge $e$ in $\Tc'$, pointwise. 
	Notice that $e$ is an edge of $\Tc$ that $q$ does not shrink, and denote by $\Hc$ its half-tree in $\Tc$. One has $q(\Hc) = \Hc'$.
	By minimality of $G\curvearrowright\mathcal{T}$, the edges of the orbit $G\cdot e$ which lie in $\Hc$ do generate $\Hc$. 
	Since they also lie in $\Hc'$, they are fixed pointwise by $g$, therefore the half-tree $\Hc$ itself is fixed pointwise by $g$. 
	Since $G\curvearrowright \partial\Tc$ is topologically free, $g$ has to be the trivial element,
	and this proves that $G\curvearrowright \partial\bar\Tc$ is topologically free.
\end{proof}

We now recall the statement of Theorem \ref{ThmMain} before proving it.
\begin{theoremSansNum}
	Let $\Gamma\curvearrowright\mathcal{T}$ be a minimal action of general type of a countable group $\Gamma$ on a tree $\mathcal{T}$. If the action on the boundary $\Gamma\curvearrowright \partial\mathcal{T}$ is topologically free, then $\Gamma$ admits a highly transitive and highly faithful action; in particular, $\Gamma$ is highly transitive.
\end{theoremSansNum}
\begin{proof}[Proof of Theorem \ref{ThmMain}]
	Let us consider an edge $e$ in $\Tc$. 
	The complement of the orbits $\Gamma \cdot e$ and $\Gamma \cdot \bar e$ in $E(\Tc)$ is either empty, or generates a disjoint union of subtrees of $\Tc$. 
	Since each of these subtrees contains some endpoint of some translate of $e$, there are at most two orbits of subtrees. 
	Hence by applying Lemma \ref{Permanence under shinking} zero, one, or two times, one gets a tree $\bar\Tc$ endowed with a $\Gamma$-action which is still of general type, 
	and such that $\Gamma\curvearrowright \partial\bar\Tc$ is topologically free. 
	Moreover, the action on $E(\bar \Tc)$ is transitive. 
	Now, the quotient $\Gamma\backslash \bar\Tc$ is either a segment or a loop, 
	the fundamental group of the corresponding graph of groups, which is either an amalgam or an HNN extension, is isomorphic to $\Gamma$, 
	and $\bar\Tc$ is the associated Bass-Serre tree.
	Applying Theorem~\ref{ThmGroupsHTAmalgam} or Theorem~\ref{ThmMainHNN}, we finally get that $\Gamma$ admits a highly transitive and highly faithful action.
\end{proof}

We can now also prove Theorem \ref{Thm td}, but let us first recall its 
statement.

\begin{theoremSansNum}
	Let $\Gamma\curvearrowright\mathcal{T}$ be a faithful minimal action of 
	general type of a countable group $\Gamma$ on a tree $\mathcal{T}$.  
	The following are equivalent
	\begin{enumerate}[label=(\arabic*)]
		\item \label{it: td}$\mathrm{td}(\Gamma)\geq 4$;
		\item \label{it: HT}$\Gamma$ is highly transitive;
		\item \label{it: MIF}$\Gamma$ is MIF;
		\item \label{it: top free boundary}$\Gamma\act \partial\Tc$ is 
		topologically free.
	\end{enumerate}
\end{theoremSansNum}
\begin{proof}[Proof of Theorem \ref{Thm td}]
	The implication \ref{it: HT} $\implies$ \ref{it: td} is clear. The 
	implication \ref{it: td} $\implies$ \ref{it: top free boundary} is 
	Le Boudec and Matte Bon's main result \cite[Thm 
	1.4]{leboudecTripleTransitivityNonfree2019}. The implication \ref{it: 
		top free boundary} $\implies$ \ref{it: HT} is a consequence of 
	Theorem \ref{ThmMain}. 
	So \ref{it: td}, \ref{it: HT} and \ref{it: top free boundary} are 
	all equivalent.
	
	To prove that these three statements are also equivalent to \ref{it: 
		MIF}, note that Theorem \ref{ThmMain} shows moreover that, under the 
	assumption \ref{it: top free boundary}, the group $\Gamma$ admits a 
	highly transitive \emph{highly faithful} action. 
	So for such an action, all its elements have infinite support, which by 
	\cite[Corollary~
	5.8]{hullTransitivitydegreescountable2016} implies that $\Gamma$ is 
	MIF. The implication \ref{it: top free boundary} $\implies$ 
	\ref{it: MIF} thus holds.
	Finally, the implication \ref{it: MIF} $\implies$ \ref{it: top free 
		boundary} follows from \cite[Proposition~
	3.7]{leboudecTripleTransitivityNonfree2019}	
\end{proof}

\subsection{Corollary \ref{core-free implies high transitivity} 
	and its implication of former results}
We now turn to the proof of a lemma which directly implies Corollary 
\ref{core-free implies high transitivity} via Theorem \ref{ThmMain}, and 
then we check that Corollary \ref{core-free implies high transitivity} 
applies to all groups acting on trees which can be proven to be highly 
transitive by 
previous results quoted in the introduction.

Recall that, given a subtree $\Uc$ of $\Tc$, we  denote by $G_\Uc$ the 
pointwise stabilizer of $\Uc$ in $G$. The following lemma is a 
generalization of Prop. 19 (iv) and Prop. 20 (iv) from 
\cite{delaharpeSimpleGroupsAmalgamated2011}.

\begin{lemma}\label{core-free implies topologically free}
	Let $G\curvearrowright \Tc$ be a faithful and minimal action such that 
	$G$ contains a hyperbolic element $h$.
	If there exist a bounded subtree $\mathcal{B}$ and a vertex $u$ in 
	$\mathcal{B}$ such that $G_\mathcal{B}$ is core-free in $G_u$, then the 
	induced action $G\curvearrowright \partial\Tc$ is topologically free.
\end{lemma}
\begin{proof}
	Let $\Bc'$ be the union of the translates $g \Bc$ for $g\in G_u$. 
	This is a subtree, since all $g\Bc$ contain $u$, which is 
	$G_u$-invariant and contained in the ball of radius $\diam(\Bc)$  
	centered at $u$. 
	Let $g_0$ be an element of $G$ fixing a half-tree $\Hc$ pointwise.
	Up to conjugating by a suitable power of $h$, we may and will assume 
	that $\Hc$ contains $\Bc'$, so that $g_0$ is in $G_{\Bc'}$. 
	Now, as $G_\mathcal{B}$ is core-free in $G_u$, we have
	$G_{\Bc'} = \bigcap_{g\in G_u}G_{g\Bc} = \bigcap_{g\in G_u} g G_\Bc 
	g\inv = \{1\}$.
	Thus, we get $g_0 = 1$, which proves that $G\curvearrowright 
	\partial\Tc$ is topologically free by Corollary~\ref{half-trees and 
		topological freeness}.
\end{proof}

We now prove that Corollary \ref{core-free implies high transitivity} 
applies to all groups acting on trees which are highly 
transitive via the combination of the results of Minasyan-Osin 
\cite{minasyanAcylindricalHyperbolicityGroups2015} and Hull-Osin 
\cite{hullTransitivitydegreescountable2016}. 

\begin{proposition}\label{proof applicability Cor F to MO+HO}
	Let $\Gamma$ be a countable group acting minimally on a tree $\Tc$. Suppose that 
	\begin{enumerate}[label=(\roman*)]
		\item $\Gamma$ is not virtually cyclic,
		\item $\Gamma$ does not fix any point of $\partial\Tc$,
		\item there exist vertices $u, v$ of $\Tc$ such that the stabilizer $\Gamma_{[u,v]}$ is finite;
		\item the finite radical of $\Gamma$ is trivial.
	\end{enumerate}
	Then the action $\Gamma\curvearrowright\mathcal{T}$ is faithful, of general type, 
	and there exists a bounded subtree $\mathcal{B}$ such that $\Gamma_\Bc$ is trivial. In particular, $\Gamma_\Bc$ is core-free in $\Gamma_u$ for every vertex $u$ in $\mathcal{B}$.
\end{proposition}
\begin{proof}
	Starting with a finite stabilizer $\Gamma_{[u,v]}$ given by (iii), one observes that $\bigcap_{\gamma\in \Gamma} \gamma \Gamma_{[u,v]} \gamma\inv$ is contained in the finite radical, hence trivial by (iv). In particular, $\Gamma\curvearrowright\mathcal{T}$ is faithful.
	
	The action $\Gamma\curvearrowright\mathcal{T}$ cannot be elliptic. Indeed, if it were, then $\Tc$ would be a singleton, $\Gamma$ would be finite by (iii), and this contradicts (i). 
	Furthermore, this action cannot be lineal, because of (i), faithfulness and minimality, nor parabolic, nor quasi-parabolic, because of (ii).
	Hence, $\Gamma\curvearrowright\mathcal{T}$ is of general type.
	
	Finally, there is a finite subset $F\Subset \Gamma$, containing $1$, such that $\bigcap_{f\in F} f \Gamma_{[u,v]} f\inv$ is already trivial, hence $\bigcap_{f\in F} \Gamma_{f\cdot[u,v]}$ is trivial. 
	Now, we are done by considering the smallest subtree $\Bc$ of $\Tc$ containing the geodesics $f\cdot[u,v]$, for $f\in F$.
\end{proof}
Second, we prove that Corollary \ref{core-free implies high transitivity} applies to all groups which are highly transitive thanks to  \cite{fimaHighlyTransitiveActions2015}.
This is a straightforward consequence of the following result.
\begin{proposition}\label{proof applicability Cor F to FMS}
	Let a countable group $\Gamma$ act without inversion on a tree $\mathcal{T}$, and consider a set $R\subset E(\Tc)$ of representatives of the edges of the quotient graph $\Gamma \backslash \Tc$. 
	Assume that $\Gamma_v$ is infinite and $\Gamma_e$ is highly core-free in $\Gamma_v$,
	for every couple $(e,v)$ where $e\in R$ and $v$ is one of its endpoints.
	Then, there exists a subtree $\Tc'$ of $\Tc$ such that: 
	\begin{enumerate}[label=(\arabic*)]
		\item the action $\Gamma \curvearrowright \Tc'$ is faithful, of general type, and minimal;
		\item there exist a bounded subtree $\Bc$ of $\Tc'$ and $u \in V(\Bc)$ such that $\Gamma_\Bc$ is core-free in $\Gamma_u$.
	\end{enumerate}
\end{proposition}
\begin{proof}
	First, assume that $\gamma\in \Gamma$ fixes $\Tc$ pointwise. Then, one has $\gamma\in \Gamma_v$ for some endpoint $v$ of an edge $e\in R$. As $\Gamma_e$ is highly core-free in $\Gamma_v$, the $\Gamma_v$-action on the orbit $\Gamma_v \cdot e$, which is conjugate to $\Gamma_v \curvearrowright \Gamma_v/\Gamma_e$, is highly faithful. Hence, we get $\gamma=1$, so that $\Gamma \curvearrowright \Tc$ is faithful.
	
	Second, let us consider any edge $e\in R$, and set $v = s(e)$, $w=r(e)$. 
	By high core-freeness, the indexes
	$[\Gamma_v : \Gamma_e]$ and 
	$[\Gamma_w : \Gamma_e]$ are both infinite. Thus, there exist elements $g_1,g_2 \in \Gamma_w$ and $h_1,h_2 \in \Gamma_v$ 
	such that the edges $e, g_1\inv e, g_2\inv e, h_1^{}e, h_2^{}e$ are pairwise distinct.
	Notice that $(g_1\inv e , \bar e , h_1^{}e)$ and $(g_2\inv e , \bar e , h_2^{}e)$ are oriented paths. 
	For $j=1,2$, the element $h_jg_j$ is hyperbolic and its axis contains $(g_j\inv e , \bar e , h_j^{}e)$. 
	We have got transverse hyperbolic elements $h_1g_1, h_2g_2$, which proves that $\Gamma \curvearrowright \Tc$ is of general type.
	
	In fact, $\Gamma_v$ is infinite and $\Gamma_e$ is highly core-free in $\Gamma_v$,
	for every couple $(e,v)$ where $e$ is any edge of $\Tc$ and $v$ is one of its endpoints. 
	Notice this property passes to the smallest subtree $\Tc'$ of $\Tc$ 
	containing the axes of all hyperbolic elements in $\Gamma$. This 
	subtree is $\Gamma$-invariant, and
	the action $\Gamma \curvearrowright \Tc'$ is still faithful and of general type. Of course, it is also minimal, and Assertion (2) is trivially satisfied when $\Bc$ is  any segment (that is, any subtree with exactly two vertices).
\end{proof}

Finally, we state two natural consequences of Corollary \ref{core-free 
	implies high transitivity} when considering the natural actions of HNN 
extensions (resp. amalgams) on their Bass-Serre tree.

\begin{corollary}\label{CorGroupsHTHNN}
	Consider a non-ascending HNN extension $\Gamma = 
	\HNN(H,\Sigma,\vartheta)$. If one of the subgroups $\Sigma, 
	\vartheta(\Sigma)$ is core-free in $H$, then $\Gamma$ admits a highly 
	transitive and highly faithful action; in particular, $\Gamma$ is 
	highly transitive.
\end{corollary}
\begin{proof}
	Apply Corollary \ref{core-free implies high transitivity} to the tree 
	induced by the edge 
	$\Sigma$ and the vertex $H$, or to the tree induced by the edge 
	$\vartheta(\Sigma)$ and the vertex
	$H$.
\end{proof}
\begin{corollary}\label{CorGroupsHTAmalgam}
	Consider a non-degenerate amalgam $\Gamma = \Gamma_1 *_\Sigma 
	\Gamma_2$. If $\Sigma$ is core-free in one factor $\Gamma_j$, then then 
	$\Gamma$ admits a highly transitive and highly faithful action; in 
	particular, $\Gamma$ is highly transitive.
\end{corollary}
\begin{proof}
	Apply Corollary \ref{core-free implies high transitivity} to the tree 
	induced by the edge 
	$\Sigma$ and the vertex $\Gamma_j$.
\end{proof}

\section{Examples and applications}\label{Examples and applications}

As mentioned in the Introduction, it is worth giving examples of groups which are highly transitive thanks to Theorem~\ref{ThmMain}, or to its consequences, but for which previous results from
\cite{minasyanAcylindricalHyperbolicityGroups2015, hullTransitivitydegreescountable2016, fimaHighlyTransitiveActions2015, gelander_maximal_2020} do not apply. 
In particular, we will prove that some groups are neither acylindrically 
hyperbolic nor linear. To that end, we will use the following well-known 
results.

\begin{proposition}\label{s-normal subgroups}
	\cite[Corollary 1.5]{osin_acylindrically_2016}
	If a group $\Gamma$ is acylindrically hyperbolic,
	then so is any s-normal subgroup of $\Gamma$.
\end{proposition}
Let us recall that a subgroup $\Lambda\leq\Gamma$ is called 
\textbf{s-normal} if for every $\gamma\in\Gamma$, the subgroup $\gamma 
\Lambda\gamma\inv\cap \Lambda$ is infinite.
Every infinite normal subgroup is clearly s-normal.

\begin{proposition}\label{free subgroups in acylindrically hyperbolic groups}
	Every acylindrically hyperbolic group contains a non-abelian free subgroup.
	In particular, every acylindrically hyperbolic group is non-amenable.
\end{proposition}
The latter proposition can either be proved by a standard ping-pong argument, or deduced from Theorem 6.8 and Theorem 8.1 in \cite{dahmaniHyperbolicallyEmbeddedSubgroupsAndRotatingFamilies2017}, 
which imply that every acylindrically hyperbolic group is SQ-universal
(see also the discussion around Conjecture 9.6 in the same book). 

Let us recall that a group is called \textbf{linear over a field $k$} if it is isomorphic to a subgroup of ${\rm GL}(V)$, where $V$ is a finite dimensional $k$-vector space. Note that if $k'$ is an extension of $k$ then any group linear over $k$ is linear over $k'$. Hence, if a group is linear over $k$ then it is also linear over the algebraic closure of $k$.

A group $\Gamma$ is called \textbf{linear} if there exists a field $k$ such that $\Gamma$ is linear over $k$. It follows from the preceding discussion that $\Gamma$ is linear if and only if there exists an algebraically closed field $k$ such that $\Gamma$ is linear over $k$.

\subsection{Examples around Baumslag-Solitar groups}\label{exAroundBaumslag-Solitar}

\subsubsection{Baumslag-Solitar groups themselves}\label{sec: ex BS}

Let us recall the definition: for any $m,n\in\Z^*$, the Baumslag-Solitar group with parameters $m,n$ is  
\[
{\rm BS}(m,n):=\langle a,b\,\vert\,ab^ma^{-1}=b^n\rangle .
\]
Hull and Osin asked what the transitivity degree of Baumslag-Solitar 
groups is \cite[Question 6.3]{hullTransitivitydegreescountable2016}, and 
noted that it was actually  unknown whether ${\rm{BS(2,3)}}$ is highly 
transitive or not. We completely answer this question in 
Proposition \ref{HTforBaumslag-Solitar} and Corollary \ref{tdBS} below.

Notice that ${\rm BS}(m,n)$ is isomorphic to $\HNN(\Z,n\Z,\vartheta)$,
where $\vartheta(nq)=mq$ for all $q\in\Z$, 
and the isomorphism from $\HNN(\Z,n\Z,\vartheta)$ to ${\rm BS}(m,n)$ is given by $t\mapsto a$ and $q\mapsto b^q$ for $q\in \Z$.
We will freely identify ${\rm BS}(m,n)$ to this HNN extension below
without recalling it explicitly.
Hence, ${\rm BS}(m,n)$ has a natural action on the Bass-Serre tree of this 
HNN extension, which we denote by $\Tc_{m,n}$. 
\begin{remark}\label{well-knownBaumslag-Solitar}
	The following facts are well-known:
	\begin{itemize}
		\item ${\rm BS}(m,n)$ is solvable if and only if $|m| = 1$ or $|n| = 1$;
		
		\item ${\rm BS}(m,n)$ is icc if and only if $\vert m\vert\neq\vert n\vert$;
		
		\item ${\rm BS}(m,n)$ is residually finite if and only if $|m| = 1$, $|n| = 1$, or $\vert m\vert = \vert n\vert$; see \cite{meskinNonresiduallyFiniteOne-relatorGroups1972};
		
		\item ${\rm BS}(m,n)$ is  non-linear whenever $|m| \neq 1$, $|n| \neq 1$, and $\vert m\vert\neq\vert n\vert$; this is a consequence of the former fact and Malcev's theorem \cite{malcevIsomorphicMatrixRepresentations1940}.
	\end{itemize}
\end{remark}
\begin{remark}\label{rmk:BS not AH}
	Hull and Osin observed that ${\rm BS}(m,n)$ is never acylindrically hyperbolic (for $m,n\in \Z^*$); this is \cite[Example 7.4]{osin_acylindrically_2016}. Let us recall the argument: since the cyclic subgroup $\langle b \rangle$ is s-normal, the group ${\rm BS}(m,n)$ is not acylindrically hyperbolic by Proposition \ref{s-normal subgroups}.
\end{remark}

Let us note the following result for later use. 

\begin{lemma}\label{s-normal subgroups in BS groups}
	For all $m,n\in \Z^*$ and $r\geq 1$, 
	the subgroup $\langle b^r \rangle$ is s-normal in ${\rm BS}(m,n)$.
\end{lemma}
\begin{proof}
	Let $\gamma$ be any element of ${\rm BS}(m,n)$, with normal form $b^{s_0}a^{\varepsilon_1} b^{s_1} \cdots a^{\varepsilon_k} b^{s_k}$.
	It is easy to check that $\gamma b^{r m^k n^k} \gamma\inv$ is still a non-trivial power of $b$, say $b^s$.
	Then, $\langle b^r \rangle \cap \gamma \langle b^r \rangle \gamma\inv$ contains $\langle b^{r m^k n^k s} \rangle$, hence is infinite.
\end{proof}

Let us now turn to a crucial lemma before stating our result. It is due to 
de la Harpe and Préaux \cite[Lem. 
21]{delaharpeSimpleGroupsAmalgamated2011}, but we include a proof for 
the reader's convenience.

\begin{lemma}[de la Harpe-Préaux]\label{tdBSLemma} The action ${\rm 
		BS}(m,n)\curvearrowright\partial\Tc_{m,n}$ is topologically free
	if and only if $\vert n\vert\neq\vert m\vert$.
\end{lemma}

\begin{proof}
	($\Longleftarrow$) Since $\vert n\vert \neq \vert m\vert$, either $n\nmid 
	m$ or $m\nmid n$.
	Assume that $n\nmid m$ and let $d=\gcd( n, m)$ and $n=dn_0$, $m=dm_0$. Since $n\nmid m$ we have $\vert n_0\vert\geq 2$.
	In particular, we have $\Sigma \neq H$, hence there are several positive edges in the star at any vertex in $\Tc_{m,n}$.
	Consequently every half-tree in $\Tc_{m,n}$ contains a half-tree corresponding to a positive edge.
	Now, let $\gamma\in{\rm BS}(m,n)$, suppose $\gamma$ fixes pointwise a 
	half-tree. 
	Since the action of ${\rm BS}(m,n)$ is transitive on the positive edges of $\mathcal{T}_{m,n}$,
	we may assume that the fixed half-tree $\mathcal{H}$ contains the one given by the edge $\Sigma$. 
	For all $k\geq 1$, one has $\pat_1(t^k)\subset\mathcal{H}$, hence $\gamma$ fixes $\pat_1(t^k)$ pointwise. 
	It follows that $\gamma\in\Sigma \cap t^k\Sigma t^{-k}=n_0^{k+1}d\Z$, for all $k\geq 1$. 
	Hence, $\gamma=1$, since $\vert n_0\vert\geq 2$.  
	In the case $m\nmid n$ the proof is similar 
	(we could also deduce this case from the isomorphism ${\rm BS}(m,n)\simeq{\rm BS}(n,m)$).
	
	($\Longrightarrow$) Suppose that $\vert n\vert =\vert m\vert$, so that $\Sigma=\vartheta(\Sigma)$ and $\vartheta=\pm\id$. In this case, $\Sigma$ is a non-trivial normal subgroup of ${\rm BS}(m,n)$, and $\Tc_{m,n}$ itself is fixed pointwise by any element of $\Sigma$. 
\end{proof}
\begin{remark}\label{core-free subgroups in BS groups}
	The lemma immediately implies that $\langle b\rangle$ is a core-free subgroup of ${\rm BS}(m,n)$
	whenever $\vert n\vert \neq \vert m\vert$.
	Indeed, the action on ${\rm BS}(m,n)\curvearrowright\partial\Tc_{m,n}$ being topologically free,
	the action ${\rm BS}(m,n)\curvearrowright\Tc_{m,n}$ is faithful.
	Hence $\langle b\rangle$ is core-free in ${\rm BS}(m,n)$, since the conjugates of $\langle b \rangle$ in ${\rm BS}(m,n)$ are exactly the vertex stabilizers. See also Lemma \ref{highly core-free subgroups in BS groups}.
\end{remark}

Our first new examples of highly transitive groups are given by the following result.
\begin{proposition}\label{HTforBaumslag-Solitar}
	Let $m,n\in\Z^*$. The following are equivalent:
	\begin{enumerate}[label=(\roman*)]
		\item $|m| \neq 1$, $|n| \neq 1$, and $\vert m\vert\neq\vert n\vert$;
		\item ${\rm BS}(m,n)$ admits a highly transitive and highly faithful action;
		\item ${\rm BS}(m,n)$ is highly transitive;
		\item ${\rm BS}(m,n)$ is non-solvable and icc.
	\end{enumerate}
\end{proposition}
\begin{proof}
	The implication (i)$\implies$(ii) is a direct consequence of Theorem \ref{ThmMainHNN} and Lemma \ref{tdBSLemma}. 
	Then, (ii)$\implies$(iii) is trivial, (iii)$\implies$(iv) results from classical obstructions to high transitivity recalled in the Introduction, and (iv)$\implies$(i) results from Remark~\ref{well-knownBaumslag-Solitar}.
\end{proof}

\begin{remark}
	As reminded in Remark \ref{well-knownBaumslag-Solitar} and Remark 
	\ref{rmk:BS not AH}, the highly transitive groups arising in 
	Proposition 
	\ref{HTforBaumslag-Solitar} are non-acylindrically hyperbolic and 
	non-linear. Moreover, edge-stabilizers are not highly core-free in 
	their endpoints stabilizers (they have finite index). Hence, 
	results from
	\cite{minasyanAcylindricalHyperbolicityGroups2015, hullTransitivitydegreescountable2016, fimaHighlyTransitiveActions2015, gelander_maximal_2020} do not apply. 
\end{remark}
\begin{remark}
	The following lemma proves that Corollary \ref{core-free implies high transitivity} cannot apply to the action ${\rm BS}(m,n) \curvearrowright \Tc_{m,n}$.
	Consequently, Theorem \ref{ThmMain} is stronger than Corollary \ref{core-free implies high transitivity}. Note that Baumslag-Solitar groups are our only examples which testify to this fact.
\end{remark}
\begin{lemma}
	Set $\Gamma:={\rm BS}(m,n)$. If $\Bc$ is any bounded subtree of $\Tc_{m,n}$ and $u$ is any vertex of $\Bc$, 
	then the pointwise stabilizer $\Gamma_\Bc$ is not core-free in $\Gamma_u$.    
\end{lemma}
\begin{proof}
	There exists a positive integer $r$ such that $\Bc$ is contained in 
	the ball $\Bc(r)$ of radius $r$ at $\langle b \rangle$.
	Then, every stabilizer $\Gamma_{v}$, where $v\in V(\Bc(r))$,
	is a conjugate subgroup $\gamma \langle b \rangle \gamma\inv$,
	where the normal form of $\gamma$ contains at most $r$ occurrences of 
	$a^{\pm 1}$.
	Consequently $\gamma\inv b^{m^r n^r} \gamma$ is still a power of $b$, so that $b^{m^r n^r}$ lies in $\gamma \langle b \rangle \gamma\inv = \Gamma_v$. This proves that $b^{m^r n^r}$ lies in the pointwise stabilizer $\Gamma_{\Bc(r)}$.
	Now, for every $\gamma\in \Gamma_u$, one has $\gamma \Gamma_\Bc \gamma\inv = \Gamma_{\gamma\Bc} \supseteq \Gamma_{\Bc(r)}$, 
	since $\gamma\Bc \subseteq \Bc(r)$. 
	Thus all conjugates $\gamma \Gamma_\Bc \gamma\inv$ where $\gamma\in \Gamma_u$ contain $b^{m^r n^r}$, so $\Gamma_\Bc$ is not core-free in $\Gamma_u$.
\end{proof}

Let us now complete the answer to Hull and Osin's question. As they noticed 
in \cite[Lemma~4.2 and Corollary~4.6]{hullTransitivitydegreescountable2016},
infinite non-icc groups and infinite residually finite solvable groups have 
transitivity degree $1$. Hence we can compute the transitivity degree of 
all 
Baumslag-Solitar groups.
\begin{corollary}\label{tdBS}
	Let $m,n\in\Z^*$. The following hold.
	\begin{enumerate}[label=(\arabic*)]
		\item If $\vert n\vert =1$ or $\vert m\vert =1$ or $\vert n \vert =\vert m\vert$, then ${\rm td}({\rm BS}(m,n))=1$.
		\item In the other cases, ${\rm BS}(m,n)$ is highly transitive, so ${\rm td}({\rm BS}(m,n))=+\infty$.
	\end{enumerate}
\end{corollary}

\begin{proof}
	(1) If $\vert n\vert =1$ or $\vert m\vert =1$, the group ${\rm BS}(m,n)$ is infinite, residually finite, and solvable, hence ${\rm td}({\rm BS}(m,n))=1$.
	If $\vert n \vert =\vert m\vert$, the group ${\rm BS}(m,n)$ is infinite and non-icc, hence ${\rm td}({\rm BS}(m,n))=1$.
	
	(2) This follows from Proposition \ref{HTforBaumslag-Solitar} directly.
\end{proof}

\subsubsection{Amalgams with Baumslag-Solitar groups}\label{ExampleHTnonAcylHyp}

Let us now turn to examples of highly transitive groups given by amalgams. Let us begin by some more preliminaries.

\begin{lemma}\label{highly core-free subgroups in BS groups}
	Let $m,n\in \Z^*$. If $\vert n\vert\neq\vert m\vert$, the subgroup $\langle b \rangle$ is highly core-free in ${\rm BS}(m,n)$.
\end{lemma}
Notice this is essentially the same example as the one given in
\cite[Corollary 5.12]{hullTransitivitydegreescountable2016}.
As the action on ${\rm BS}(m,n)\curvearrowright\partial\Tc_{m,n}$ is topologically free, it is a particular case of a general phenomenon described in the following lemma.
\begin{lemma}
	Let $\Gamma\curvearrowright\mathcal{T}$ be a minimal action of a countable group $\Gamma$ on a tree $\mathcal{T}$. 
	If the action on the boundary $\Gamma\curvearrowright \partial\mathcal{T}$ is topologically free, then for every vertex $v$ in $\Tc$, the stabilizer $\Gamma_v$ is highly core-free in $\Gamma$.
\end{lemma}
\begin{proof}
	We have to prove that the action $\Gamma\curvearrowright \Gamma/\Gamma_v$ is highly faithful; 
	by Corollary \ref{EquivStrongAndHighFaithfulness}
	it is sufficient to prove it is strongly faithful.
	Notice the orbit $\Gamma v$ in $\Tc$ meets every half-tree in $\Tc$ 
	by minimality of the action $\Gamma\curvearrowright\mathcal{T}$, 
	and that $\Gamma\curvearrowright \Gamma v$ is conjugate to $\Gamma\curvearrowright \Gamma/\Gamma_v$.
	
	Take any non-trivial elements $\gamma_1, \ldots, \gamma_k \in \Gamma$.
	Let us start with any half-tree $\Hc_0$. 
	Then, since the fixed points of $\gamma_1$ form a subtree, and since $\Gamma\curvearrowright \partial\mathcal{T}$ is topologically free, there exists a half-tree $\Hc_1 \subseteq \Hc_0$, all of whose vertices are moved by $\gamma_1$.
	Then applying the same argument to $\gamma_2$, we get a half-tree $\Hc_2 \subseteq \Hc_1$, 
	all of whose vertices are moved by $\gamma_1$ and $\gamma_2$. 
	And so on, and so forth, we finish with a half-tree $\Hc_k$, 
	all of whose vertices are moved by all elements $\gamma_1, \ldots, \gamma_k$. Finally $\Hc_k$ contains a point of $\Gamma v$, which is moved by all elements $\gamma_1, \ldots, \gamma_k$. This proves that $\Gamma\curvearrowright \Gamma v$ is strongly faithful.
\end{proof}

We also need a general fact about s-normality in amalgams.
\begin{lemma}\label{s-normality in amalgams}
	Let us consider an amalgam $\Gamma = \Gamma_1 *_\Sigma \Gamma_2$.
	If all infinite subgroups $\Sigma' < \Sigma$ are s-normal in both $\Gamma_1$ and $\Gamma_2$, then they are also all s-normal in $\Gamma$.
\end{lemma}
\begin{proof}
	Let $\Sigma_0$ be any infinite subgroup of $\Sigma$, and let $\gamma$ be any element of $\Gamma$, that we write as a product
	$\gamma = \gamma_1 \cdots \gamma_n$
	of elements of $\Gamma_1$ or $\Gamma_2$. Set
	$\Sigma_k = \Sigma_{k-1} \cap \gamma_k\inv \Sigma_{k-1}\gamma_k$ for $k=1,\ldots,n$. 
	Let us prove by induction that $\Sigma_k$ is infinite,
	and contained in $\Sigma_0 \cap (\gamma_1 \cdots \gamma_k)\inv \Sigma_0 (\gamma_1 \cdots \gamma_k)$ for $k=0,\ldots,n$. 
	
	For $k=0$, the group $\Sigma_0$ has been supposed infinite,
	and it coincides with the intersection $\Sigma_0 \cap (\gamma_1 \cdots \gamma_k)\inv \Sigma_0 (\gamma_1 \cdots \gamma_k)$
	in this case.
	Then, for $k\geq 1$, the subgroup $\Sigma_{k-1}$ is infinite by induction hypothesis, 
	therefore $\Sigma_{k-1}$ is s-normal in $\Gamma_1$ and $\Gamma_2$. 
	Consequently, the subgroup $\Sigma_k = \Sigma_{k-1} \cap \gamma_k\inv \Sigma_{k-1}\gamma_k$ is infinite.
	Moreover, one has
	\begin{align*}
	\Sigma_k &= \Sigma_{k-1} \cap \gamma_k\inv \Sigma_{k-1}\gamma_k \\
	&\subseteq
	\Sigma_0 \cap (\gamma_1 \cdots \gamma_{k-1})\inv \Sigma_0 (\gamma_1 \cdots \gamma_{k-1})
	\cap \gamma_k\inv \Sigma_0 \gamma_k
	\cap (\gamma_1 \cdots \gamma_k)\inv \Sigma_0 (\gamma_1 \cdots \gamma_k)
	\end{align*}
	by induction hypothesis, whence $\Sigma_k \subseteq \Sigma_0
	\cap (\gamma_1 \cdots \gamma_k)\inv \Sigma_0 (\gamma_1 \cdots \gamma_k)$.
	
	Finally, applying the result with $k=n$, we get that $\Sigma_n$ is infinite and $\Sigma_n \subseteq \Sigma_0
	\cap \gamma\inv \Sigma_0 \gamma$, which proves that $\Sigma_0$ is s-normal in $\Gamma$, as desired.
\end{proof}

Let us finally turn to our examples of highly transitive amalgams.
\begin{proposition}\label{ExProdAmalgBS}
	Let $m,n,k\in\Z^*$, and let $\Lambda$ be a countable group containing a 
	proper infinite cyclic subgroup $\langle c \rangle$.
	The amalgam
	\(
	\Gamma = {\rm BS}(m,n) *_{\langle b^k = c\rangle} \Lambda
	\)
	has the following properties:
	\begin{enumerate}[label=(\arabic*)]
		\item if $\vert n\vert \neq \vert m\vert$, then $\Gamma$ admits an action which is both highly transitive and highly faithful;
		\item if $\vert n\vert \neq \vert m\vert$ and $\langle c \rangle$ is 
		s-normal in $\Lambda$, then $\Gamma$ is not acylindrically hyperbolic;
		\item if $|m| \neq 1$, $|n| \neq 1$, and $\vert m\vert\neq\vert n\vert$, then $\Gamma$ is not a linear group.
	\end{enumerate}
\end{proposition}
\begin{proof} 
	(1) The amalgam $\Gamma$ is non-degenerate.
	By Lemma \ref{highly core-free subgroups in BS groups}, or Remark \ref{core-free subgroups in BS groups}, $\langle b^k \rangle$ is a core-free subgroup of ${\rm BS}(m,n)$. 
	Then Corollary \ref{core-free implies high transitivity}, or in this case Corollary \ref{CorGroupsHTAmalgam}, 
	implies that $\Gamma$ admits an action which is both highly transitive and highly faithful.
	
	(2) By Lemma \ref{s-normal subgroups in BS groups}, every non-trivial subgroup of $\langle b \rangle$ is s-normal in ${\rm BS}(m,n)$. 
	Furthermore, in $\Lambda$, for any $\lambda\in\Lambda$, the 
	intersection $\langle c \rangle \cap \lambda \langle c \rangle 
	\lambda\inv$ 
	is infinite cyclic, say generated by $c_\lambda$, 
	since $\langle c \rangle$ is s-normal in $\Lambda$. 
	Then, for every $l\geq 1$, one has 
	$\langle c^l \rangle \cap \lambda \langle c^l \rangle \lambda\inv
	= \langle c_\lambda^l \rangle$, which is infinite.
	Hence, every non-trivial subgroup of $\langle c \rangle$ is s-normal in 
	$\Lambda$.
	
	Then, Lemma~\ref{s-normality in amalgams} implies that every non-trivial subgroup of $\langle b^k \rangle = \langle c \rangle$ is s-normal in the amalgam $\Gamma$.
	Now, $\langle c \rangle$ is cyclic, so it is not acylindrically hyperbolic by Proposition \ref{free subgroups in acylindrically hyperbolic groups}, 
	so $\Gamma$ is not acylindrically hyperbolic either by Proposition \ref{s-normal subgroups}. 
	
	(3) The group $\Gamma$ contains a copy of ${\rm BS}(m,n)$, which is non-linear (see Remark \ref{well-knownBaumslag-Solitar}). Hence $\Gamma$ cannot be a linear group.
\end{proof}
\begin{remark}
	The previous proposition shows in particular that if $|m| \neq 1$, $|n| 
	\neq 1$, $\vert m\vert\neq\vert n\vert$, and if one chooses $\Lambda$ such 
	that $\langle c \rangle$ is not highly core-free in $\Lambda$ (e.g. $\Lambda=\Z$ 
	and $c \geq 2$), then results from
	\cite{minasyanAcylindricalHyperbolicityGroups2015, hullTransitivitydegreescountable2016, fimaHighlyTransitiveActions2015, gelander_maximal_2020} do not apply to prove that $\Gamma$ is highly transitive. 
\end{remark}
\begin{remark}
	On the other hand, if $|m| \neq 1$, $|n| \neq 1$, $\vert m\vert\neq\vert n\vert$, and if one chooses $\Lambda$ such that $\langle c \rangle$ is highly core-free in $\Lambda$ (e.g. $\Lambda = {\rm BS}(m,n)$ and $c=b^k$), then Corollary B of \cite{fimaHomogeneousActionsUrysohn2018}
	shows that $\Gamma$ admits homogeneous actions on bounded Urysohn spaces. 
\end{remark}

\subsection{Examples around finitely supported permutations}\label{ex around fsupp permutations}

We now turn to examples constructed from the group of finitely supported 
permutations on an infinite countable set. 

\subsubsection{Examples of HNN extensions over $S_f(X)$}\label{sec: ex HNN 
	from fs perm}

We denote by $S_f(X)$ the subgroup of $S(X)$ consisting of finitely supported permutations.

The group $S_f(X)$ is known to be not linear but we could not find any 
elementary proof in the literature and this is why we have chosen to 
include a complete proof below. We thank Julien Bichon for explaining to us 
the following argument.

\begin{lemma}\label{LemmaPermutationNotLinear}
	Let $\Gamma$ be a group. If, for any prime number $q$ and any $N\in\N^*$, $\Gamma$ contains a subgroup $G$ with $G\simeq(\Z/q\Z)^N$ then $\Gamma$ is not linear.
\end{lemma}

\begin{proof}
	Let $k$ be any algebraically closed field and denote by $p$ its 
	characteristic. Let us recall some elementary facts. For $n\in\N^*$, let us 
	denote by $U_n(k)\subset k^*$ the multiplicative 	subgroup of $n$-th roots of 
	unity. Elements of $U_n(k)$ are exactly the roots of the polynomial 
	$P=X^n-1\in k[X]$. Since $P'=nX^{n-1}$ all the roots of $P$ are simple if 
	$p=0$ or if $p$ is a prime number which does not divide $n$. Hence, if 
	$p=0$ or $p$ is prime and does not divide $n$ one has $\vert U_n(k)\vert 
	=n$.
	To deduce the Lemma, it suffices to prove the following claim.
	\vspace{0.2cm}
	
	\noindent\textbf{Claim.}\textit{ Let $V$ be a finite dimensional vector 
		space over an algebraically closed field $k$. If ${\rm GL}(V)$ contains a 
		subgroup $G$ isomorphic to $(\Z/q\Z)^N$, where $N\in\N^*$ and $q$ is any 
		prime number with $q\neq {\rm char}(k)$ then $N\leq {\rm dim}(V)$.}
	\vspace{0.2cm}
	
	Note that any element $g\in G$ satisfies $g^q=1$ hence, the minimal 
	polynomial $\mu_g$ of $g$ divides $X^q-1$. Since $q\neq{\rm char}(k)$, 
	$X^q-1$ has only simple roots so $\mu_g$ has only simple roots and $g$ is 
	diagonalizable with eigenvalues in $U_q(k)$. Moreover, since $G$ is finite 
	abelian and all its elements are diagonalisable, there exists a basis 
	$\mathcal{B}=(e_1,\dots,e_n)$ of $V$ which simultaneously diagonalises 
	every element of $G$. Let us denote by $\lambda(g)\in U_q(k)^n$ the element 
	$\lambda(g)=(\lambda_1(g),\dots,\lambda_n(g))$, where 
	$g(e_k)=\lambda_k(g)e_k$. This defines an injective map $G\rightarrow 
	U_q(k)^n$, $g\mapsto\lambda(g)$. It follows that $\vert G\vert=q^N\leq 
	\vert U_q(k)^n\vert=q^n$,
	hence $N\leq n = \dim(V)$.
\end{proof}

\begin{proposition}\label{PermutationNotLinear}
	The group $S_f(X)$ is not linear.
\end{proposition}

\begin{proof}
	Let $N\in\N^*$ and $q$ be a prime number. 
	Since $X$ is infinite, one can choose $q$-cycles $\sigma_1, \ldots, \sigma_N$ in $S_f(X)$ with pairwise disjoint supports. It is then easy to check that there is an injective morphism of groups defined by
	\[
	(\Z/q\Z)^N \to S_f(X) \quad ; \quad
	(x_1, \ldots, x_n) \mapsto \sigma_1^{x_1} \cdots \sigma_N^{x_N} \, .
	\]
	Hence, the proof follows from Lemma \ref{LemmaPermutationNotLinear}.\end{proof}

For any subset $F\subseteq X$, let us denote by $\Sigma(F)$ the pointwise stabilizer of $F$ in $S_f(X)$,
and remark that whenever $F$ is finite, the subgroup $\Sigma(F)$ is 
infinite, in fact isomorphic to $S_f(X)$ itself. We will abbreviate 
$\Sigma(\{x\})$ as $\Sigma(x)$.
For any $k\geq 1$, let $X^{(k)}$ denote the set of $k$-tuples of pairwise distinct points in $X$. 
\begin{lemma}\label{core-freeness in S_f(X)}
	Let $F$ be a non-empty subset of $X$. The following hold:
	\begin{enumerate}[label=(\arabic*)]
		\item the stabilizer $\Sigma(F)$ is a core-free subgroup of $S_f(X)$, with infinite index;
		\item if $F$ is finite, then $\Sigma(F)$ is not highly core-free in $S_f(X)$;
		\item if $F$ is finite, then $\Sigma(F)$ is s-normal in $S_f(X)$.
	\end{enumerate}
\end{lemma}
\begin{proof}
	(1) Let $x\in F$. Since the action $X \curvearrowleft S_f(X)$ is transitive (even highly transitive) and faithful we have 
	$\bigcap_{g\in S_f(X)}g\inv \Sigma(x) g
	=\bigcap_{g\in S_f(X)}\Sigma(x\cdot g)
	=\bigcap_{y\in X} \Sigma(y)=\{1\}$, 
	hence $\Sigma(x)$ is core-free in $S_f(X)$. A fortiori, $\Sigma(F)$ is 
	core-free in $S_f(X)$.
	Let us denote by $\tau_y\in S_f(X)$ the transposition $\tau_y=(x\,\,\,y)$ for $y\neq x$. 
	The subgroup $\Sigma(x)$ has infinite index since, for all $y,z\in X \setminus \{x\}$,
	one has $\tau_y^{-1}\tau_z\in \Sigma(x) \Leftrightarrow y=z$. 
	A fortiori, $\Sigma(F)$ has infinite index in $S_f(X)$.
	
	(2) Let us write $F = \{x_1,\ldots, x_k\}$, 
	with $\bar x =(x_1,\ldots, x_k) \in X^{(k)}$.
	The action $X^{(k)} \curvearrowleft S_f(X)$ is transitive, 
	since the action $X \curvearrowleft S_f(X)$ is  highly transitive,
	and the stabilizer of $\bar x$ is $\Sigma(F)$.
	Consequently, the action $\Sigma(F) \backslash S_f(X) \curvearrowleft S_f(X)$ is conjugate to $X^{(k)} \curvearrowleft S_f(X)$.
	Now, $X^{(k)} \curvearrowleft S_f(X)$ is not strongly faithful, 
	since taking $k+1$ permutations with pairwise disjoint and finite supports in $X$, every point in $X^{(k)}$ will be fixed by at least one of them.
	Consequently, $\Sigma(F) \backslash S_f(X) \curvearrowleft S_f(X)$ is not highly faithful.
	
	(3) For any $g\in S_f(X)$, we have
	$\Sigma(F) \cap g\inv \Sigma(F) g = \Sigma(F \cup F\cdot g)$,
	and $F \cup F\cdot g$ is still finite,
	hence $\Sigma(F) \cap g\inv \Sigma(F) g$ is infinite.
\end{proof}
\begin{proposition}\label{examples HNN with core-free subgroups}
	Let $Y$ and $Z$ be two distinct infinite proper subsets of $X$, let 
	$\tau: Y\to Z$ be a bijection, 
	and let $\vartheta = \tau_*: S_f (Y) \to S_f(Z)$ be the isomorphism defined by $\vartheta(\sigma) = \tau\inv \sigma \tau$. 
	Then, the HNN extension
	\(
	\Gamma = \HNN(S_f(X),S_f(Y),\vartheta)
	\)
	has the following properties:
	\begin{enumerate}[label=(\arabic*)]
		\item it admits an action which is both highly transitive and highly faithful;
		\item it is not linear;
		\item if $Y$ and $Z$ are both cofinite, then it is not acylindrically hyperbolic;
		\item if $Y$ and $Z$ are both cofinite, then for every bounded subtree $\Bc$ of its Bass-Serre tree, 
		the pointwise stabilizer $\Gamma_\Bc$ is not highly core-free in a vertex stabilizer $\Gamma_u$.
	\end{enumerate}
\end{proposition}
\begin{proof}
	(1) Note that $S_f(Y) = \Sigma(X\setminus Y)$. Thus, $S_f(Y)$ is a core-free subgroup of $S_f(X)$ by Lemma~\ref{core-freeness in S_f(X)}. Hence, Corollary \ref{core-free implies high transitivity},
	or in this case Corollary \ref{CorGroupsHTHNN}, applies. 
	
	(2) Follows from Proposition \ref{PermutationNotLinear}.
	
	(3) As $Y$ and $Z$ are both cofinite and since the intersection of finitely many cofinite subsets is cofinite, the powers $\tau^n$, 
	defined by composition of partial bijections in $X$ (for $n\in \Z$) 
	all have a cofinite domain and a cofinite range, that we will denote by $Y_n$ and $Z_n$ respectively. 
	Let $U$ be any cofinite subset of $X$.
	For any $g\in S_f(X)$, one has $S_f(U)\cap g\inv S_f(U) g = S_f(U\cap U\cdot g)$, 
	and $U\cap U\cdot g$ is still cofinite.
	Moreover, for any $n\in \Z$, one has
	$S_f(U)\cap t^{-n} S_f(U) t^n 
	= S_f(U\cap U\tau^n)$, where $t\in\Gamma$ is the stable letter. Note that $U\tau^n$ is cofinite since the bijection $\tau^n$ realizes a bijection between $U\cap Y_n$ and $U\tau^n$, 
	so that the subset $U\tau^n$ is cofinite in $Z_n$, hence cofinite. 
	Therefore, $U\cap U\tau^n$ is cofinite.
	
	Then, given $\gamma\in \Gamma$, one can write $\gamma = \gamma_1 \cdots \gamma_k$, 
	where each $\gamma_j$ is either a power of the stable letter $t$, or an element of $S_f(X)$, 
	and an easy induction based on previous facts shows that 
	$S_f(Y) \cap \gamma\inv S_f(Y) \gamma$ contains $S_f(V)$ for some cofinite set $V$. This proves that $S_f(Y)$ is an s-normal subgroup in $\Gamma$.
	
	Furthermore, $S_f(Y)$ is an amenable group, hence it is not acylindrically hyperbolic by Proposition \ref{free subgroups in acylindrically hyperbolic groups}. 
	Finally, Proposition \ref{s-normal subgroups} implies that $\Gamma$ is not acylindrically hyperbolic.
	
	(4) Up to conjugating and to enlarging $\Bc$, we may and will assume without loss of generality 
	that the stabilizer $\Gamma_u$ is $S_f(X)$, and that $\Gamma_e = S_f(Y)$ for some edge $e$ in $\Bc$. 
	Since $\Gamma$ acts transitively on the positive edges, 
	there exists $\gamma_1, \ldots, \gamma_k \in \Gamma$ 
	such that
	\[
	\Gamma_\Bc = S_f(Y) \cap \bigcap_{j=1}^k \gamma_j\inv S_f(Y) \gamma_j \, .
	\]
	As in the proof of (2), we see that there exist cofinite sets $V_1, \ldots, V_k$ such that the intersection
	$S_f(Y) \cap  \gamma_j\inv S_f(Y) \gamma_j$
	contains $S_f(V_j)$ for every $j$, hence $\Gamma_\Bc$ contains $S_f(\bigcap_{j=1}^k V_j)$, where $\bigcap_{j=1}^k V_j$ is cofinite. Now, $S_f(\bigcap_{j=1}^k V_j)$ is not highly core-free in $\Gamma_u = S_f(X)$ by Lemma~\ref{core-freeness in S_f(X)}, hence $\Gamma_\Bc$ is not either.
\end{proof}
\begin{remark}
	When both $Y$ and $Z$ are cofinite in $X$, this proposition 
	provides more explicit new examples of highly transitive groups, 
	since 
	items (1), (3) and (4) show that
	the results from
	\cite{minasyanAcylindricalHyperbolicityGroups2015, 
		hullTransitivitydegreescountable2016, 
		fimaHighlyTransitiveActions2015, 
		gelander_maximal_2020} do not apply.
\end{remark}

\begin{remark}\label{rmk: obvious HT action}In the context of Proposition 
	\ref{examples HNN with 
		core-free subgroups}, 
	notice that $\Gamma$ obviously admits a highly transitive action when $\tau$ can be extended to a permutation $\tilde\tau \in S(X)$. 
	Indeed, the $\Gamma$-action defined by $t\mapsto \tilde\tau$ and $\sigma\mapsto \sigma$ for $\sigma\in S_f(X)$ 
	is highly transitive since its restriction to $S_f(X)$ already is (in the terminology of Section \ref{SectFreeGlobalizationHNN}, this action 
	corresponds to the global pre-action $(X,\tilde\tau)$).
	Nevertheless, the $\Gamma$-action we obtain factors through the semi-direct product $S_f(X) \rtimes \langle \tilde \tau \rangle$, which is amenable while $\Gamma$ is not; hence the $\Gamma$-action is not faithful.
	Furthermore, such an extension to a permutation $\tilde \tau$ is not possible when $X-Y$ and $X-Z$ have different cardinalities.
\end{remark}
\subsubsection{Examples of HNN extensions over $S_f(\Z)\rtimes 
	\Z$}\label{sec: ex fg HNN from fs perm}
Let us now move to a modification of former examples to get groups which 
are moreover finitely generated. 
For these examples, we consider 
the permutation $s\in S(\Z)$ given by $k\cdot s = k+1$.
It is straightforward to check that the subgroup $\langle S_f(\Z), s 
\rangle < S(\Z)$ is finitely generated, and isomorphic to a semi-direct 
product of the form $S_f(\Z) \rtimes \Z$.
As before, for any subset $F\subseteq \Z$, let us denote by $\Sigma(F)$ the pointwise stabilizer of $F$ in $S_f(\Z)$.
For some purposes, we will need the action 
\[
\Z \times \Z \curvearrowleft \langle S_f(\Z), s \rangle \, ,
\qquad
(k,l)\cdot s^n g := (ks^n g, ls^n) = ((k+n)g,l+n)
\]
for $g\in S_f(\Z)$ and $n\in \Z$. Notice that this action is faithful, as is the action $\Z\curvearrowleft \langle S_f(\Z), s \rangle$. Moreover, given a subset $F\subseteq\Z$, we observe that the pointwise stabilizer in $\langle S_f(\Z), s \rangle$ of the subset $F\times \{0\}\subset \Z\times \Z$ is the subgroup $\Sigma(F) < S_f(\Z)$.
\begin{lemma}\label{core-freeness in S_f(ZxZ)}
	Let $F$ be a non-empty subset of $\Z$. The following hold:
	\begin{enumerate}[label=(\arabic*)]
		\item the stabilizer $\Sigma(F)$ is a core-free subgroup of $\langle S_f(\Z), s \rangle$, with infinite index;
		\item if $F$ is finite, then $\Sigma(F)$ is not highly core-free in $\langle S_f(\Z), s \rangle$;
	\end{enumerate}
\end{lemma}
\begin{proof}
	(1) The group $\Sigma(F)$ is already core-free and has infinite index in $S_f(\Z)$ by Lemma~\ref{core-freeness in S_f(X)}.
	
	(2) Let us write $F = \{x_1,\ldots, x_k\}$, 
	with $x_1,\ldots,x_k$ pairwise distinct,
	and set
	\[
	\bar x = \big( (x_1,0),\ldots, (x_k,0) \big) \in (\Z\times \Z)^{(k)} \, .
	\]
	Let us denote by $\Omega$ the orbit of $\bar x$ under $\langle S_f(\Z), s \rangle$.
	As the action $\Z \curvearrowleft S_f(\Z)$ is  highly transitive, $\Omega$ is the union $\bigcup_{n\in\Z}(\Z \times \{n\})^{(k)}$.
	Furthermore, the stabilizer of $\bar x$ is the pointwise stabilizer of $F\times \{0\}$, that is, $\Sigma(F)$.
	Consequently, the action $\Sigma(F) \backslash \langle S_f(\Z), s \rangle \curvearrowleft \langle S_f(\Z), s \rangle$ is conjugate to $\Omega \curvearrowleft \langle S_f(\Z), s \rangle$.
	Now, $\Omega \curvearrowleft \langle S_f(\Z), s \rangle$ is not strongly faithful, 
	since taking $k+1$ elements of $S_f(\Z)$ with pairwise disjoint supports,
	every point in $\Omega = \bigcup_{n\in\Z}(\Z \times \{n\})^{(k)}$ will be fixed by at least one of them.
	Consequently, $\Sigma(F) \backslash \langle S_f(\Z), s \rangle \curvearrowleft \langle S_f(\Z), s \rangle$ is not highly faithful.\end{proof}

Using Lemma \ref{core-freeness in S_f(ZxZ)} we can prove the following Proposition exactly as we proved Proposition \ref{examples HNN with core-free subgroups}.

\begin{proposition}\label{examples HNN with core-free subgroups II}
	Let $Y$ and $Z$ be two distinct infinite proper subsets of $\Z$, let 
	$\tau: Y\to Z$ be a bijection, 
	and let $\vartheta = \tau_*: S_f (Y) \to S_f(Z)$ be the isomorphism 
	defined by $\vartheta(\sigma) = \tau\inv \sigma \tau$. 
	Then, the HNN extension
	\(
	\Gamma = \HNN(\langle S_f(\Z),s \rangle,S_f(Y),\vartheta)
	\)
	has the following properties:
	\begin{enumerate}[label=(\arabic*)]
		\item it admits an action which is both highly transitive and highly faithful;
		\item it is finitely generated and not linear;
		\item if $Y$ and $Z$ are both cofinite in $\Z$, then it is not 
		acylindrically hyperbolic;
		\item if $Y$ and $Z$ are both cofinite in $\Z$, then for every bounded 
		subtree $\Bc$ of its Bass-Serre tree, 
		the pointwise stabilizer $\Gamma_\Bc$ is not highly core-free in a vertex stabilizer $\Gamma_u$.
	\end{enumerate}
\end{proposition}

\begin{remark}
	Again, when both $Y$ and $Z$ are cofinite in $\Z$, this proposition 
	provides explicit new examples of groups which 
	are highly transitive. 
\end{remark}
Note that when the complements of $Y$ and $Z$ have the same cardinality, 
these groups admit a natural highly transitive action, but it fails 
to be faithful. Indeed, the subgroup $\langle S_f(\Z), t \rangle < \Gamma$ is isomorphic to the HNN extension $\HNN(S_f(\Z), S_f(Y),\vartheta)$, and the action of this subgroup is not faithful by Remark \ref{rmk: obvious HT action}.
%
%
\subsubsection{Examples of amalgams}\label{sec: amalgams of fs permutations}
We now switch to the context of amalgams. We will need a refinement of 
Lemma 
\ref{s-normality in amalgams}.

\begin{lemma}\label{refined s-normality in amalgams}
	Let us consider an amalgam $\Gamma = \Gamma_1 *_\Sigma \Gamma_2$. 
	Assume there is a collection $\Cc$ of infinite subgroups of $\Sigma$ such that, for every $\Sigma' \in \Cc$ and every $\gamma \in \Gamma_1 \cup \Gamma_2$, 
	the intersection $\Sigma' \cap \gamma\inv \Sigma' \gamma$ contains an element of $\Cc$.
	Then, for every $\Sigma' \in \Cc$ and every $\gamma \in \Gamma$, 
	the intersection $\Sigma' \cap \gamma\inv \Sigma' \gamma$ contains an element of $\Cc$.
	In particular, all $\Sigma'\in \Cc$ are s-normal subgroups of $\Gamma$.
\end{lemma}
\begin{proof}
	Let $\Sigma_0$ be any element of $\Cc$, and let $\gamma$ be any element of $\Gamma$, that we write as a product
	$\gamma = \gamma_1 \cdots \gamma_n$
	of elements of $\Gamma_1$ or $\Gamma_2$. 
	Let us prove by induction that, for $k=0,\ldots,n$, there exists $\Sigma_k\in \Cc$ which is contained
	in $\Sigma_0 \cap (\gamma_1 \cdots \gamma_k)\inv \Sigma_0 (\gamma_1 \cdots \gamma_k)$. 
	
	For $k=0$, the group $\Sigma_0$ has been chosen in $\Cc$,
	and it coincides with \mbox{$\Sigma_0 \cap (\gamma_1 \cdots \gamma_k)\inv \Sigma_0 (\gamma_1 \cdots \gamma_k)$}
	in this case.
	Then, for $k\geq 1$, the subgroup $\Sigma_{k-1} \cap \gamma_k\inv \Sigma_{k-1} \gamma_k$ contains some $\Sigma_k \in \Cc$. 
	Moreover, one has $\Sigma_{k-1} \subseteq \Sigma_0 \cap (\gamma_1 \cdots \gamma_{k-1})\inv \Sigma_0 (\gamma_1 \cdots \gamma_{k-1})$ by induction hypothesis, hence
	\begin{align*}
	\Sigma_k &\subseteq \Sigma_{k-1} \cap \gamma_k\inv \Sigma_{k-1}\gamma_k \\
	&\subseteq
	\Sigma_0 \cap (\gamma_1 \cdots \gamma_{k-1})\inv \Sigma_0 (\gamma_1 \cdots \gamma_{k-1})
	\cap \gamma_k\inv \Sigma_0 \gamma_k
	\cap (\gamma_1 \cdots \gamma_k)\inv \Sigma_0 (\gamma_1 \cdots \gamma_k)
	\end{align*}
	whence $\Sigma_k \subseteq \Sigma_0
	\cap (\gamma_1 \cdots \gamma_k)\inv \Sigma_0 (\gamma_1 \cdots \gamma_k)$.
	
	Finally, for $k=n$, we get $\Sigma_n \subseteq \Sigma_0
	\cap \gamma\inv \Sigma_0 \gamma$ with $\Sigma_n \in \Cc$,  as desired.
\end{proof}

\begin{proposition}\label{examples of amalgams on fin supp permutations}
	Let $X,Y,Z$ be infinite countable sets such that $Z$ is proper subset of 
	$X \cap Y$.
	Then, the amalgam
	\(
	\Gamma = S_f(X) *_{S_f(Z)} S_f(Y)
	\)
	has the following properties:
	\begin{enumerate}[label=(\arabic*)]
		\item it admits an action which is both highly transitive and highly faithful;
		\item it is not a linear group;
		\item if $Z$ is cofinite in both $X$ and $Y$, then it is not acylindrically hyperbolic;
		\item if $Z$ is cofinite in $X$ (resp. $Y$), then $S_f(Z)$ is not highly core-free in $S_f(X)$ (resp. $S_f(Y)$).
	\end{enumerate}
\end{proposition}
\begin{proof} Lemma \ref{core-freeness in S_f(X)} and Corollary \ref{CorGroupsHTAmalgam} imply (1). Proposition \ref{PermutationNotLinear} implies (2) while (4) follows from  Lemma \ref{core-freeness in S_f(X)}. Let us prove (3). Let $\Cc$ be the collection of subgroups of the form $S_f(U)$ where $U$ is cofinite in $Z$ (hence cofinite in both $X$ and $Y$).
	For every $\Sigma'=S_f(U) \in \Cc$ and every $\gamma \in S_f(X) \cup S_f(Y)$, one can check as in former proofs that
	the intersection $\Sigma' \cap \gamma\inv \Sigma' \gamma$ contains an element of $\Cc$. 
	Hence, Lemma~\ref{refined s-normality in amalgams} implies that $S_f(Z)$ is s-normal in $\Gamma$. Since $S_f(Z)$ is amenable it is not acylindrically hyperbolic hence $\Gamma$ is not acylindrically hyperbolic.
\end{proof}

Again, one can easily modify the previous examples to get groups which are 
moreover finitely generated. The proof of the following Proposition is 
exactly the same as the proof of Proposition \ref{examples of amalgams on 
	fin supp permutations} (by using Lemma \ref{core-freeness in S_f(ZxZ)}). We 
can now provide one last new class of highly transitive examples.

\begin{proposition}\label{examples of fg amalgams on fin supp permutations}
	Let $Z$ be an infinite proper subset of $\Z$ and consider the amalgam\\ $\Gamma:=\langle S_f(\Z),s\rangle*_{S_f(Z)}\langle S_f(\Z),s\rangle$. The following holds.
	\begin{enumerate}[label=(\arabic*)]
		\item $\Gamma$ admits an action which is both highly transitive and highly faithful;
		\item $\Gamma$ is finitely generated and not linear;
		\item If $Z$ is cofinite in $\Z$, then $\Gamma$ is not acylindrically hyperbolic and $S_f(Z)$ is not highly core-free in $\langle S_f(\Z),s\rangle$.
		
	\end{enumerate}
\end{proposition}


\subsection{Faithful actions which are non-topologically free on the boundary}\label{ExampleFaithfulNonTopolFree}

Although our main result provides a complete characterization of high 
transitivity for groups admitting a faithful minimal action of general type 
on a tree, one may wonder if Corollaries \ref{CorGroupsHTHNN} and 
\ref{CorGroupsHTAmalgam} can hold in a wider context, namely, if the 
core-freeness assumption of the edge group in a vertex group can 
be weakened to core-freeness in the whole group. We will see that it is not 
the case. 

Thanks to the quoted result from 
\cite{leboudecTripleTransitivityNonfree2019}, this amounts to finding 
examples of amalgams and HNN extensions whose action on their Bass-Serre 
trees are minimal of general type and faithful, but the action on the 
boundary 
is not topologically free. By Bass-Serre theory, we essentially need to 
find faithful edge transitive actions on trees without inversions which are 
not 
topologically free on the boundary, but which in the amalgam case have two 
vertex 
orbits, while in the HNN case they have only one vertex orbit. 

Our examples belong to a class which was explored 
in depth by Le Boudec \cite{leboudecGroupsActingTrees2016, 
	le_boudec_c-simplicity_2017}, generalizing a construction of 
Bader-Caprace-Gelander-Mozes \cite{baderSimpleGroupsLattices2012} which 
takes its roots in the work of 
Burger-Mozes \cite{burger_groups_2000}. Such examples already 
appeared in Le Boudec and Matte-Bon's 
work on high transitivity, so our only contribution here is to point out 
that some of those naturally decompose as amalgams or HNN 
extensions. We will focus on specific easy examples instead of seeking 
large generality. For more examples, we refer the reader to Ivanov's recent 
work \cite{ivanovTwoFamiliesExamples2020}.

\subsubsection{An example of amalgam}
Let $\Tc_d$ be a $d$-regular tree of finite degree $d\geq 3$. As in 
\cite{burger_groups_2000}, let us fix a coloring on the set of 
edges $c:E(\Tc_d) \to \{1,\ldots,d\}$ such that:
\begin{itemize}
	\item every edge has the same color as its antipode;
	\item for any vertex $v$, the restriction of $c$ to 
	the star $\st(v)$ is a bijection onto $\{1,\ldots,d\}$.
\end{itemize}
For any vertex $v$, any automorphism $g\in \Aut(\Tc_d)$ 
induces a bijection $g_v:\st(v)\to \st(gv)$, which itself induces a 
permutation $\sigma(g,v) \in S_d$, where $S_d=\sym(\{1,\ldots,d\})$. Let $C_d$ be a cyclic subgroup of $S_d$ generated by a $d$-cycle. Consider the group, coming from \cite{leboudecGroupsActingTrees2016},
\[
G=G(C_d) = \{g\in \Aut(\Tc_d) : \, \sigma(g,v) \in C_d \text{ for 
	all but finitely many vertices} \}.
\]
\begin{remark}
	The group $G(C_d)$ is countable.
	Indeed, given an automorphism $g\in G(C_d)$ and any edge $e$ such that $\sigma(g, r(e))$ has to be in $C_d$,
	the permutation $\sigma(g, r(e))$ is determined by the color of $g(e)$.
	It follows that $g$ is completely determined 
	by its restriction to any finite subtree containing at least one edge
	and all stars at vertices $v$ such that $\sigma(g,v)\not\in C_d$.
\end{remark}

In order to forbid inversions, recall there is a natural equivalence 
relation $\mathcal R_{even}$ on $V(\Tc_d)$ which relates any two vertices 
at even distance from each other, and that this equivalence relation is 
preserved by any automorphism of $\Tc_d$. We then let 
\[
\Gamma= G^+ = \{g\in G : \, g \text{ does not exchange the two classes of }
\mathcal R_{even} \} \, .
\]
It is fairly easy to see that the action $\Gamma \curvearrowright \Tc_d$ is transitive on undirected 
edges, hence minimal, of general type, and without inversion.
Let us now fix some edge $e_0$ from  $v_1$ to $v_2$, and 
consider the stabilizers $\Gamma_1$, $\Gamma_2$ and $\Sigma$ of $v_1$, 
$v_2$ and $e_0$ respectively (in $\Gamma$). By Bass-Serre theory, we have the 
following.
\begin{remark}\label{IsomAmalgamTreeAmalgam}
	The morphism $\Gamma_1 \ast_\Sigma \Gamma_2 \to \Gamma$ given by 
	inclusions is an isomorphism, and $\Tc_d$ is the Bass-Serre tree of 
	$\Gamma_1 \ast_\Sigma \Gamma_2$.
\end{remark}

The following result summarizes well-known properties of $\Gamma$ showing that the hypothesis that $\Sigma$ is core-fre in $\Gamma_1$ or $\Gamma_2$ cannot be relaxed in Corollary \ref{CorGroupsHTAmalgam}. We provide a proof for the reader's convenience.
\begin{proposition}\label{prop:properties nonHTamalgam}
	With the above notations:
	\begin{enumerate}[label=(\arabic*)]
		\item $\Sigma$ is core-free in $\Gamma$, and the amalgam $\Gamma_1 
		\ast_\Sigma \Gamma_2$ is non-degenerate;
		\item the $\Gamma$-action on $\partial \Tc_d$ is not topologically 
		free;
		\item $\Gamma$ is not highly transitive;
		\item $\Gamma$ is icc.
	\end{enumerate}
\end{proposition}

\begin{proof}
	(1) The $\Gamma$-action on $\Tc_d$ is faithful (by definition) and of general type, so that 
	$\Sigma$ is core-free in $\Gamma$, and the amalgam 
	$\Gamma_1 \ast_\Sigma \Gamma_2$ is non-degenerate.
	
	(2) Consider the half-tree $\Hc$ associated to $e_0$. It suffices to prove that the pointwise stabilizer of $\Hc$ is a non-trivial group. 
	
	To do so, we follow the proof of \cite[Theorem C]{le_boudec_c-simplicity_2017}. 
	First, we take a non-trivial permutation $\sigma \in S_d$  which fixes the color of $e_0$ 
	(this exists since $d\geq 3$; notice it lives in $S_d \setminus C_d$). 
	Then, we define a non-trivial automorphism $\gamma\in \Gamma$ fixing $\Hc$ pointwise as follows.
	\begin{itemize}
		\item The restriction of $\gamma$ to $\Hc$ is the identity.
		\item Then, we let $\gamma$ act on $\st(v_1)$ so that $\sigma(g,v_1) = \sigma$ 
		(this is possible since $ \sigma$ fixes the color of $e_0$, and will guarantee that $\gamma$ is non-trivial since $\sigma$ is non-trivial).
		\item Then, we extend the action inductively: given any vertex $w$ 
		outside 
		$\Hc$, 
		we set $w'$ to be the unique neigbour of $w$ which is closer to $\Hc$ than $w$, 
		and define the $\gamma$-action on $\st(w)$ in terms of the (previously defined) $\gamma$-action on $\st(w')$. 
		Namely, denoting by $e_w$ the edge from $w$ to $w'$, the 
		$\gamma$-action on $\st(w')$ provides the edge $\gamma e_w$. Then 
		there is a unique element $\sigma_w\in C_d$ sending $c(e_w)$ onto 
		$c(\gamma e_w)$, 		 
		and we let $\gamma$ act on $\st(w)$ so that $\sigma(g,w) = \sigma_w$ 
		(note that $\sigma(g,w) = \sigma(g,w')$ as soon as $\sigma(g,w')$ was 
		already in $C_d$). 
	\end{itemize}
	
	(3) The $\Gamma$-action on $\Tc_d$ is minimal, by edge-transitivity, and of general type. Consequently, \cite[Corollary 1.5]{leboudecTripleTransitivityNonfree2019} applies, and 
	the transitivity degree of $\Gamma$ is at most $2$.
	
	(4) Let $\gamma_0$ be a non-trivial element of $\Gamma$ and $\xi$ be a
	point in $\partial \Tc_{d}$ which is not fixed by $\gamma_0$.
	Given edges $e,e'$ with sources in the same class of vertices, 
	such that $e$ is on the geodesic $[\xi,s(e')]$,
	there exists $\gamma \in \Gamma$ such that $\gamma e = e'$.
	This $\gamma$ is a hyperbolic element whose axis contains $e$ and $e'$.
	Moreover, if we choose $e$ close enough to $\xi$, then the repelling point $\xi^-$ of $\gamma$ in 
	$\partial\Tc_{d}$ is close enough to $\xi$ so that $\gamma_0 \xi^- \neq \xi^-$. Now, since $\xi^-$ is not fixed by $\gamma_0$, 
	the set of fixed points of $\gamma^n\gamma_0\gamma^{-n}$ moves into 
	smaller and smaller neighborhoods (in $\Tc_{d} \cup \partial \Tc_{d}$) 
	of the attracting point of $\gamma$ as $n\to +\infty$. 
	Thus, the set $\{ \gamma^n\gamma_0\gamma^{-n} : \, n\geq 1 \}$ is infinite.  
\end{proof}

\subsubsection{An example of HNN extension}
In the previous example, notice that even the subgroup
\[
U(\id) = \{g\in \Aut(\Tc_d) : \, \sigma(g,v) =\id \text{ for 
	every vertex } v \}
\]
includes inversions, so that one cannot easily find a subgroup of $G$ without inversion and acting transitively on $V(\Tc_d)$. 
Hence, we slightly modify the construction in order to get an example of HNN extension.

Let $\Tc_{d,d}$ be a $(d,d)$-biregular tree, where $d\geq 2$ 
(by this, we mean an oriented tree in which every star $\st(v)$ contains exactly $d$ positive edges and $d$ negative edges). 
Let us denote $\st(v)^+ = \st(v)\cap E(\Tc_{d,d})^+$ and $\st(v)^- = \st(v)\cap E(\Tc_{d,d})^-$,
and fix a coloring on the set of 
edges $c:E(\Tc_{d,d}) \to \{1,\ldots,2d\}$ such that:
\begin{itemize}
	\item every edge has the same color as its antipode;
	\item for any vertex $v$, the restriction of $c$ to 
	the star $\st(v)$ is a bijection onto $\{1,\ldots,2d\}$;
	\item for any vertex $v$, the image $c(\st(v)^+)$ is either $\{1,\ldots,d\}$ or $\{d+1,\ldots,2d\}$.
\end{itemize}
By $\Aut(\Tc_{d,d})$, we mean the group of automorphisms of $\Tc_{d,d}$ preserving the orientation. 
For any vertex $v$, any automorphism $g\in \Aut(\Tc_{d,d})$ 
induces bijections $g_v^\pm:\st(v)^\pm\to \st(gv)^\pm$, which themselves induces a permutation $\sigma(g,v) \in S_{2d}$, 
where $S_{2d}=\sym(\{1,\ldots,2d\})$, preserving the partition \mbox{$\{1,\ldots,d\} \sqcup \{d+1,\ldots,2d\}$.} 
Let $F_d$ be a the subgroup of $S_{2d}$ generated by the commuting 
elements 
\begin{align*}
\sigma_1&= (1 \quad 2\quad \cdots\quad d)(d+1\quad d+2\quad\cdots \quad 
2d),\\
\sigma_2&=(1\quad d+1)(2\quad d+2)\cdots(d\quad 2d).
\end{align*}
Then, the group, 
\[
\Gamma = G(F_d) = \{g\in \Aut(\Tc_{d,d}) : \, \sigma(g,v) \in F_d \text{ for 
	all but finitely many vertices} \}
\]
is countable (this is not hard to prove, using that $F_d$ acts freely on $\{1,\ldots,2d\}$).

It is fairly easy to see that the action $\Gamma \curvearrowright \Tc_{d,d}$ is transitive on positive 	edges, 
hence minimal, of general type. It is moreover without inversion since it 
preserves the orientation.
Let us fix some vertex $v$, some positive edges $e_1,e_2$ such that $r(e_1) = v = s(e_2)$, 
and some automorphism $\tau \in \Gamma$ such that $\tau(e_1) = e_2$.  
Now, consider the stabilizers $H = \Gamma_v$ and $\Sigma = \Gamma_{e_2}$, and the isomorphism $\vartheta: \Sigma \to \Gamma_{e_1}$ given by $\vartheta(\sigma)(x) = \tau\inv \sigma \tau (x)$. By Bass-Serre theory, we have the 
following.
\begin{remark}\label{IsomAmalgamTreeHNN}
	The morphism $\HNN(H,\Sigma,\vartheta) \to \Gamma$ given by 
	inclusions is an isomorphism, and $\Tc_{d,d}$ is the Bass-Serre tree of 
	$\HNN(H,\Sigma,\vartheta)$.
\end{remark}

The following result summarizes well-known properties of $\Gamma$ showing that the hypothesis that one of the subgroups $\Sigma, \vartheta(\Sigma)$ is core-free in $H$ cannot be relaxed in Corollary \ref{CorGroupsHTHNN}. We omit the proof, which is similar to the one of Proposition \ref{prop:properties nonHTamalgam}. 
\begin{proposition}
	With the above notations:
	\begin{enumerate}[label=(\arabic*)]
		\item $\Sigma$ is core-free in $\Gamma$, and the HNN extension $\HNN(H,\Sigma,\vartheta)$ is non-ascending;
		\item the $\Gamma$-action on $\partial \Tc_{d,d}$ is not topologically 
		free;
		\item $\Gamma$ is not highly transitive;
		\item $\Gamma$ is icc.
	\end{enumerate}
\end{proposition}

\section{Other types of actions and necessity of the minimality 
	assumption}\label{MinimalityIsNecessary}

We now discuss various natural extensions of Theorem 
\ref{ThmMain} by considering  other types of actions (recall that group 
actions on trees are classified in five different types, see Section 
\ref{PrelimTrees}). As we will see, non-general type actions which are 
topologically free on the boundary seem to play 
no role regarding high transitivity. We will also see that the minimality 
hypothesis in Theorem 
\ref{ThmMain} cannot be avoided. 

Let us recall that a group action by homeomorphisms on a topological space
is called \textbf{minimal} when every orbit is dense. 
Note that given a group action on a tree, 
the minimality of the action on the boundary implies the minimality of 
the action on the tree but the converse does not hold: for instance the standard
$\Z$-action on itself is minimal but the action on the boundary is not since it
has two distinct fixed points.

\begin{proposition}\label{prop: rf has elliptic}
	Every residually finite group admits an elliptic faithful 
	action 
	on a tree with non-empty boundary
	such that the action on the boundary is both free and minimal.
\end{proposition}
\begin{proof}
	Let $\Gamma$ be a residually finite group, let $(\Gamma_n)_{n\geq 0}$ be a 
	decreasing 
	chain of  finite index normal subgroups with trivial intersection, 
	where 
	$\Gamma_0=\Gamma$. Then the disjoint union of the coset spaces 
	$\Gamma/\Gamma_n$ has a natural tree structure where we connect each 
	$\gamma\Gamma_{n+1}$ to $\gamma\Gamma_n$, and the boudary is non-empty. Since the action is 
	transitive on 
	each level of the tree, this action is minimal on the boundary. It is 
	free 
	because the subgroups $\Gamma_n$ are normal and intersect trivially. 
	Finally it is elliptic because the vertex $\Gamma$ is fixed 
	(moreover, the only non-trivial invariant subtrees are balls around $\Gamma$).
\end{proof}
Since there are both highly transitive residually finite groups  (such as 
$\mathbb F_2$) and non-highly transitive residually finite groups (such as 
$\mathbb Z$), we see that there is no hope for a classification of the 
transitivity degree of groups admitting an elliptic faithful action on a 
tree 
such that the induced action on the boundary is free and minimal. We can use also this 
construction in 
order to show that the minimality assumption for the action on the tree is needed in Theorem 
\ref{ThmMain}.

\begin{proposition}\label{prop: free group times rf action on tree}
	Let $\Gamma$ be a non-abelian free group. Then for every residually finite group $\Lambda$, the group 
	$\Gamma\times \Lambda$ admits an action of general type on a tree which 
	is 
	topologically free on the boundary.
\end{proposition}
\begin{proof}
	Since $\Gamma$ is free, we have a free $\Gamma$-action on a tree $\Tc_1$ which is of general type because $\Gamma$ is not abelian. 
	Let $\Lambda\act\Tc_2$ be an action provided by 
	the previous proposition, let $o$ be its unique fixed point. 
	Our new tree $\Tc$ 
	is obtained by gluing over each vertex of $\Tc_1$ a copy of $\Tc_2$ at 
	its origin $o$. 
	To be more precise, the vertex set is 
	\( V(\Tc)=V(\Tc_1)\times V(\Tc_2) \), 
	and on the vertex set $\Tc_1\times \{o\}$ we put a copy of the 
	edges of $\Tc_1$, while for each  $v\in V(\Tc_1)$, we put a copy of the 
	edges of $\Tc_2$ on the vertex set $\{v\}\times \Tc_2$.
	
	Then the $\Gamma\times \Lambda$ action on $V(\Tc)$ given by $(\gamma, 
	\lambda) \cdot (x_1, x_2)=(\gamma\cdot x_1,\lambda\cdot x_2)$ is an action by 
	automorphisms on our new tree $\Tc$. Noting that each half-tree 
	contains 
	a copy of a half-tree of $\Tc_2$, it is not hard to check that this action is moreover 
	topologically free. Moreover, it is of general type since the 
	$\Gamma$-action on $\mathcal T_1$ was of general type.
\end{proof}

The previous proposition will allow us to show that in our main result, the hypothesis of minimality of the action on the tree is needed, 
using the well-known fact that non-trivial product groups cannot be highly transitive.
For the convenience of the reader, we provide a proof of the latter fact via the following stronger result.

\begin{proposition}\label{prop: dense in topo simple}
	Let $G$ be a non-abelian topologically simple group. If $\Gamma$ is a 
	dense 
	subgroup of $G$, then the centralizer of every non-trivial element of 
	$\Gamma$ is core-free.
\end{proposition}
\begin{proof}
	Let $\gamma\in\Gamma\setminus\{1_\Gamma\}$. If the centralizer of 
	$\Gamma$ 
	is not core-free, then there exists a normal subgroup $N\leq\Gamma$ 
	such 
	that every element of $N$ commutes with $\gamma$. Since $G$ is 
	topologically simple, $N$ is dense in $G$, so by continuity of group 
	multiplication we conclude that every element of $G$ commutes with 
	$\gamma$. In particular, $G$ has a non-trivial center, which 
	contradicts 
	the topological simplicity of $G$ since $G$ is not abelian.
\end{proof}
\begin{corollary}\label{cor: ht restriction}
	Let $\Gamma$ be a highly transitive group. Then the centralizer of 
	every 
	non-trivial element of $\Gamma$ must be core-free, and $\Gamma$ cannot 
	be 
	decomposed as a non-trivial direct product.
\end{corollary}
\begin{proof}
	If $\Gamma$ is highly transitive, it can be embedded as a dense 
	subgroup of 
	$S(X)$ for some infinite set $X$, and since the latter is topologically 
	simple, the conclusion follows from the previous result. Moreover, if 
	$\Gamma$ could be decomposed as a non-trivial direct product 
	$\Gamma_1\times\Gamma_2$, then if $\gamma\in\Gamma_1\setminus 
	\{1_{\Gamma_1}\}$, the element $(\gamma,1_{\Gamma_2})$ would commute 
	with 
	every element of the non-trivial normal subgroup 
	$\{1_{\Gamma_1}\}\times\Gamma_2$, a contradiction.
\end{proof}
Applying Proposition \ref{prop: free group times rf action on tree} for instance to $\Gamma=\mathbb F_2$ and 
$\Lambda=\Z$, we see that the finitely generated group $\mathbb F_2\times \Z$ admits a (faithful) 
action of general type on a tree $\Tc$ which is topologically free on the 
boundary, although the group $\mathbb F_2\times \Z$ is not highly transitive because it decomposes as a 
direct product. We see moreover that when restricting to the minimal 
component 
of this action, we will loose the 
faithfulness of 
the action (in particular the topological freeness), which is why Theorem 
\ref{ThmMain} cannot be applied.

We now move on to showing that no general classification can be hoped for 
in 
the case of parabolic actions (note that minimal parabolic actions only 
arise 
for non-finitely generated groups).

\begin{lemma}
	Every non-finitely generated group admits a faithful 
	parabolic action on a tree.
\end{lemma}
\begin{proof}
	Since $\Gamma$ is countable, it can be written as a countable 
	increasing 
	union of finitely generated subgroups 
	$\Gamma=\bigcup_{n\in\N}\Gamma_n$, 
	where $\Gamma_0=\{1\}$. We now put a tree structure on the vertex set 
	$\bigsqcup_n \Gamma/\Gamma_n$ by connecting each $\gamma \Gamma_n$ to 
	$\gamma\Gamma_{n+1}$. $\Gamma$ acts on this tree by left translation.  
	Since $\Gamma_0=\{1\}$, this action is faithful. Note that every group 
	element $g\in\Gamma$ is elliptic (with fixed point $\Gamma_n$, where 
	$n\in\N$ is such that $g\in\Gamma_n$). Since $\Gamma$ is not finitely 
	generated, this action has no global fixed point, so we have a 
	parabolic 
	action.
\end{proof}
Note that the $\Gamma$-invariant subtrees of the action constructed above 
are exactly those of the form $\sqcup_{n\geq 
	m}\Gamma/\Gamma_n$ for some $m\geq 0$. So this action has no minimal 
globally invariant 
subtree, as opposed to what happens for finitely generated groups. 

\begin{proposition}
	Let $\Gamma$ be a non-finitely generated group, let $\Lambda$ be a 
	residually finite infinite group. Then $\Gamma\times \Lambda$ admits a 
	parabolic action on a tree such that the action on the boundary is 
	topologically free.
\end{proposition}
\begin{proof}
	As before we use the $\Gamma$-action on a tree $\Tc_1$ provided by the 
	previous lemma. 
	Let $\Lambda\act\Tc_2$ be an elliptic action provided 
	by 
	Proposition \ref{prop: rf has elliptic}, whose unique fixed point is 
	denoted by $o\in V(\Tc_2)$. This time, we glue a copy of $\Tc_2$ on 
	each terminal vertex
	(that is, each vertex of the form $\gamma\Gamma_0$) in the tree $\Tc_1$, thus yielding a 
	$\Gamma\times 
	\Lambda$ action which is easily seen to be parabolic. Moreover, the 
	action on the boundary is topologically free. Note furthermore that the 
	action on 
	the 
	boundary has a unique fixed point, corresponding to the unique element 
	of 
	the boundary of $\Tc_1$.
\end{proof}
As an example, we can take for $\Gamma$ the group of finitely supported 
permutations, and for $\Lambda$ the group $\Z$, and we get a non-highly 
transitive group with a parabolic action which is topologically free on the 
boundary, and which has no minimal component. We do not know if there is a 
highly transitive group with a parabolic action which is topologically free 
on the boundary. \\

Let us now treat the quasi-parabolic case. First, note that the 
Baumslag-Solitar groups ${\rm BS}(1,n)$ for $n\geq 2$ provide examples of 
groups 
admitting an action on a tree which is minimal and quasi-parabolic (since 
it is an ascending HNN extension, cf. Sec. \ref{PrelimHNN}) and 
topologically free on the boundary (by Lemma \ref{tdBSLemma}), but which 
are not highly transitive since they are solvable.

\begin{remark}
	Another example of a non-highly transitive group with a quasi-parabolic 
	minimal action on a tree which is topologically free on the boundary is 
	provided by Thompson's group $F=\la x_0,x_1,x_2,\ldots\vert\, x_k\inv 
	x_nx_k=x_{n+1}\text{ for all }k<n\ra$. $F$ is not highly transitive 
	since, 
	by \cite[Corollary~
	5.3]{leboudecTripleTransitivityNonfree2019}, it has transitivity degree 
	at 
	most $2$. Let $H$ be the 
	subgroup generated by $\{x_i: i\geq 1\}$ (which is isomorphic to $F$), 
	and $\vartheta$ be the endomorphism which takes $x_i$ to $x_{i+1}$ and 
	observe that $F=\HNN(H,H,\vartheta)$. By Section \ref{PrelimHNN} the 
	action 
	of $F$ on the associated Bass-Serre tree is quasi-parabolic and 
	minimal. 
	Moreover, since $x_0^{-k}\vartheta(H)x_0^k=\langle x_n\,\vert\,n\geq 
	k+2\rangle$ for all $k\geq 1$, it is not difficult to check that 
	$\vartheta(H)$ is core-free in $H$ hence, the action is also 
	topologically 
	free on the boundary.
\end{remark}

Coupled with the previous examples, the following
proposition shows that for groups admitting minimal quasi-parabolic 
actions on a tree, the topological freeness of the action does not play a 
role in their high transitivity.

\begin{proposition}
	The finitely generated group $\Gamma=S_f(\Z)\rtimes\Z$ is highly 
	transitive and finitely 
	generated but admits a minimal quasi-parabolic action on a tree which 
	is topologically free on the boundary.
\end{proposition}
\begin{proof}
	We have already observed that $\Gamma$ is highly transitive thanks to 
	its natural action on $\Z$. We will obtain our desired action on a tree by showing that it can be written 
	as an 
	\emph{ascending} HNN extension. Note that $\Gamma$ is a semi-direct product, so it does have a natural HNN extension decomposition, but this decomposition provides a lineal action so we need another one.
	
	Denote by $\tau$ the translation on $\Z$. Let 
	$\vartheta$ be the corresponding inner automorphism of 
	$S(\Z)$, i.e. $\vartheta(\gamma)=\tau\inv\gamma\tau$.
	Consider the subgroups $H=\Sigma=S_f(\N)$ which we view as subgroups 
	of 
	$S_f(\Z)<\Gamma$. Note that $\vartheta(\Sigma)$ is the stabilizer of 
	$0$ in 
	$S_f(\N)$ i.e. $\vartheta(\Sigma)=\Sigma(0)<S_f(\N)$ with the notations 
	of 
	Section \ref{ex around fsupp permutations}. We claim that 
	$\Gamma=\HNN(H,\Sigma,\vartheta)$.
	
	First, since $\vartheta(h) = \tau\inv h\tau$ for all $h\in H$, we have a quotient map
	$\pi:\HNN(H,\Sigma,\vartheta)\to 
	\Gamma$ given by $t\mapsto \tau$ and $h\mapsto h$ for every $h\in H$.
	To show that $\pi$ is injective, we use the fact that the HNN extension is ascending:
	since $H= \Sigma$, for every $g \in \HNN(H,\Sigma,\vartheta)$,
	there exists $k\in \Z$ such that $t^{-k} g t^{k} = h t^n$ with $h\in H$ and $n\in \Z$ (it suffices to take $k$ sufficently large).
	If $g\neq 1$, one must have $h\neq \id$ or $n\neq 0$, and it is clear that $\pi(g) \neq 1$ in both cases.
	
	So we do have $\Gamma=\HNN(H,\Sigma,\vartheta)$, in particular it is an 
	ascending non-degenerate HNN extension. So as explained in Section 
	\ref{PrelimHNN}, its action on its Bass-Serre tree is minimal and 
	quasi-parabolic. Finally, $\vartheta(\Sigma)=\Sigma(0)$ is core-free in 
	$H=S_f(\N)$ by Lemma \ref{core-freeness in S_f(X)}, so 
	by Lemma \ref{core-free implies topologically free} we conclude that 
	the 
	action on the boundary is topologically free as wanted.
\end{proof}
\begin{remark}
	One can also construct a tree $\Tc$ directly, with vertex set  $V = \bigsqcup_{k\in\Z} S_f(\Z)/H_k$, where $H_k := S_f(\Z_{\geq -k})$, and positive edges corresponding to inclusions $\sigma H_k \subseteq \sigma H_{k+1}$.
	The reader can verify that $\Gamma=S_f(\Z)\rtimes\Z$ acts on $\Tc$ 
	via 
	\[
	(\sigma,n) \cdot \sigma'H_k :=
	\sigma \tau^n (\sigma'H_k) \tau^{-n}
	= \sigma (\tau^n \sigma' \tau^{-n})H_{k+n},
	\]
	where $\tau$ is still the translation on $\Z$, and then check ``by hand'' that this action has all the properties announced in the above proposition.
\end{remark}

We finally mention the lineal case. Note that in this case, minimal actions 
are not interesting with respect to high transitivity since no subgroup 
of the automorphism group of the biinfinite line is highly transitive. 

In the elliptic case, a natural weakening of the minimality assumption was 
provided by asking that the 
action on the boundary is minimal. Here however, this is still too strong a 
condition since any lineal action will have the two ends corresponding to 
the axis as an invariant set. We thus replace it by topological 
transitivity (which means the existence of a dense orbit) and observe that in this setup, there seems to be no 
connection 
between high transitivity and lineal actions.

\begin{proposition}
	Let $\Gamma$ be a residually finite group, then the group $\Gamma\times 
	\Z$ admits a lineal  action on a tree which is both topologically free 
	and topologically transitive on the boundary.
\end{proposition}	
\begin{proof}
	Let $\Tc$ be a tree equipped with an elliptic $\Gamma$-action which is 
	minimal and free on the boundary as provided by Proposition \ref{prop: rf has elliptic}, let $o$ be the fixed point. As in the previous 
	constructions, we then glue a copy of $\Tc$ at the vertex $o$ on top of 
	every element of $\Z$, thus obtaining a tree with a natural 
	$\Gamma\times \Z$-action which is both topologically free 
	and topologically transitive on the boundary: every element of the boundary which does not belong to the two element set $\partial \Z$ has a dense orbit and is fixed by no nontrivial group element.
\end{proof}
The previous proposition provides us many non-highly transitive groups with 
a lineal action on a tree which is both topologically free 
and topologically transitive on the boundary.

In the opposite direction, the group $S_f(\Z)\rtimes\Z$ provides us an 
example of a highly 
transitive group satisfying the assumptions of the previous proposition, 
thus showing that lineal actions which are topologically free and 
topologically transitive 
on the boundary do not play a role in high transitivity.

\begin{proposition}
	The highly transitive group $S_f(\Z)\rtimes\Z$ admits a lineal  action 
	on a tree which is both topologically free 
	and topologically transitive on the boundary.
\end{proposition}
\begin{proof}
	Let us construct a tree $\Tc$ as follows.
	We start with $\Z$, seen as a bi-infinite line.
	Then, for every $(\sigma, k) \in S_f(\Z)\rtimes\Z$,
	we consider an infinite ray $\Rc_{(\sigma, k)}$, and link its origin $o_{(\sigma, k)}$ to the vertex $k\in \Z$ by an edge.
	The boundary $\partial \Tc$ consists of points $\xi_\gamma$ for $\gamma \in S_f(\Z)\rtimes\Z$ (the extremities of the rays, which are isolated in $\partial \Tc$) and two accumulation points $\eta_\pm$.
	
	The group $S_f(\Z)\rtimes\Z$ acts on $\Tc$ by $(\sigma, k) \cdot k' = k+k'$ and by $\gamma \cdot \Rc_{\gamma'} = \Rc_{\gamma\gamma'}$ (it permutes the rays).
	The induced action on $\partial \Tc$ fixes $\eta_\pm$ and is transitive free on $\{\xi_\gamma : \gamma \in S_f(\Z)\rtimes\Z\}$.
	Thus, it is both topologically free and topologically transitive.
\end{proof}

\bibliographystyle{amsplain}


\begin{dajauthors}

\begin{authorinfo}[fima]
  Pierre Fima\\
  Université Paris Cité and Sorbonne Université, CNRS, IMJ-PRG, \\
  F-75013 Paris, France.\\
  pierre.fima@imj-prg.fr
\end{authorinfo}

\begin{authorinfo}[lemaitre]
  François Le Maître\\
  Université Paris Cité and Sorbonne Université, CNRS, IMJ-PRG, \\
F-75013 Paris, France.\\
françois.le-maitre@imj-prg.fr
\end{authorinfo}

\begin{authorinfo}[moon]
  Soyoung Moon\\
  Université de Bourgogne, Institut Mathématiques de Bourgogne,  CNRS, BP 47870,\\
   21078 Dijon cedex France.\\
   soyoung.moon@u-bourgogne.fr
\end{authorinfo}

\begin{authorinfo}[stalder]
  Yves Stalder\\
  Université Clermont Auvergne, CNRS, LMBP, \\
  F-63000 Clermont-Ferrand, France.\\
  yves.stalder@uca.fr
\end{authorinfo}
\end{dajauthors}

\end{document}